\documentclass[UTF-8,reqno]{amsart}

\usepackage{amsmath, amsfonts, amsthm, amssymb, stmaryrd, color, cite}

\usepackage{hyperref}
\newtheorem{theorem}{Theorem}[section]
\newtheorem{lemma}{Lemma}[section]
\newtheorem{proposition}{Proposition}[section]
\newtheorem{corollary}{Corollary}[section]

\theoremstyle{definition}
\newtheorem{definition}{Definition}[section]

\theoremstyle{remark}
\newtheorem{remark}{Remark}[section]

\numberwithin{equation}{section}
\allowdisplaybreaks

\def\f{\frac}

\def\hf1{^\f{1}{1-\xi^2}}

\def\be{\begin{equation}}
\def\en{\end{equation}}
\def\bs{\begin{split}}
\def\es{\end{split}}
\def\ba{\begin{align}}
\def\ea{\end{align}}

\title[Well-posedness for the stochastic electrokinetic flow]
{Well-posedness for the stochastic electrokinetic flow}

\author[Z. Qiu]{Zhaoyang Qiu}
\address{School of Applied Mathematics, Nanjing University of Finance and Economics, Nanjing, 210046, China.}
\email{zhqmath@163.com}

\author[H. Wang]{Huaqiao Wang}
\address{College of Mathematics and Statistics, Chongqing University, Chongqing, 401331, China.}
\email{wanghuaqiao@cqu.edu.cn}

\keywords{Stochastic Nernst-Planck-Navier-Stokes system, weak martingale solution, maximal pathwise solution, global pathwise solution, stochastic compactness, blow-up criterion, fixed point theorem}
\subjclass[2010]{35Q35, 76D05, 35R60, 60F10}

\date{\today}

\begin{document}
\begin{abstract}
We investigate the stochastic electrokinetic flow modelled by a stochastic Nernst-Planck-Navier-Stokes system with a blocking boundary conditions for ionic species concentrations in a smooth bounded domain $\mathcal{D}$. 
In both $2D$ and $3D$ cases, we establish the global existence of weak martingale solution when the capacitance $\varsigma>0$, 
and also establish the existence of a unique maximal strong pathwise solution 
and a blow-up criterion when the capacitance  $\varsigma=0$. In particular, we show that the maximal pathwise solution is global in $2D$ case without the restriction of smallness of initial data.
\end{abstract}

\maketitle
\section{Introduction}

Electrokinetic flows could be formed when the viscous Newtonian fluid with lots of positive and negative charged particles on the micro and nano scale acted by electrical force. The interaction between the charged particles and the fluid causes electrokinetic effect which includes two most well-known phenomena in physics: electrophoresis and electro-osmotic. Electrophoresis phenomenon describes the colloidal particle drew by an external electrical force through a fluid, while electro-osmotic characterizes the motion of an aqueous solution past a solid wall dragged by an electrical field. Here, we only focus on the phenomenon of electro-osmotic which captures many interesting features in the nature, therefore it could be applied to various specific areas, such as cooling system in microelectronics, separation and mixing techniques in analytical chemistry, filtration processes and material sciences, see \cite{bothe2, chang, prob, kirby} and for more applications see \cite{bothe, rubin, troj}.  In mathematics, the researchers use the Nernst-Planck-Navier-Stokes system to model the electro-osmotic flow. The incompressible Navier-Stokes equations as a subsystem of the Nernst-Planck-Navier-Stokes system, read as
$$\partial_t u-\mu\Delta u+(u\cdot \nabla)u-\nabla p+ \mathbf{F} =0,$$
with the divergence free condition
$${\rm div}u=0,$$
where $u$ is the velocity, $p$ is the pressure, $\mu$ is the viscosity coefficient. $\mathbf{F}:=-\rho\mathbf{E}$ is the Coulomb force, where $\rho$ is the local charge density given by
\begin{align}\label{e1.1}
\rho=\sum_{i=1}^{m}z_i\mathbf{c}^i,
\end{align}
for $m\in \mathbb{N}^+$ and $\mathbf{E}$ is the electric field
$$\mathbf{E}=-\nabla \Psi.$$
The electrical potential $\Psi$ solves the Poisson equation:
$$-\Delta \Psi=\rho.$$
The non-negative function $\mathbf{c}^i, i=1,2,\cdots, m,$ in \eqref{e1.1} is the $i$-th ionic species concentrations satisfying the following Nernst-Planck equations:
\begin{align}\label{1.2}
\partial_t\mathbf{c}^i+{\rm div}b^i=0,
\end{align}
where the fluxes $b^i$ is given by
$$u\mathbf{c}^i-a_i\nabla\mathbf{c}^i-a_i\frac{ez_i}{\mathbf{B}T}\mathbf{c}^i\nabla\Psi,$$
here, $T$ is temperature, $\mathbf{B}$ is the Boltzmann constant, $e$ is the elementary charge, the constants $a_i>0, i=1,2,\cdots, m$ are the ionic diffusivity, $z_i\in \mathbb{R}$ represent the valence. The single Nernst-Planck equations \eqref{1.2} also have wide range of applications. One of applications is in the area of semiconductor theory. In addition, the equations make it possible to characterize the function of neurons with the influence of electrical and chemical conduction which will conduce to well-understand the behaviour of neurons membrane, see \cite{biler2,corr,gaj1} for more details. For the related results of system \eqref{1.2}, in the literatures \cite{biler,choi,gaj}, the global existence and stability of solution have been investigated.

The transport of particles is easily affected by external random noise. Therefore, equations influenced by a random fluctuation contribute to the uncertainty in modelling the fluid dynamical systems, the uncertainty and randomness have a far-reaching impact for our understanding of the complex dynamical phenomenon, especially in biology, climate dynamics, gene regulation system. The monograph \cite{duan} gives more introduction of the stochastic dynamics and their applications. Therefore, to rigorous understand the turbulence and study stochastic PDEs, the randomness must be taken into account, which has been commonly accepted as an important research field both in theoretical analysis and practical application in the last decades. Here, we perturb the Navier-Stokes equations governing the fluids by a multiplicative noise, therefore the stochastic Nernst-Planck-Navier-Stokes system reads as the following:
\begin{eqnarray}\label{Equ1.1}
\left\{\begin{array}{ll}
\!\!\!\partial_t\mathbf{c}^i+u\cdot \nabla \mathbf{c}^i=a_i{\rm div}(\nabla \mathbf{c}^i+z_i\mathbf{c}^i\nabla\Psi),~i=1,2,\cdots, m,\\
\!\!\!-\Delta \Psi=\sum_{i=1}^m z_i\mathbf{c}^i=\rho,\\
\!\!\!\partial_t u-\mu\Delta u+(u\cdot \nabla)u-\nabla p+ \kappa\rho\nabla \Psi =f(u, \nabla \Psi)\frac{d\mathcal{W}}{dt},\\
\!\!\!\nabla\cdot u=0,\\
\end{array}\right.
\end{eqnarray}
where $f$ is a noise intensity operator,  $\kappa$ is a constant related to the temperature $T$ and the Boltzmann constant $\mathbf{B}$, $\mathcal{W}$ is a cylindrical Wiener process, the specific definition will be given in Section 2. System \eqref{Equ1.1} is imposed on the following initial data:
\begin{align}
u(0, x)=u_0, ~\mathbf{c}^i(0,x)=\mathbf{c}^i_0,
\end{align}
where we omit the random element $\omega$.
The velocity $u$ is equipped with the no-flux boundary condition:
\begin{align}\label{e1.5}
u|_{\partial \mathcal{D}}=0,
\end{align}
and the Poisson equation is equipped with the Robin boundary condition:
\begin{align}\label{e1.6}
(\partial_{\vec{n}}\Psi(x,t)+\varsigma\Psi(x,t))|_{\partial \mathcal{D}}=\eta,
\end{align}
where $\eta=\eta(x)$ is a given smooth function  that is an externally applied potential on the boundary in physics, the constant $\varsigma$ represents the capacitance of the double layer and $\vec{n}$ is the unit outward normal vector to the boundary $\partial\mathcal{D}$. The elliptic boundary value problem $\eqref{Equ1.1}_2$, \eqref{e1.6} originated from the Maxwell equations of electrostatics. For the ionic concentrations $\mathbf{c}^i$ we consider blocking boundary condition:
\begin{align}\label{e1.7}
(\partial_{\vec{n}}\mathbf{c}^i+z_i\mathbf{c}^i\partial_{\vec{n}}\Psi)|_{\partial \mathcal{D}}=0,
\end{align}
where boundary condition \eqref{e1.7} is closely related to the occurrence of electrochemical double layers in physics.

Taking integral on both sides of equation $\eqref{Equ1.1}_1$, using the boundary conditions \eqref{e1.5} and \eqref{e1.7}, from the physical point of view it is significant to  find that the function $\mathbf{c}^i$ satisfies the conservation form:\vspace{-2.2mm}
\begin{align}\label{e1.8}
\int_{\mathcal{D}}\mathbf{c}^idx=\int_{\mathcal{D}}\mathbf{c}^i_0dx,~ {\rm for~ all}~ t\geq 0.
\end{align}
Since the function $\mathbf{c}^i$ is non-negative, we infer that $\|\mathbf{c}^i(t)\|_{L^1}=\|\mathbf{c}^i_0\|_{L^1}$ for all $t\geq 0$.

Corresponding to different boundary conditions with different physical meanings, the deterministic Nernst-Planck-Navier-Stokes system has been investigated extensively in many literatures. The existence and in some cases uniqueness of weak solution with homogeneous Neumann boundary condition was proved in \cite{schm} for dimension $d\leq 3$, in addition, the local strong solution was verified under the assumption of no-flux boundary condition of ionic species concentrations. With respect to different species diffusivity and the magnitudes of the valences, many outstanding results were built in recent years. The global existence of weak solution with blocking boundary condition in $3D$ was proved in \cite{fischer}.  In \cite{bothe}, the authors proved the existence and uniqueness of local strong solution for any $d\geq 2$, and when $d=2$, the existence of unique global strong solution and exponential convergence to uniquely determined steady states with  blocking boundary condition for the ions and a Robin boundary condition for the electric potential were built.  For ionic concentrations that satisfy both blocking and selective boundary conditions, \cite{const} proved the global existence of strong solution in $2D$, and in the case of uniform selective boundary conditions, the solution was proved unconditional global stability, which converged to unique selected Boltzmann states. For more publications, we refer the readers to \cite{bedin, consta, jer, jero, ryham} which dealt with different physical boundary situations and the references therein.

Note that there is no work available on the stochastic version of system \eqref{Equ1.1}-\eqref{e1.7}, we are the first to consider the stochastic issues. However, for the stochastic Navier-Stokes equations, many remarkable literatures have been published, see \cite{Temam, Millet, Breckner, brec, Brze, cap, Real, Flandoli, Flandoli2, mena, Glatt, Rozovskii, ww}. In this paper, the existence of global weak martingale solution is proved in $2D$ and $3D$ cases under the assumption that two valences are opposite, also the existence of unique maximal strong pathwise solution is proved for arbitrary valences, species diffusivity, and then the blow-up criterion is established. In particular,  we prove that the maximal strong pathwise solution is global in time for large initial data in $2D$ case.

Since the Navier-Stokes equations as a subsystem, we can not expect any better result than it. But in our case, coupled constitution and the boundary conditions representing the presence of an electrical double layer are more complex, making the proof technically especially in constructing the approximate solutions and the a priori estimates. We adopt different strategy to build the approximate solutions. For one thing, in the process of proving weak martingale solution, we use the mixed method to build the approximate solutions. In this step, a two-layer approximate scheme is developed to surmount the obstacle stemming from stochastic integral. For another, in the process of proving strong pathwise solution, due to the restriction of blocking boundary conditions, we construct the smooth approximate solutions by linearizing the modified stochastic system to a linear system with additive noise and Neumann boundary condition. Therefore, after defining a suitable mapping, proving the well-posedness of approximate solutions of original system is equivalent to finding a fixed point of the mapping. The random element $\omega$ brings difficulties in closing the estimate of the mapping, we will further develop a stopping time technique to overcome it. Besides, we have to handle estimate of the boundary integral arising from the blocking boundary.

The rest of paper is organized as follows. In Section 2, we give some preliminaries, classical results and assumptions. Section 3 is devote to establishing the global existence of weak martingale solution by three steps for two oppositely charged ionic species. In Section 4, we establish the existence and uniqueness of maximal strong pathwise solution following the line of Yamada-Watanabe argument, and establish a blow-up criterion. In Section 5, we extend the maximal strong pathwise solution to a global solution in $2D$ case.

\section{Preliminaries}
In this section, we recall some deterministic, stochastic preliminaries and results used throughout this paper.

 For the smooth bounded domain $\mathcal{D}$, let $W^{k,p}(\mathcal{D}), k\in\mathbb{N}, p\in [1,\infty)$ be the Sobolev space having distributional derivatives up to $k$, and the derivatives are integrable in $L^p(\mathcal{D})$. When $p=2$, we denote by $H^k(\mathcal{D})=W^{k,2}(\mathcal{D})$.  For $\alpha\in (0,1)$, let $C^\alpha([0,T];X)$ be the $X$-valued $\alpha$-order H\"{o}lder continuous space with respect to time $t$, $L_w^p(0,T;X)$ be the space $L^p(0,T; X)$ with the weak topology, where $X$ is a Banach space. The notations $\mathcal{B}(X)$, $X^*$ represent the Borel $\sigma$-algebra of $X$ and the dual space of $X$, respectively.

 To define the variational setting, we introduce the following spaces:
\begin{align*}
&\mathcal{V}:=\big\{u: u\in C^\infty_c(\mathcal{D}), ~{\rm div} u=0\big\},\\
&\mathbb{L}^p(\mathcal{D}):=\overline{\mathcal{V}}^{\|\cdot\|_{L^p(\mathcal{D})}}, ~ \mathbb{H}^1(\mathcal{D}):=\overline{\mathcal{V}}^{\|\cdot\|_{H^1(\mathcal{D})}},
\end{align*}
for $p\in[1, \infty)$, endowed with the norms
$$\|\cdot\|_{\mathbb{L}^p}:=\|\cdot\|_{L^p(\mathcal{D})}, ~~\|\cdot\|_{\mathbb{H}^1}:=\|\nabla\cdot\|_{L^2(\mathcal{D})}.$$
When $p=2$, $\|\cdot\|_{\mathbb{L}^2}=\|\cdot\|_{L^2(\mathcal{D})}=(\cdot, \cdot)$ where $(\cdot, \cdot)$ means the inner product of $L^2(\mathcal{D})$. Thanks to the Poincar\'{e} inequality and the no-flux boundary condition of $u$, the norm of $\mathbb{H}^1(\mathcal{D})$ is equivalent to the norm of $H^1(\mathcal{D})$.

 The elliptic Neumann problem takes the form:
\begin{eqnarray*}
\left\{\begin{array}{ll}
\!\!\!-\Delta \pi =h,~ &{\rm in} ~\mathcal{D},\\
\!\!\!\partial_{\vec{n}} \pi =g, ~&{\rm on}~ \partial\mathcal{D},
\end{array}\right.
\end{eqnarray*}
where $h,g$ are given smooth functions. We recall the result in \cite{niren} giving the estimate
\begin{align}\label{e1.3}
\|\nabla \pi\|_{W^{k,p}(\mathcal{D})}\leq C\left(\|h\|_{W^{k-1,p}(\mathcal{D})}+\|g\|_{W^{k-\frac{1}{p},p}(\partial\mathcal{D})}\right),
\end{align}
where $C=C(k, p, \mathcal{D})>0$ is a constant.

Introducing the following lemma (see \cite{simon}), it could be applied for the time regularity estimate.
\begin{lemma}\label{lem6.1} Let $r^*=\left\{\begin{array}{ll}
\!\!\frac{dr}{d-r}, &{\rm if } ~r<d,\\
\!\!  {\rm any ~finite~nonnegetive~ real~ number},~&{\rm if}~r=d,\\
\!\! \infty, &{\rm if } ~r>d,
\end{array}\right.$ where $d=2,3$.  For $1\leq r\leq \infty, 1\leq s\leq \infty$, if $\frac{1}{r}+\frac{1}{s}\leq 1$, and $\frac{1}{r^*}+\frac{1}{s}=\frac{1}{t}$, $f\in W^{1,r}$ and $g\in W^{-1,s}$, then $fg\in W^{-1,t}$, that is,
$$\|fg\|_{W^{-1,t}}\leq \|f\|_{W^{1,r}}\|g\|_{W^{-1,s}}.$$
\end{lemma}

To control the boundary integral, we introduce the following trace inequality (see \cite[Lemma 1]{lee}).
\begin{lemma}\label{lem2.1*}Let $\mathcal{D}$ be a smooth bounded domain, for any $p\in [2,4]$, there exist constants $C_1, C_2$ such that
\begin{align}
\|g\|_{L^p(\partial\mathcal{D})}\leq C_1\|\nabla g\|_{L^2(\mathcal{D})}^\frac{1}{p}\|g\|_{L^{2(p-1)}(\mathcal{D})}^\frac{p-1}{p}+C_2\|g\|_{L^p(\mathcal{D})}.
\end{align}
If $p\in[2,4)$, for any $\delta>0$ there exists constant $C_\delta$ depending on $\mathcal{D}, p, \delta$ such that
\begin{align}
\|g\|_{L^p(\partial\mathcal{D})}^2\leq \delta \|\nabla g\|_{L^2(\mathcal{D})}^2+C_\delta\|g\|^2_{L^1(\mathcal{D})}.
\end{align}
\end{lemma}

\begin{lemma}[\!\!\cite{Tay}]\label{lem4.4}
Let $\mathcal{D}$ be a smooth bounded domain and $s\geq 0$. For every $\varepsilon>0$, the mollifier operator $\mathcal{J}_{\varepsilon}$ maps $H^{s}(\mathcal{D})$ into $H^{s'}(\mathcal{D})$ where $s'>s$ and has the following properties:

$i)$ the collection $\{\mathcal{J}_{\varepsilon}\}_{\varepsilon >0}$ is uniformly bounded in $H^{s}(\mathcal{D})$ independence of $\varepsilon$, i.e., there exists a positive constant C=C(s) such that,
$$\|\mathcal{J}_{\varepsilon}f\|_{H^s}\leq C\|f\|_{H^s},~~~f\in H^{s}(\mathcal{D});$$

$ii)$ for every $\varepsilon>0$, if $s\geq 1$  then for $f\in H^{s}(\mathcal{D})$,
\begin{equation*}
\|\mathcal{J}_{\varepsilon}f\|_{H^{s^\prime}}\leq \frac{C}{\varepsilon^{s'-s}}\|f\|_{H^{s}};
\end{equation*}

$iii)$ sequence $\mathcal{J}_{\varepsilon}f$ converges to $f$, for $f\in H^{s}(\mathcal{D})$, that is,
\begin{equation*}
\lim_{\varepsilon\rightarrow 0}\|\mathcal{J}_{\varepsilon}f-f\|_{H^s}=0.
\end{equation*}
\end{lemma}
\begin{remark}\label{rem4.1} The mollifier operator $\mathcal{J}_{\varepsilon}$ could be defined as $\mathcal{J}_{\varepsilon}u=\mathcal{R}\widetilde{\mathcal{J}_{\varepsilon}}Eu$, where $\widetilde{\mathcal{J}_{\varepsilon}}$ is a standard Friedrich's mollifier on $\widetilde{\mathcal{D}}$ which is the un-boundary extend of smooth domain $\mathcal{D}$, $E$ is a extension operator from $H^s(\mathcal{D})$ into $H^s(\widetilde{\mathcal{D}})$, and  $\mathcal{R}$ is a restriction operator from $H^s(\widetilde{\mathcal{D}})$ into $H^s(\mathcal{D})$ (see \cite[Chapter 5]{Evans} for more details).
\end{remark}

Next, we introduce the stochastic background. Let $\mathbf{S}:=(\Omega,\mathcal{F},\{\mathcal{F}_{t}\}_{t\geq0},\mathbb{P}, \mathcal{W})$ be a fixed stochastic basis and $(\Omega,\mathcal{F},\mathbb{P})$ be a complete probability space. $\{\mathcal{F}_{t}\}_{t\geq0}$ is a filtration satisfying all usual conditions.
Denote by $L^p(\Omega; L^q(0,T;X)), p\in [1,\infty), q\in [1, \infty]$ the space of processes with values in $X$ defined on $\Omega\times [0,T]$ such that

i. $u$ is measurable with respect to $(\omega, t)$, and for each $t$, $u(t)$ is $\mathcal{F}_{t}$-measurable.

ii. For almost all $(\omega, t)$, $u\in X$ and
\begin{align*}
\|u\|^p_{L^p(\Omega; L^q(0,T;X))}\!=\!
\begin{cases}
\mathbb{E}\left(\int_{0}^{T}\|u\|_{X}^qdt\right)^\frac{p}{q},~& {\rm if}~q\in [1,\infty),\\
\mathbb{E}\left(\sup_{t\in[0,T]}\|u\|^p_X\right),~ & {\rm if}~q=\infty.
\end{cases}
\end{align*}
Here, $\mathbb{E}$ denotes the mathematical expectation.

To control the stochastic term, we introduce the following well-known Burkholder-Davis-Gundy inequality: for any $\sigma\in  L^{2}(\Omega;L^{2}_{loc}([0,\infty),L_{2}(\mathcal{H},X)))$ (the definition of space $L_{2}(\mathcal{H},X)$ see below),  by taking $\sigma_{k}=\sigma e_{k}$, it holds
\begin{eqnarray*}
\mathbb{E}\left(\sup_{t\in [0,T]}\left\|\int_{0}^{t}\sigma d\mathcal{W}\right\|_{X}^{p}\right)\leq c_{p}\mathbb{E}\left(\int_{0}^{T}\|\sigma\|_{L_{2}(\mathcal{H},X)}^{2}dt\right)^{\frac{p}{2}}
=c_{p}\mathbb{E}\left(\int_{0}^{T}\sum_{k\geq 1}\|\sigma_k\|_{X}^{2}dt\right)^{\frac{p}{2}},
\end{eqnarray*}
for any $p\in[1,\infty)$. Here, $\mathcal{W}$ is the cylindrical Wiener process in system \eqref{Equ1.1} defined on an Hilbert space $\mathcal{H}$, which is adapted to the filtration $\{\mathcal{F}_{t}\}_{t\geq 0}$. Namely, $\mathcal{W}=\sum_{k\geq 1}e_k\beta_{k}$ with $\{e_k\}_{k\geq 1}$ is a complete orthonormal basis of $\mathcal{H}$ and $\{\beta_{k}\}_{k\geq 1}$ is a sequence of independent standard one-dimensional Brownian motions.

We consider an auxiliary space $\mathcal{H}_0\supset \mathcal{H}$, defined by
\begin{eqnarray*}
\mathcal{H}_0=\left\{h=\sum\nolimits_{k\geq 1}\alpha_k e_k: \sum\nolimits_{k\geq 1}\alpha_k^2k^{-2}<\infty\right\},
\end{eqnarray*}
 with the norm $\|h\|_{\mathcal{H}_0}^2=\sum_{k\geq 1}\alpha_k^2k^{-2}$. We have that $\mathcal{W}\in C([0,\infty),\mathcal{H}_0)$ almost all $\omega$, see \cite{Zabczyk}.

To pass the limit in stochastic integral of approximate sequence, we need the following result from \cite[Lemma 2.1]{Debussche}.
\begin{lemma}\label{lem4.6} Assume that $f_\varepsilon$ is a sequence of $X$-valued $\mathcal{F}_{t}^\varepsilon$-predictable processes such that
$$f_\varepsilon\rightarrow f~ {\rm in~ probability~ in}~L^2(0,T;L_2(\mathcal{H};X)),$$
and the cylindrical Wiener process sequence $\mathcal{W}_\varepsilon$ satisfies
$$\mathcal{W}_\varepsilon\rightarrow \mathcal{W}~ {\rm in ~probability~ in}~C([0,\infty);\mathcal{H}_0),$$
then,
$$\int_{0}^{t}f_\varepsilon d\mathcal{W}_\varepsilon\rightarrow \int_{0}^{t}f d\mathcal{W} ~ {\rm in ~probability~ in}~ L^2(0,T;X).$$
\end{lemma}

At the end of this section, we give several assumptions of noise operator $f$.

{\bf Assumptions}.  Assume that the noise operator $f:H\rightarrow L_2(\mathcal{H};L^2)$ satisfies the usual Lipschitz and linear growth conditions, that is, there exist two positive constants $\ell_1, \ell_2$ such that
\begin{align}
\|f(u, \nabla\Psi)\|_{L_2(\mathcal{H};L^2)}^2&\leq \ell_1(\|u\|_{\mathbb{L}^2}^2+\|\nabla\Psi\|^2_{L^2}),\label{as1}\\
\|f(u_1, \nabla\Psi_1)-f(u_2, \nabla\Psi_2)\|_{L_2(\mathcal{H};L^2)}^2&\leq \ell_2(\|u_1-u_2\|_{\mathbb{L}^2}^2+\|\nabla\Psi_1-\nabla\Psi_2\|^2_{L^2}), \label{as2}
\end{align}
for $u\in \mathbb{L}^2, \nabla\Psi\in L^2$, where $L_{2}(\mathcal{H},X)$ denotes the collection of Hilbert-Schmidt operators, that is, the set of all linear operators $G$ from $\mathcal{H}$ to a Banach space $X$, with the norm $\|G\|_{L_{2}(\mathcal{H},X)}^2=\sum_{k\geq 1}\|Ge_k\|_{X}^2$.
The assumptions \eqref{as1}, \eqref{as2} will be used for establishing the existence of weak martingale solution and the uniqueness of pathwise strong solution.

Moreover, to get the existence of strong pathwise solution, we need further assumptions: there exist two positive constants $\ell_3, \ell_4$ such that the operator $f:H\rightarrow L_2(\mathcal{H};H^1)$ satisfies
\begin{align}
\|f(u, \nabla\Psi)\|_{L_2(\mathcal{H};H^1)}^2&\leq \ell_3(\|u\|_{\mathbb{H}^1}^2+\|\nabla\Psi\|^2_{H^1}),\label{as3}\\
\|f(u_1, \nabla\Psi_1)-f(u_2, \nabla\Psi_2)\|_{L_2(\mathcal{H};H^1)}^2&\leq \ell_4(\|u_1-u_2\|_{\mathbb{H}^1}^2+\|\nabla\Psi_1-\nabla\Psi_2\|^2_{H^1}),\label{as4}
\end{align}
for any $u\in \mathbb{H}^1, \nabla\Psi\in H^1$.

We arrange the definition of solutions and main results to the forthcoming sections.

\section{Global existence of weak martingale solution}
 In this section, we consider the case of two charge species (i.e. $i=1,2$) and establish the global existence of weak martingale solution in both $2D$ and $3D$ cases. The proof could be completed by constructing the approximate solutions, proving  the uniform estimates and stochastic compactness, and identifying the limit. Hereafter, we take $\varsigma$=1 for simplicity.

 First, we give the definition of a weak martingale solution.
\begin{definition}[weak martingale solution]\label{def3.1} Let $\lambda$ be a Borel probability measure on space $L^2(\mathcal{D})\times \mathbb{L}^2(\mathcal{D})$ with $$\int_{L^2(\mathcal{D})\times \mathbb{L}^2(\mathcal{D})}|x|^p d\lambda\leq C,$$ for a positive constant $C$. We call $(\Omega, \mathcal{F},\{\mathcal{F}_t\}_{t\geq 0},\mathbb{P}, u, \mathbf{c}^i, \mathcal{W}), i=1,2$ is a weak martingale solution to system \eqref{Equ1.1}-\eqref{e1.7} with the initial data $\lambda$ if

i. $(\Omega,\mathcal{F},\{\mathcal{F}_t\}_{t\geq 0},\mathbb{P})$ is a stochastic basis with a complete right-continuous filtration, $\mathcal{W}$ is a Wiener process relative to the filtration $\mathcal{F}_t$;

ii. the process $u$ is an $\mathbb{L}^2$-valued $\mathcal{F}_t$-progressively measurable satisfying
$$u\in L^p(\Omega; L^\infty(0,T; \mathbb{L}^2)\cap L^2(0,T;\mathbb{H}^1)) \mbox{ for any } p\in [2,\infty);$$

iii.~the processes $\mathbf{c}^i, i=1,2$ are $L^2$-valued $\mathcal{F}_t$-progressively measurable satisfying
$$\mathbf{c}^i\geq 0,~\mbox{a.e.} ~{\rm and} ~\mathbf{c}^i\in L^p(\Omega; L^\infty(0,T; L^2)\cap L^2(0,T;H^1)) \mbox{ for any } p\in [2,\infty),$$
and $\Psi$ is an $H^2$-valued $\mathcal{F}_t$-progressively measurable process with
$$\Psi\in L^p(\Omega; L^\infty(0,T; H^2)\cap L^2(0,T;H^3));$$

iv. the initial law $\lambda=\mathbb{P}\circ (u_0, \mathbf{c}^i_0)^{-1}$;

v. for $\varphi\in \mathbb{H}^1, \phi\in H^1$, it holds $\mathbb{P}$-a.s.
\begin{eqnarray*}
&&(\mathbf{c}^i(t),\phi)=(\mathbf{c}^i(0),\phi)-\int_{0}^{t}(u\cdot \nabla \mathbf{c}^i, \phi)ds-\int_{0}^{t}a_i(\nabla \mathbf{c}^i+z_i\mathbf{c}^i\nabla\Psi, \nabla\phi)ds,~i=1,2,\\
&&(u(t),\varphi)=(u(0), \varphi)-\mu\int_{0}^{t}(\nabla u, \nabla\varphi)ds-\int_{0}^{t}(u\cdot \nabla u,\varphi)ds- \int_{0}^{t}(\kappa\rho\nabla \Psi,\varphi) ds\nonumber \\ &&\qquad\qquad+\int_{0}^{t}(f(u,\nabla\Psi), \varphi)d\mathcal{W},\\
&&-\Delta \Psi=\sum_{i=1}^2 z_i\mathbf{c}^i=\rho, ~\mbox{a.e.}
\end{eqnarray*}
for all $t>0$.
\end{definition}

The following theorem is the main result of this section.
\begin{theorem}
Assume that the initial data $(u_0, \mathbf{c}_0^i), i=1,2$ are $\mathcal{F}_0$-measurable random variables satisfying
$$u_0\in L^p(\Omega; \mathbb{L}^2), ~\mathbf{c}_0^i\in L^p(\Omega; L^2)~{\rm and }~ \mathbf{c}_0^i\geq 0,~ \mathbb{P}\mbox{-a.s.}$$
for all $p\in [2,\infty)$. Suppose that $f$ satisfies the assumptions \eqref{as1}, \eqref{as2} and the two valence $z_i,~ i=1,2$ are opposite $(z_1>0>z_2)$. Moreover, the given smooth function $\eta$ belongs to $H^\frac{3}{2}(\partial\mathcal{D})$. Then, there exists a global weak martingale solution of system  \eqref{Equ1.1}-\eqref{e1.7} for $2D$ and $3D$ cases in the sense of Definition \ref{def3.1}.
\end{theorem}
\begin{remark}
In $2D$ case, using the estimate
$$\|(u\cdot \nabla v, u)\|\leq \varepsilon\|\nabla u\|_{\mathbb{L}^2}^2+ C(\varepsilon)\|u\|_{\mathbb{L}^2}^2\|\nabla v\|_{\mathbb{L}^2}^2,~ u,v \in \mathbb{H}^1$$
to nonlinear terms and an easier argument than that of in Step 4.2, we can also deduce that the solution is unique.
\end{remark}
\begin{remark}Here we only consider the case of the two opposite charged species. Under this case, it holds
\begin{align}\label{3.1*}
\rho (z^2_1(\mathbf{c}^1)^2-z^2_2(\mathbf{c}^2)^2)=\rho^2 (|z_1|\mathbf{c}^1+|z_2|\mathbf{c}^2)>0,
\end{align}
see \eqref{e3.17} in Lemma \ref{lemma3.2}, which is an essential element in achieving the uniform $L^2(\mathcal{D})$ estimate of $\mathbf{c}^i$. A natural question is: Can we extend the result to the multiple species setting, that is $m>2$ ? Unfortunately, inequality \eqref{3.1*} fails for the multiple species case, therefore we have trouble in obtaining the uniform $L^2(\mathcal{D})$ estimate. However, under the following special multiple species setting: all diffusivities are equal, i.e., $a_1=a_2=\cdots=a_m=a$ and
all valences have the same magnitude, i.e., $|z_1|=|z_2|=\cdots=|z_m|=z$, we could give a positive answer of the problem analogues to the deterministic case \cite{consta, lee}.
Indeed, let $\chi=z(\mathbf{c}^1+\cdots+\mathbf{c}^m)$, then $\chi$ and $\rho$ satisfy equation $\eqref{Equ1.1}_1$, that is
\begin{align*}
\partial_t\rho+u\cdot\rho=a{\rm div}(\nabla\rho+z\chi\nabla\Psi),\;
\partial_t\chi+u\cdot\chi=a{\rm div}(\nabla\chi+z\rho\nabla\Psi).
\end{align*}
After taking inner product with $\chi$ and $\rho$ respectively, we have
\begin{align}\label{3.2}
 z(\chi\nabla\Psi, \nabla\rho)\leq -z(\rho\nabla\Psi, \nabla\chi)+z\int_{\partial\mathcal{D}}\sigma\rho(\Psi-\eta)d\mathcal{S}.
\end{align}
Noted that the first term on the right-hand side of \eqref{3.2} could be cancelled, and the second term could be estimated as Lemma \ref{lemma3.2}. Consequently, we can obtain the desired $L^2$ estimate.
\end{remark}

\subsection{Approximate solutions and uniform estimates}

At first, we try to find the approximate solutions to the following approximate system:
\begin{eqnarray}\label{Equ3.1}
\left\{\begin{array}{ll}
\!\!\!\partial_t\mathbf{c}_n^i+u_n\cdot \nabla \mathbf{c}_n^i=a_i{\rm div}(\nabla \mathbf{c}_n^i+z_i\mathbf{c}_n^i\nabla\Psi_n),~i=1,2,\\
\!\!\!-\Delta \Psi_n=\sum_{i=1}^2 z_i\mathbf{c}^i_n=\rho_n,\\
\!\!\!\partial_t  u_n-\mu P\Delta u_n+P_nP(u_n\cdot \nabla)u_n+ \kappa P_nP\rho_n\nabla \Psi_n =P_nPf(u_n, \nabla \Psi_n)\frac{d\mathcal{W}}{dt},\\
\end{array}\right.
\end{eqnarray}
with the initial data $\mathbf{c}_n^i(0)=\mathbf{c}^i_0, u_n(0)=P_n u_0$,
where $P$ is the Helmholtz-Leray projection from $L^2$ into $\mathbb{L}^2$ and $P_n$ is an orthogonal projection from $\mathbb{L}^2$ into $X_n$. The space $X_n$ is defined by
\begin{eqnarray*}
X_n=\mbox{span}\{\psi_1,\ldots ,\psi_n\},
\end{eqnarray*}
where the sequence $\{\psi_n\}_{n\geq 1}$ is  an orthonormal basis for $\mathbb{L}^2$ of the Stokes operator $A:=-P\Delta$.

For the deterministic system, the global existence and uniqueness of approximate solutions are established in \cite[Lemma 4.1]{fischer}. The proof was implemented by a combined method, the ionic species concentrations equation was solved by the Leray-Schauder fixed-point theorem, while the approximate solutions of the velocity equation were established by fixed point theorem. In our case, we still solve the first equation following same line of \cite{fischer} to obtain that for any smooth function $u$, the solution $$\mathbf{c}^i\in L^\infty(0,T;L^p)\cap L^2(0,T;H^1)\mbox { for }p\in [1,\infty),$$ $\mathbb{P}$-a.s. Since the initial data $\mathbf{c}_n^i(0)\geq 0$, then the solution $\mathbf{c}_n^i$ remains positive a.e. for all $t\geq 0$ using the maximal principle (for the proof, see \cite{const, schm}).

Furthermore, defining a solution mapping $\mathcal{T}$, we show that the mapping $\mathbf{c}^i=\mathcal{T}(u)$ is continuous in $L^\infty(0,T;L^2)\cap L^2(0,T;H^1)$. Indeed, for $\mathbf{c}_m^i:=\mathcal{T}(u_m)$ where $u_m$ is a sequence in $L^\infty(0,T;C^1(\mathcal{D}))$ and $\mathbf{c}^i$, the difference $\hat{\mathbf{c}}^i=\mathbf{c}^i-\mathbf{c}_m^i, \hat{u}=u-u_m$ with
\begin{equation}\label{new3.2}
\lim_{m\rightarrow\infty}\|u_m-u\|_{L^\infty(0,T;C^1(\mathcal{D}))}=0,
\end{equation}
satisfies (in the weak sense)
\begin{align*}
\partial_t \hat{\mathbf{c}}^i+(\hat{u}\cdot \nabla) \mathbf{c}^i+(u_m\cdot\nabla)\hat{\mathbf{c}}^i=a_i{\rm div}(\nabla \hat{\mathbf{c}}^i+z_i\hat{\mathbf{c}}^i\nabla \Psi_1+\mathbf{c}_m^i\nabla\hat{\Psi}).
\end{align*}
Taking inner product with $\hat{\mathbf{c}}^i$, using the boundary condition, we obtain
\begin{align}\label{3.2*}
\frac{1}{2}\|\hat{\mathbf{c}}^i\|_{L^2}^2+\int_{0}^{t}(\hat{u}\cdot \nabla \mathbf{c}^i, \hat{\mathbf{c}}^i)ds=-\int_{0}^{t}a_i(\nabla \hat{\mathbf{c}}^i+z_i\hat{\mathbf{c}}^i\nabla \Psi_1+\mathbf{c}_m^i\nabla\hat{\Psi}, \nabla\hat{ \mathbf{c}}^i)ds.
\end{align}
Using the H\"{o}lder inequality and the embedding $H^1\hookrightarrow L^6$, we get
\begin{align*}
&\left|-\int_{0}^{t}(\hat{u}\cdot \nabla \mathbf{c}^i_1, \hat{\mathbf{c}}^i)ds\right|\leq \int_{0}^{t} \|\nabla\mathbf{c}^i\|^2_{L^2}\|\hat{\mathbf{c}}^i\|^2_{L^2}+C\|\hat{u}\|^2_{L^\infty}ds,\\
&\left|-\int_{0}^{t}a_i(z_i\hat{\mathbf{c}}^i\nabla \Psi_1+\mathbf{c}_m^i\nabla\hat{\Psi}, \nabla\hat{ \mathbf{c}}^i)ds\right|\\
&\leq \frac{a_i}{2}\int_{0}^{t}\|\nabla\hat{ \mathbf{c}}^i\|_{L^2}^2ds+C\int_{0}^{t}(\|\hat{\mathbf{c}}^i\|_{L^2}^2\|\nabla \Psi_1\|^2_{W^{1,3+}}+\|\nabla \hat{\Psi}\|^2_{L^6}\|\mathbf{c}_m^i\|_{L^3}^2)ds\\
&\leq \frac{a_i}{2}\int_{0}^{t}\|\nabla\hat{ \mathbf{c}}^i\|_{L^2}^2ds+C\int_{0}^{t}(\|\hat{\mathbf{c}}^i\|_{L^2}^2\|\nabla \Psi_1\|^2_{W^{1,3+}}+\| \hat{\mathbf{c}}^i\|^2_{L^2}\|\mathbf{c}_m^i\|_{L^3}^2)ds,
\end{align*}
where the constant $3^+:=3+\delta$ for small $\delta>0$. Combining the above estimates, by the Gronwall lemma and the bound of $\mathbf{c}^i$, \eqref{new3.2} and \eqref{3.2*}, we have
$$\lim_{m\rightarrow\infty}\bigg(\sup_{t\in [0,T]}\|\hat{ \mathbf{c}}^i\|_{L^2}^2+a_i\int_{0}^{T}\|\nabla\hat{ \mathbf{c}}^i\|_{L^2}^2dt\bigg)=0.$$

Now, we proceed to solve the velocity equation by the Banach fixed point theorem in the space $L^p(\Omega; L^\infty(0,T;X_n))$ for small $T$. For the stochastic issue, we have to further cut-off the velocity $u_n$. If not, we can not close the estimate caused by the integral $L^p(\Omega)$.

Define $C^\infty$-smooth cut-off function
\begin{eqnarray*}
\Theta_M(z)\!=\!\left\{\begin{array}{ll}
\!\!\!1,~ |z|\leq M,\\
\!\!\!0,~ |z|>2M.
\end{array}\right.
\end{eqnarray*}
For any $\mathbf{v}=\sum_{i=1}^n \mathbf{v}_i\psi_i\in X_n$, defining $\mathbf{v}^M=\sum_{i=1}^{n}\Theta_M(\mathbf{v}_i)\mathbf{v}_i\psi_i$ as \cite{Hofmanova}, then we have $$\|\mathbf{v}^M\|_{C([0,T];X_n)}\leq 2M,$$ and
$$\|\mathbf{v}_1^M-\mathbf{v}_2^M\|_{X_n}\leq C(n)\|\mathbf{v}_1-\mathbf{v}_2\|_{X_n}.$$
Define the functional
\begin{align}
\langle\mathcal{U}[\rho, u],\psi_i\rangle=\int_{\mathcal{D}}(\mu\Delta u-u\cdot \nabla u-\kappa\rho\nabla \Psi)\psi_i dx,
\end{align}
for $\psi_i\in X_n$.

We now find the approximate velocity field  $u_n$ of the following momentum equation:
\begin{align}\label{def}
&u(t)=u_0^M+\int_{0}^{t}\mathcal{U}[\rho(u^M), u^M]ds+\int_{0}^{t}f(u^M, \nabla \Psi(u^M) )d\mathcal{W}.
\end{align}
Observe that, here $u_0=u_0^M$. Define by the mapping $\mathfrak{M}$ the right-hand side of \eqref{def}, and we prove that the mapping is a contraction.

Denote $\rho_n^l=\rho(u_n^{M,l}), \Psi_n^l=\Psi(u_n^{M,l}), \mathbf{c}_n^{i,l}=\mathcal{T}(u^{M,l}_n), l=1,2$. After the truncation, we obtain 
\begin{align}\label{nb}
\sup_{t\in [0,T]}\|\mathbf{c}_n^{i}\|_{L^2}^2\leq C(M),
\end{align}
where the constant $C(M)$ does not depend on $w$.
Applying the continuity of mapping $\mathcal{T}(u)$, bound \eqref{nb} and the equivalence of finite-dimensional norm, we have
\begin{align*}
&\mathbb{E}\left(\sup_{t\in [0,T]}\left\|\int_{0}^{t}P_nP(\rho^1_n\nabla\Psi^1_n-\rho^2_n\nabla\Psi^2_n)ds\right\|_{X_n}^p\right)\\
&\leq C\mathbb{E}\left[\sup_{t\in [0,T]}\left(\int_{0}^{t}\left\|P_nP(\rho^1_n\nabla\Psi^1_n-\rho^2_n\nabla\Psi^2_n)\right\|_{\mathbb{L}^\frac{3}{2}}ds\right)^p\right]\\
&\leq C\mathbb{E}\left[\sup_{t\in [0,T]}\left(\int_{0}^{t}\|\rho^1_n-\rho^2_n\|_{L^2}\|\nabla\Psi^1_n\|_{L^6}+\|\rho_n^2\|_{L^2}\|\nabla\Psi^1_n-\nabla\Psi^2_n\|_{L^6}ds\right)^p\right]\\
&\leq CT^p\mathbb{E}\left[\sup_{t\in [0,T]}\left(\sum_{i=1}^2\|\mathbf{c}_n^{i,1}-\mathbf{c}_n^{i,2}\|^p_{L^2}\sum_{i=1}^2\|(\mathbf{c}_n^{i,1}, \mathbf{c}_n^{i,2)}\|^p_{L^2}\right)\right]\\
&\leq C(M)T^p\mathbb{E}\left(\sup_{t\in [0,T]}\sum_{i=1}^2\|\mathbf{c}_n^{i,1}-\mathbf{c}_n^{i,2}\|^{p}_{L^2}\right)\\
&\leq C(M)T^p\mathbb{E}\left(\sup_{t\in [0,T]}\|u_n^{M,1}-u_n^{M,2}\|^{p}_{X_n}\right)\\
&\leq C(n,M)T^p\mathbb{E}\left(\sup_{t\in [0,T]}\|u_n^{1}-u_n^{2}\|^{p}_{X_n}\right).
\end{align*}
By the global Lipschitz assumption on the noise operator $f$, the equivalence of finite-dimensional norm and the Burkholder-Davis-Gundy inequality, we have
\begin{align*}
&\mathbb{E}\left(\sup_{t\in [0,T]}\left\|\int_{0}^{t}P_nPf(u^{M,1}_n, \nabla \Psi^{1}_n)-P_nPf(u^{M,2}_n, \nabla \Psi^2_n)d\mathcal{W}\right\|_{X_n}^p\right)\\
&\leq C\mathbb{E}\left(\int_{0}^{T}\left\|P_nPf(u^{M,1}_n, \nabla \Psi^1_n)-P_nPf(u^{M,2}_n, \nabla \Psi^2_n)\right\|^2_{X_n}dt\right)^\frac{p}{2}\\
&\leq C\mathbb{E}\left(\int_{0}^{T}\left\|P_nPf(u^{M,1}_n, \nabla \Psi^1_n)-P_nPf(u^{M,2}_n, \nabla \Psi^2_n)\right\|^2_{\mathbb{L}^2}dt\right)^\frac{p}{2}\\
&\leq C(n)\ell_2^pT^\frac{p}{2}\mathbb{E}\left(\sup_{t\in [0,T]}\|u^1_n-u^2_n\|_{X_n}^{p}\right).
\end{align*}
For the rest of nonlinear term, we could estimate it by same argument as the Navier-Stokes equations. Considering these estimates, for a small time $T$, we build the contraction of $\mathfrak{M}$. Then, the local existence and uniqueness of approximate solutions to equation \eqref{def} follows by the Banach fixed point theorem for any fixed $n,M$.

Now, for fixed $n$, we can pass $M\rightarrow \infty$ to establish the existence of unique approximate solutions to system \eqref{Equ3.1}. Using Lemma \ref{lem3.1}, and a standard extension theory \cite[Corollary 3.2]{Hofmanova}, we could complete the proof.
Next, we verify the following necessary uniform a priori estimates.
\begin{lemma}\label{lem3.1}The approximate solutions $(\mathbf{c}_n^i, u_n)$ have the following a priori estimate for any $p\in[1,\infty), T>0$
\begin{eqnarray}
&&\mathbb{E}\left[\sup_{t\in [0,T]}\left(\|u_n\|_{\mathbb{L}^2}^2+\sum_{i=1}^{2}\int_{\mathcal{D}}\mathbf{c}_n^i{\rm log}\mathbf{c}_n^idx+\|\nabla\Psi_n\|_{L^2}^2+\|\Psi_n\|_{L^2(\partial\mathcal{D})}^2\right)^p\right]\nonumber\\&&+\mathbb{E}\left(\int_{0}^{T}\int_{\mathcal{D}}\mathbf{c}_n^i|\nabla \theta^i_n|^2dxdt\right)^p
+\mathbb{E}\left(\int_{0}^{T}\mu\|\nabla u_n\|_{\mathbb{L}^2}^2dt\right)^p\leq C,
\end{eqnarray}
where $\theta^i_n={\rm log}\mathbf{c}_n^i+z_i\Psi_n$, the constant $C$ is independent of $n, M$.
\end{lemma}
\begin{proof} To simplify the notation, we use $(u, \mathbf{c}^i)$ instead of $(u_n, \mathbf{c}_n^i)$. Define the stopping time $\tau_M$ as
\begin{eqnarray*}
\tau_M=\inf\left\{t\geq 0; \sup_{s\in [0,t]}\|u\|_{\mathbb{L}^2}^2\geq M\right\}.
\end{eqnarray*}
If the set is empty, we take $\tau_M=T$. Notice that $\tau_M$ is an increasing sequence with $\lim_{ M\rightarrow \infty}\tau_M= T$ and equation \eqref{def} coincides with the original equation on interval $[0, \tau_M)$.

 We use the It\^{o} formula to the function $\frac{1}{2}\|u\|_{\mathbb{L}^2}^2$, and apply the fact that
$(u\cdot \nabla u,u)=0$
to find
\begin{align}\label{e3.3}
\frac{1}{2}d\|u\|_{\mathbb{L}^2}^2+\mu\|\nabla u\|_{\mathbb{L}^2}^2dt&=-(\kappa\rho\nabla\Psi, u)dt+(f(u, \nabla \Psi), u)d\mathcal{W}\nonumber\\&\quad+\frac{1}{2}\|P_nPf(u, \nabla\Psi)\|_{L_2(\mathcal{H};\mathbb{L}^2(\mathcal{D}))}^2dt.
\end{align}
Note that, equation $\eqref{Equ3.1}_1$ could be rewritten as
\begin{align}\label{e3.1}
\partial_t\mathbf{c}^i+u\cdot \nabla \mathbf{c}^i=a_i{\rm div}(\mathbf{c}^i\nabla\theta_i).
\end{align}
Taking the inner product with $\theta_i$ on both sides of \eqref{e3.1}, we have
\begin{align}\label{e3.2}
 (\partial_t\mathbf{c}^i, \theta_i)+(u\cdot \nabla \mathbf{c}^i, \theta_i)=-a_i\|\sqrt{\mathbf{c}^i}\nabla\theta_i\|_{L^2}^2.
\end{align}
For the first term on the left-hand side of \eqref{e3.2}, integrating by parts and using the Robin boundary condition \eqref{e1.6}, we obtain 
\begin{align}
\sum_{i=1}^{2}(\partial_t\mathbf{c}^i, \theta_i)&=\sum_{i=1}^{2}(\partial_t\mathbf{c}^i, {\rm log}\mathbf{c}^i+z_i\Psi)\nonumber\\
&=\sum_{i=1}^{2}\frac{d}{dt}\int_{\mathcal{D}}\mathbf{c}^i {\rm log}\mathbf{c}^i-\mathbf{c}^idx-(\partial_t\Delta \Psi, \Psi)\nonumber\\
&=\sum_{i=1}^{2}\frac{d}{dt}\int_{\mathcal{D}}\mathbf{c}^i {\rm log}\mathbf{c}^i+\frac{d}{dt}\|\nabla \Psi\|_{L^2}^2-(\partial_t\partial_n\Psi, \Psi)\nonumber\\
&=\sum_{i=1}^{2}\frac{d}{dt}\int_{\mathcal{D}}\mathbf{c}^i {\rm log}\mathbf{c}^i+\frac{d}{dt}\|\nabla \Psi\|_{L^2}^2+\frac{1}{2}\frac{d}{dt}\|\Psi\|_{L^2(\partial\mathcal{D})}^2.
\end{align}
In the second line, we have used $$\frac{d}{dt}\int_{\mathcal{D}}\mathbf{c}^i {\rm log}\mathbf{c}^i-\mathbf{c}^idx=\frac{d}{dt}\int_{\mathcal{D}}\mathbf{c}^i {\rm log}\mathbf{c}^idx,$$ due to the fact that $\mathbf{c}^i$ satisfies the quantity of mass conservation, see \eqref{e1.8}.

For the second term on the left-hand side of \eqref{e3.2}, using the non-slip boundary condition of $u$, {taking sum over $i$, we have
\begin{align}\label{e3.7}
\sum_{i=1}^{2}(u\cdot \nabla \mathbf{c}^i, \theta_i)&=\sum_{i=1}^{2}\int_{\mathcal{D}}u\cdot\nabla(\mathbf{c}^i {\rm log}\mathbf{c}^i-\mathbf{c}^i)dx+\sum_{i=1}^{2}(u\cdot\nabla\mathbf{c}^i, z_i\Psi )\nonumber\\
&=(u\cdot\nabla\rho, \Psi)=-(u\cdot \nabla\Psi, \rho).
\end{align}}
Combining \eqref{e3.3}-\eqref{e3.7}, we arrive at
\begin{align}\label{3.9}
&d\left(\frac{1}{2}\|u\|_{\mathbb{L}^2}^2+\sum_{i=1}^{2}\int_{\mathcal{D}}\mathbf{c}^i{\rm log}\mathbf{c}^idx+\frac{1}{2}\|\nabla\Psi\|_{L^2}^2+\frac{1}{2}\|\Psi\|_{L^2(\partial\mathcal{D})}^2\right)+\left(\sum_{i=1}^{2}a_i\|\sqrt{\mathbf{c}^i}\nabla\theta_i\|_{L^2}^2+\mu\|\nabla u\|_{L^2}^2\right)dt\nonumber\\
&=(f(u, \nabla\Psi), u)d\mathcal{W}+\frac{1}{2}\|P_nPf(u, \nabla\Psi)\|_{L_2(\mathcal{H};\mathbb{L}^2(\mathcal{D}))}^2dt.
\end{align}
 By the Burkholder-Davis-Gundy inequality and assumption \eqref{as1}, we get
\begin{align}\label{e3.11}
&\mathbb{E}\left(\sup_{t\in [0,T]}\left|\int_{0}^{t\wedge \tau_M}(f(u, \nabla \Psi), u)d\mathcal{W}\right|^p\right)\nonumber\\
&\leq C\mathbb{E}\left(\int_{0}^{T\wedge \tau_M}\sum_{k\geq 1}\left(\int_{\mathcal{D}}u\cdot f(u, \nabla \Psi)e_kdx\right)^2dt\right)^\frac{p}{2}\nonumber\\
&\leq C\mathbb{E}\left(\int_{0}^{T\wedge\tau_M}\|u\|_{\mathbb{L}^2}^2\sum_{k\geq 1}\left(\int_{\mathcal{D}}Pf^2(u, \nabla \Psi)e^2_kdx\right)dt\right)^\frac{p}{2}\nonumber\\
&\leq \frac{1}{2}\mathbb{E}\left(\sup_{t\in [0,T\wedge\tau_M]}\|u\|_{\mathbb{L}^2}^{2p}\right)+C\mathbb{E}\left(\int_{0}^{T\wedge\tau_M}\|u\|_{\mathbb{L}^2}^{2}+\|\nabla \Psi\|_{L^2}^{2}dt\right)^p.
\end{align}
Assumption \eqref{as1} yields
\begin{align}\label{3.11}
\|P_nPf(u, \nabla\Psi)\|_{L_2(\mathcal{H};\mathbb{L}^2(\mathcal{D}))}^2\leq \ell_1(\|u\|_{\mathbb{L}^2}^{2}+\|\nabla \Psi\|_{L^2}^{2}).
\end{align}
Taking integral of time $t$, then supremum over $t\in[0,T\wedge \tau_M]$ and power $p$, finally taking expectation in \eqref{3.9}, combining estimates \eqref{e3.11}-\eqref{3.11} and using the Gronwall lemma, we conclude
\begin{eqnarray}\label{3.13}
&&\mathbb{E}\left[\sup_{t\in [0,T\wedge \tau_M]}\left(\|u\|_{\mathbb{L}^2}^2+\sum_{i=1}^{2}\int_{\mathcal{D}}\mathbf{c}^i{\rm log}\mathbf{c}^idx+\|\nabla\Psi\|_{L^2}^2+\|\Psi\|_{L^2(\partial\mathcal{D})}^2\right)^p\right]\nonumber\\&&+\mathbb{E}\left(\int_{0}^{T\wedge \tau_M}\int_{\mathcal{D}}\mathbf{c}^i|\nabla \theta_i|^2dxdt\right)^p
+\mathbb{E}\left(\int_{0}^{T\wedge \tau_M}\mu\|\nabla u\|_{\mathbb{L}^2}^2dt\right)^p\leq C,
\end{eqnarray}
where the constant $C$ is independent of $n,M$. Letting $M\rightarrow \infty$, by the monotone convergence theorem, we complete the proof.
\end{proof}

\begin{remark} Notice that ${\rm log}\mathbf{c}^i$ has no meaning when $\mathbf{c}^i$ attains the value $0$. To be rigorous, we should work with the quantity ${\rm log}(\mathbf{c}^i+\delta)$, and then pass $\delta\rightarrow 0$.
\end{remark}

\begin{remark} The quantity in the first brackets of \eqref{3.13} is the total free energy, including three parts $\|u\|_{\mathbb{L}^2}^2$, $\int_{\mathcal{D}}\mathbf{c}^i{\rm log}\mathbf{c}^idx$,  $\|\nabla\Psi\|_{L^2}^2+\|\Psi\|_{L^2(\partial\mathcal{D})}^2$ which stand for the kinetic energy, the Gibbs free energy and the electrical energy, respectively.
\end{remark}
By the estimates $\nabla \Psi_n\in L^p(\Omega; L^\infty(0,T;L^2))$, $\Psi_n\in L^p(\Omega; L^\infty(0,T;L^2(\partial\mathcal{D})))$ and the generalized Poincar\'{e} inequality: $\|f\|_{L^2(D)}\le C(\|\nabla f\|_{L^2(D)}+\|f\|_{L^2(\partial D)})$, we have
\begin{align}\label{e1.21}
 \Psi_n\in L^p(\Omega; L^\infty(0,T;H^1)).
\end{align}

\begin{lemma}\label{lemma3.2} For the approximate solutions sequence $\mathbf{c}_n^i$, it holds for all $p\in [1,\infty), T>0$
\begin{align}\label{e3.13}
\mathbb{E}\left(\sup_{t\in [0,T]}\|\mathbf{c}_n^i\|_{L^2}^{2p}\right)+\mathbb{E}\left(\int_{0}^{T}a_i\|\nabla \mathbf{c}_n^i\|^2_{L^2}dt\right)^p\leq C,
\end{align}
where the constant $C$ is independent of $n$.
\end{lemma}
\begin{proof}
Taking inner product with $|z_i|\mathbf{c}^i$ in equation $\eqref{Equ3.1}_1$, using the Robin boundary condition \eqref{e1.6}, we have
\begin{align*}
\frac{|z_i|}{2}d\|\mathbf{c}^i\|_{L^2}^2+a_i|z_i|\|\nabla \mathbf{c}^i\|_{L^2}^2dt&=-z_i|z_i|(\mathbf{c}^i\nabla \mathbf{c}^i, \nabla\Psi)dt=\frac{-z_i|z_i|}{2}(\nabla (\mathbf{c}^i)^2, \nabla\Psi)dt\\
&=\frac{-z_i|z_i|}{2}((\mathbf{c}^i)^2, \rho)dt-\frac{z_i|z_i|}{2}\int_{\partial\mathcal{D}}(\mathbf{c}^i)^2\partial_n\Psi d\mathcal{S}dt\\
&=\frac{-z_i|z_i|}{2}((\mathbf{c}^i)^2, \rho) dt-\frac{z_i|z_i|}{2}\int_{\partial\mathcal{D}}(\mathbf{c}^i)^2(\eta-\Psi) d\mathcal{S}dt.
\end{align*}
Choosing $p=\frac{8}{3}$ in Lemma \ref{lem2.1*}, then by the Gagliardo-Nirenberg inequality, we get when $d=3$
\begin{align}\label{3.21*}
\|\mathbf{c}^i\|_{L^\frac{8}{3}(\partial\mathcal{D})}&\leq C\|\nabla \mathbf{c}^i\|_{L^2}^\frac{3}{8}\|\mathbf{c}^i\|_{L^\frac{10}{3}}^\frac{5}{8}+C\|\mathbf{c}^i\|_{L^\frac{8}{3}}\nonumber\\
&\leq C\|\nabla \mathbf{c}^i\|_{L^2}^\frac{3}{8}\|\nabla \mathbf{c}^i\|_{L^2}^\frac{3}{8}\|\mathbf{c}^i\|_{L^2}^\frac{1}{4}+C\|\mathbf{c}^i\|_{L^1}^\frac{1}{4}\|\nabla\mathbf{c}^i\|_{L^2}^\frac{3}{4}\nonumber\\
&\leq C\|\nabla \mathbf{c}^i\|_{L^2}^\frac{3}{4}\|\mathbf{c}^i\|_{L^2}^\frac{1}{4}+C\|\mathbf{c}^i\|_{L^1}^\frac{1}{4}\|\nabla\mathbf{c}^i\|_{L^2}^\frac{3}{4},
\end{align}
then from \eqref{3.21*}, we bound
\begin{align*}
&\left|-\frac{z_i|z_i|}{2}\int_{\partial\mathcal{D}}(\mathbf{c}^i)^2(\eta-\Psi) d\mathcal{S}\right|\nonumber\\&\leq C\|\eta\|_{L^\infty(\partial\mathcal{D})}\|\mathbf{c}^i\|_{L^2(\partial\mathcal{D})}^2+\|\Psi\|_{L^4(\partial\mathcal{D})}\|\mathbf{c}^i\|_{L^\frac{8}{3}(\partial\mathcal{D})}\\
&\leq C(\|\eta\|_{L^\infty(\partial\mathcal{D})}+\|\Psi\|_{H^1})(\|\nabla \mathbf{c}^i\|_{L^2}^\frac{3}{4}\|\mathbf{c}^i\|_{L^2}^\frac{1}{4}+\|\mathbf{c}^i\|_{L^1}^\frac{1}{4}\|\nabla\mathbf{c}^i\|_{L^2}^\frac{3}{4})\\
&\leq C\varepsilon_1\|\nabla \mathbf{c}^i\|_{L^2}^2+C(\varepsilon_1)(\|\eta\|_{L^\infty(\partial\mathcal{D})}+\|\Psi\|_{H^1})^\frac{8}{5}(\|\mathbf{c}^i\|_{L^2}^\frac{2}{5}
+\|\mathbf{c}^i\|_{L^1}^\frac{2}{5})\nonumber\\
&\leq C\varepsilon_1\|\nabla \mathbf{c}^i\|_{L^2}^2+C(\varepsilon_1)(\|\eta\|^2_{L^\infty(\partial\mathcal{D})}+\|\Psi\|^2_{H^1})+C(\varepsilon_1)(\|\mathbf{c}^i\|_{L^2}^2
+\|\mathbf{c}^i\|_{L^1}^2).
\end{align*}
Choose $\varepsilon_1$ small enough such that
$$C\varepsilon_1<\frac{1}{2}a_i|z_i|.$$
Taking integral of time $t$, then supremum over $t\in[0,T]$ and power $p$,  taking expectation, finally taking sum of $i$, using assumption $z_1>0>z_2$, \eqref{e1.21} as well as \eqref{e1.8}, we have
\begin{align*}\label{e3.17}
&\mathbb{E}\left(\sup_{t\in [0,T]}\|\mathbf{c}^i\|_{L^2}^{2p}\right)+a_i|z_i|\mathbb{E}\left(\int_{0}^{T}\|\nabla \mathbf{c}^i\|_{L^2}^2dt\right)^p\nonumber\\&\leq \left(C+C\mathbb{E}\int_{0}^{T}\|\mathbf{c}^i\|_{L^2}^2dt-\frac{1}{2}\mathbb{E}\int_{0}^{T}\int_{\mathcal{D}}\rho (z^2_1(\mathbf{c}^1)^2-z^2_2(\mathbf{c}^2)^2)dxdt\right)^p\nonumber\\
&= \left(C+C\mathbb{E}\int_{0}^{T}\|\mathbf{c}^i\|_{L^2}^2dt-\frac{1}{2}\mathbb{E}\int_{0}^{T}\int_{\mathcal{D}}\rho^2 (|z_1|\mathbf{c}^1+|z_2|\mathbf{c}^2)dxdt\right)^p\nonumber\\
&\leq \left(C+C\mathbb{E}\int_{0}^{T}\|\mathbf{c}^i\|_{L^2}^2dt\right)^p.
\end{align*}
Then, using the Gronwall lemma, we conclude for $d=3$
\begin{align}
\mathbb{E}\left(\sup_{t\in [0,T]}\|\mathbf{c}^i\|_{L^2}^{2p}\right)+\mathbb{E}\left(\int_{0}^{T}\|\nabla \mathbf{c}^i\|_{L^2}^2dt\right)^p\leq C,
\end{align}
where $C$ is a constant independence of $n$. By a same argument, we have that estimate \eqref{e3.13} holds for $d=2$, completing the proof.
\end{proof}
Using the elliptic regularity theory, utilizing \eqref{e1.3}, \eqref{e3.13} and the assumption on $\eta$, we have the estimate for $\Psi_n$:
\begin{align*}
\nabla\Psi_n \in L^p(\Omega; L^\infty(0,T;H^1)),
\end{align*}
and
\begin{align*}
\nabla\Psi_n \in L^p(\Omega; L^2(0,T;H^2)),
\end{align*}
which together with bound \eqref{e1.21} imply
\begin{align}\label{e3.21}
\Psi_n\in L^p(\Omega; L^\infty(0,T;H^2)\cap L^2(0,T;H^3)).
\end{align}

We next improve the time regularity of $(u_n, \mathbf{c}_n^i)$.
\begin{lemma}\label{lem3.3} The sequence of approximate solutions $(u_n, \mathbf{c}_n^i)$ satisfies
\begin{align*}
&\mathbb{E}\Big(\|u_n\|^p_{C^\alpha([0,T];W^{-1,\frac{3}{2}})}\Big)\leq C,
\\
&\mathbb{E}\Big(\|\mathbf{\mathbf{c}}^i_n\|^p_{C^\alpha([0,T];H^{-1})}\Big)\leq C,
\\
&\mathbb{E}\Big(\|\Psi_n\|^p_{C^\alpha([0,T];H^{1})}\Big)\leq C,
\end{align*}
for $\alpha>0$ small, where the constant $C$ is independent of $n$.
\end{lemma}

\begin{proof} We still use $(u, \mathbf{c}^i)$ instead of $(u_n, \mathbf{c}_n^i)$ for brevity.
Note that for a.s. $\omega$, and for any $\epsilon>0$, there exist $t_1,t_2\in [0,T]$ such that
$$\sup_{t_0\neq t'_0}\frac{\left\|\int_{t_0}^{t'_0}P(u\cdot \nabla u)ds\right\|_{W^{-1,\frac{3}{2}}}}{|t'_0-t_0|^\alpha}\leq \frac{\left\|\int_{t_1}^{t_2}P(u\cdot \nabla u)ds\right\|_{W^{-1,\frac{3}{2}}}}{|t_2-t_1|^\alpha}+\epsilon.$$
Taking $t=\frac{3}{2}, s=2, r=2$ in Lemma \ref{lem6.1} and using H\"{o}lder's inequality, for $\alpha\leq 1$,we get
\begin{align*}
&\mathbb{E}\!\left(\sup_{t,t'\in [0,T]}\frac{\left\|\int_{t'}^{t}P(u\cdot \nabla u)ds\right\|_{W^{-1,\frac{3}{2}}}}{|t-t'|^\alpha}\!\right)\\
&\leq \mathbb{E}\left(\frac{\left\|\int_{t_1}^{t_2}P(u\cdot \nabla u)ds\right\|_{W^{-1,\frac{3}{2}}}}{|t_2-t_1|^\alpha}\right)+\epsilon\\
&\leq \mathbb{E}\left(\frac{\int_{t_1}^{t_2}\|u\|_{\mathbb{H}^1} \|\nabla u\|_{\mathbb{H}^{-1}}ds}{|t_2-t_1|^\alpha}\right)+\epsilon\\
&\leq \frac{\left(\mathbb{E}\sup\limits_{t\in [0,T]}\| u\|_{\mathbb{L}^2}^2\right)^\frac{1}{2}\left(\mathbb{E}\left[\int_{t_1}^{t_2}\|u\|_{\mathbb{H}^{1}}ds\right]^2\right)^\frac{1}{2}}{|t_2-t_1|^\alpha}+\epsilon\\
&\leq\! |t_2-t_1|^{1-\alpha}\!\!\left[\mathbb{E}\left(\sup_{t\in [0,T]}\| u\|_{\mathbb{L}^2}^2\right)\right]^\frac{1}{2}\!\!\left(\mathbb{E}\int_{t_1}^{t_2}\| u\|^2_{\mathbb{H}^{1}}ds\right)^\frac{1}{2}\!\!+\epsilon\\
&\leq C.
\end{align*}
Similarly,
\begin{align*}
&\mathbb{E}\!\left(\sup_{t,t'\in [0,T]}\frac{\left\|\int_{t'}^{t}P(\rho\cdot \nabla \Psi)ds\right\|_{W^{-1,\frac{3}{2}}}}{|t-t'|^\alpha}\!\right)\\
&\leq \mathbb{E}\left(\frac{\int_{t_1}^{t_2}\|\rho\|_{H^1} \|\nabla \Psi\|_{H^{-1}}ds}{|t_2-t_1|^\alpha}\right)+\epsilon\\
&\leq \frac{\left(\mathbb{E}\sup\limits_{t\in [0,T]}\| \Psi\|_{L^2}^2\right)^\frac{1}{2}\left(\mathbb{E}\left[\int_{t_1}^{t_2}\|\rho\|_{\mathrm{H}^{1}}ds\right]^2\right)^\frac{1}{2}}{|t_2-t_1|^\alpha}+\epsilon\\
&\leq\! |t_2-t_1|^{1-\alpha}\!\!\left[\mathbb{E}\left(\sup_{t\in [0,T]}\| \Psi\|_{L^2}^2\right)\right]^\frac{1}{2}\!\!\left(\mathbb{E}\int_{t_1}^{t_2}\| \rho\|^2_{H^{1}}ds\right)^\frac{1}{2}\!\!+\epsilon\\
&\leq C.
\end{align*}
By employing the Burkholder-Davis-Gundy inequality and assumption \eqref{as1}, we obtain
\begin{align*}
\mathbb{E}\left(\sup_{t_0\neq t'_0}\frac{\left|\int_{t_0}^{t'_0}Pf(u,\nabla\Psi)d\mathcal{W}\right|}{|t'_0-t_0|^\alpha}\right)^p
\leq &~C\mathbb{E}\frac{\left(\int_{t_1}^{t_2}\|f(u,\nabla\Psi)\|_{L_2(\mathcal{H};L^2)}^2
dt\right)^\frac{p}{2}}{|t_2-t_1|^{\alpha p}}+\epsilon^p\\
\leq &~\frac{C\ell_1(t_2-t_1)^\frac{p}{2}\mathbb{E}(1+\|u\|_{L_t^\infty \mathbb{L}^2}+\|\nabla\Psi\|_{L_t^\infty L^2})^p}{|t_1-t_2|^{\alpha p}}+\epsilon^p\\
\leq &~C|t_2-t_1|^{\left(\frac{1}{2}-\alpha\right)p}+\epsilon^p\leq C,
\end{align*}
for any $\alpha<\frac{1}{2}$.

{Since $\mathbf{c}^i$ has the same regularity with $u$ and the convection term has same nonlinear construction, for $\alpha\leq 1$, one has
\begin{align*}
&\mathbb{E}\!\left(\sup_{t,t'\in [0,T]}\frac{\left\|\int_{t'}^{t}u\cdot \nabla \mathbf{c}^ids\right\|_{W^{-1,\frac{3}{2}}}}{|t-t'|^\alpha}\!\right)\\
&\leq\! |t_2-t_1|^{1-\alpha}\!\!\left[\mathbb{E}\left(\sup_{t\in [0,T]}\|\mathbf{c}^i\|_{L^2}^2\right)\right]^\frac{1}{2}\!\!\left(\mathbb{E}\int_{t_1}^{t_2}\| u\|^2_{\mathbb{H}^{1}}ds\right)^\frac{1}{2}\!\!+\epsilon\\
&\leq C.
\end{align*}}
Using the H\"{o}lder inequality and bounds \eqref{e3.13}, \eqref{e3.21}, we also have
\begin{align*}
&\mathbb{E}\!\left(\sup_{t,t'\in [0,T]}\frac{\left\|\int_{t'}^{t}{\rm div}(\mathbf{c}^i \nabla \Psi)ds\right\|_{H^{-1}}}{|t-t'|^\alpha}\!\right)\\
&\leq \mathbb{E}\left(\frac{\int_{t_1}^{t_2}\|\mathbf{c}^i\|_{L^3} \|\nabla \Psi\|_{L^6}ds}{|t_2-t_1|^\alpha}\right)+\epsilon\\
&\leq\! |t_2-t_1|^{1-\alpha}\!\!\left[\mathbb{E}\left(\sup_{t\in [0,T]}\| \Psi\|_{H^2}^2\right)\right]^\frac{1}{2}\!\!\left(\mathbb{E}\int_{t_1}^{t_2}\| \mathbf{c}^i\|^2_{H^{1}}ds\right)^\frac{1}{2}\!\!+\epsilon\\
&\leq C.
\end{align*}
We complete the proof following all estimates and the equation itself.
\end{proof}

\subsection{Stochastic compactness argument}
Now, we are in position to show the stochastic compactness of the approximate solutions sequence, which relies on the following Aubin-Lions lemma (see \cite[Corollary 5]{Simon}) and the Skorokhod-Jakubowski theorem (see \cite[Theorem 1]{jak}).

\begin{lemma}[The Aubin-Lions lemma]\label{lem4.5} Suppose that $X_{1}\subset X_{0}\subset X_{2}$ are Banach spaces, $X_{1}$ and $X_{2}$ are reflexive satisfying the embedding of $X_{1}$ into $X_{0}$ is compact.
Then for any $p\in [1, \infty),~ \alpha\in (0, 1)$, the embedding
\begin{eqnarray*}
L^{p}(0,T;X_{1})\cap C^{\alpha}([0,T];X_{2}) \hookrightarrow L^{p}(0,T;X_{0}),
\end{eqnarray*}
is compact.
\end{lemma}

\begin{theorem}[The Skorokhod-Jakubowski theorem]\label{thm6.2} Let $X$ be a quasi-Polish space. If the set of probability measures $\{\nu_n\}_{n\geq 1}$ on $\mathcal{B}(X)$ is tight, then there exist a probability space $(\Omega, \mathcal{F}, \mathbb{P})$ and a sequence of random variables $u_n, u$ such that theirs laws are $\nu_n$, $\nu$ and $u_n\rightarrow u$, $\mathbb{P}$-a.s. as $n\rightarrow \infty$.
\end{theorem}

Our next aim is to construct a measure set, and then verify its tightness, then the Skorokhod-Jakubowski theorem could be applied.

 Define the probability measure
$$\pounds=\pounds^u\ast\pounds^{\mathbf{c}^i}\ast\pounds^\Psi\ast \pounds^\mathcal{W},$$
where $\pounds^u$ is the law of $u$ on the path space
$$\mathbb{X}_u:=L_{w}^{2}(0,T;\mathbb{H}^{1}(\mathcal{D}))\cap L^2(0,T; \mathbb{L}^2(\mathcal{D})),$$
$\pounds^{\mathbf{c}^i}$ is the law of $\mathbf{c}^i$ on the path space
$$\mathbb{X}_{\mathbf{c}^i}:=L_{w}^{2}(0,T;H^{1}(\mathcal{D}))\cap L^2(0,T; L^2(\mathcal{D})),$$
$\pounds^{\Psi}$ is the law of $\Psi$ on the path space
$$\mathbb{X}_{\Psi}:=L_{w}^{2}(0,T;H^{3}(\mathcal{D}))\cap L^2(0,T; H^2(\mathcal{D}))\cap C([0,T];H^1(\mathcal{D})),$$
$\pounds^{\mathcal{W}}$ is the law of $\mathcal{W}$ on the path space
$$\mathbb{X}_{\mathcal{W}}:=C([0,T];\mathcal{H}_0).$$

\begin{lemma} \label{lem3.4} The set of probability measures $\{\pounds^u_{n}\}_{n\geq 1}$ is tight on the path space $\mathbb{X}_u$.
\end{lemma}
\begin{proof} The Aubin-Lions lemma \ref{lem4.5} yields that for any $K>0$, the set
$$B^1_K:=\left\{u: \|u\|_{L^2(0,T;\mathbb{H}^1)}+\|u\|_{C^\alpha(0,T;W^{-1,\frac{3}{2}})}\leq K\right\}$$
is compactly embedded into $L^2(0,T;\mathbb{L}^2)$. The Banach-Alaoglu theorem yields that
$$B^2_K:=\left\{u: \|u\|_{L^2(0,T;\mathbb{H}^1)}\leq K\right\}$$
is relatively compact in $L_w^2(0,T;\mathbb{H}^1)$. Therefore,
\begin{align*}
\pounds_n^{u}\big\{(B^1_K\cap B^2_K)^c\big\}&\leq \pounds_n^{u}\big\{(B^1_K)^c\big\}+\pounds_n^{u}\big\{(B^2_K)^c\big\}\\
&\leq \frac{C}{K}\mathbb{E}\left(\|u\|_{L^2(0,T;\mathbb{H}^1)}+\|u\|_{C^\alpha(0,T;W^{-1,\frac{3}{2}})}\right)\\
&\leq \frac{C}{K}.
\end{align*}
In the second line, we used the Chebyshev inequality,  in the third line, we used Lemmas \ref{lem3.1} and \ref{lem3.3}. Since the constant $C$ is independent of $n$, the desired tightness follows.
\end{proof}

Analogous to Lemma \ref{lem3.4}, we could show that  $\{\pounds^{\mathbf{c}^i}_{n}\}_{n\geq 1},  \{\pounds^\Psi_{n}\}_{n\geq 1}$ are tight on the path space $\mathbb{X}_{\mathbf{c}^i}, \mathbb{X}_\Psi$, respectively. In particular, the sequence $\mathcal{W}$ is only one element, then the set $\pounds^\mathcal{W}$ is weakly compact. We conclude that $\{\pounds_n\}_{n\geq 1}$ is tight on path space $\mathbb{X}$, where
$$\mathbb{X}:=\mathbb{X}_u\times\mathbb{X}_{\mathbf{c}^i}\times\mathbb{X}_{\Psi}\times\mathbb{X}_{\mathcal{W}}.$$

According to Theorem \ref{thm6.2}, we have the following proposition.
\begin{proposition}\label{pro3.1} There exist a subsequence $\{\pounds_{n_{k}}\}_{k\geq 1}$, a new probability space $(\tilde{\Omega},\tilde{\mathcal{F}},\tilde{\mathbb{P}})$, a new $\mathbb{X}$-valued measurable random variables sequence $(\tilde{u}_{n_{k}},\tilde{\mathbf{c}}^i_{n_{k}}, \widetilde{\Psi}_{n_{k}}, \widetilde{\mathcal{W}}_{n_{k}})$ and $(\tilde{u},\tilde{\mathbf{c}}^i, \widetilde{\Psi}, \widetilde{\mathcal{W}})$ such that

{\rm i}. $(\tilde{u}_{n_{k}}, \tilde{\mathbf{c}}^i_{n_{k}}, \widetilde{\Psi}_{n_{k}}, \widetilde{\mathcal{W}}_{n_{k}})\rightarrow(\tilde{u},\tilde{\mathbf{c}}^i, \widetilde{\Psi}, \widetilde{\mathcal{W}})$ $\tilde{\mathbb{P}}$-a.s. in the topology of $\mathbb{X}$,

{\rm ii}. the laws of $(\tilde{u}_{n_{k}},\tilde{\mathbf{c}}^i_{n_{k}}, \widetilde{\Psi}_{n_{k}}, \widetilde{\mathcal{W}}_{n_{k}})$ and $(\tilde{u},\tilde{\mathbf{c}}^i, \widetilde{\Psi}, \widetilde{\mathcal{W}})$ are given by $\{\pounds_{n_{k}}\}_{k\geq 1}$ and $\mu$, respectively,

{\rm iii}.  $\widetilde{\mathcal{\mathcal{W}}}_{n_{k}}$ is a Wiener process, relative to the filtration $\tilde{\mathcal{F}}_{t}^{n_{k}}=\sigma(\tilde{u}_{n_{k}},\tilde{\mathbf{c}}_{n_k}^i,\widetilde{\mathcal{W}}_{n_{k}})$.
\end{proposition}

In the following, we still use the sequence $(\tilde{u}_{n}, \tilde{\mathbf{c}}^i_{n}, \widetilde{\Psi}_{n}, \widetilde{\mathcal{W}}_{n})$ representing new subsequence.
As a result of Proposition \ref{pro3.1}, the sequence $\tilde{u}_n, \tilde{\mathbf{c}}^i_n, \widetilde{\Psi}_n$ also shares the following bounds: for any $p\in [1,\infty)$
\begin{align}
&\widetilde{\mathbb{E}}\left(\sup_{t\in [0,T]}\|\tilde{\mathbf{c}}_n^i\|_{L^2}^{2p}\right)+\widetilde{\mathbb{E}}\left(\int_{0}^{T}a_i\|\nabla \tilde{\mathbf{c}}^i_n\|^2_{L^2}dt\right)^p\leq C, \label{e3.29}\\
&\widetilde{\mathbb{E}}\left(\sup_{t\in [0,T]}\|\tilde{u}_n\|_{\mathbb{L}^2}^{2p}\right)+\widetilde{\mathbb{E}}\left(\int_{0}^{T}\mu\|\nabla \tilde{u}_n\|^2_{\mathbb{L}^2}dt\right)^p\leq C, \label{e3.30}\\
&\widetilde{\mathbb{E}}\left(\sup_{t\in [0,T]}\|\widetilde{\Psi}_n\|_{H^2}^{2p}\right)+\widetilde{\mathbb{E}}\left(\int_{0}^{T}\|\widetilde{\Psi}_n\|^2_{H^3}dt\right)^p\leq C,\label{e.33}
\end{align}
where constant $C$ is independent of $n$.

\subsection{Identify the limit by passing $n\rightarrow\infty$}
In the following, the H\"{o}lder inequality will be frequently applied to estimate all terms, we do not mention it for brevity.
{Choose $\phi\in H^1$, and decompose}
\begin{align}\label{e3.32}
&\int_{0}^{t}\left(a_i\left(\nabla \tilde{\mathbf{c}}^i_n+z_i\tilde{\mathbf{c}}^i_n\nabla \widetilde{\Psi}_n-\nabla \tilde{\mathbf{c}}^i- z_i\tilde{\mathbf{c}}^i\nabla \widetilde{\Psi}\right), \nabla\phi\right)ds\nonumber\\
&=\int_{0}^{t}\left(a_i\left(\nabla \tilde{\mathbf{c}}^i_n-\nabla \tilde{\mathbf{c}}^i+z_i(\tilde{\mathbf{c}}^i_n-\tilde{\mathbf{c}}^i)\nabla \widetilde{\Psi}_n\right), \nabla\phi\right)ds\nonumber\\
&\quad+\int_{0}^{t}\left(a_i\left(\nabla \tilde{\mathbf{c}}^i_n-\nabla \tilde{\mathbf{c}}^i_n+z_i\tilde{\mathbf{c}}^i(\nabla \widetilde{\Psi}_n-\nabla \widetilde{\Psi})\right), \nabla\phi\right)ds\nonumber\\
&=:I_1+I_2.
\end{align}
Next, we show that each term $I_i, i=1,2$ goes to zero as $n\rightarrow\infty$, $\tilde{\mathbb{P}}$-a.s. By the Sobolev embedding $H^2\hookrightarrow L^\infty$, we have
\begin{align*}
|I_1|&\leq \int_{0}^{t}(a_i\nabla \tilde{\mathbf{c}}^i_n-a_i\nabla \tilde{\mathbf{c}}^i, \nabla \phi)ds+C\|\nabla\phi\|_{L^2}\|\tilde{\mathbf{c}}_n^i-\tilde{\mathbf{c}}^i\|_{L^2(0,T;L^2)}\|\nabla \widetilde{\Psi}\|_{L^2(0,T;L^\infty)}\\
&\leq \int_{0}^{t}(a_i\nabla \tilde{\mathbf{c}}^i_n-a_i\nabla \tilde{\mathbf{c}}^i, \nabla \phi)ds+C\|\nabla\phi\|_{L^2}\|\tilde{\mathbf{c}}_n^i-\tilde{\mathbf{c}}^i\|_{L^2(0,T;L^2)}\|\nabla \widetilde{\Psi}\|_{L^2(0,T;H^2)},
\end{align*}
and then using the convergence that $\tilde{\mathbf{c}}^i_n\rightarrow \tilde{\mathbf{c}}^i$ in $L^2(0,T; L^2)$ and weakly in $L^2(0,T;H^1)$, $\tilde{\mathbb{P}}$-a.s. and bound \eqref{e.33}, we deduce
\begin{align}
|I_1|\rightarrow 0, ~{\rm as}~ n\rightarrow \infty, ~ \tilde{\mathbb{P}}\mbox{-a.s.}
\end{align}
We use the Sobolev embedding $H^1\hookrightarrow L^6$ to find
\begin{align*}
|I_2|&\leq \int_{0}^{t}(a_i\nabla \tilde{\mathbf{c}}^i_n-a_i\nabla \tilde{\mathbf{c}}^i, \nabla \phi)ds+C\|\nabla\phi\|_{L^2}\|\nabla \widetilde{\Psi}_n-\nabla\widetilde{\Psi}\|_{L^2(0,T; L^6)}\|\mathbf{c}^i\|_{L^2(0,T;L^3)}\\
&\leq  \int_{0}^{t}(a_i\nabla \tilde{\mathbf{c}}^i_n-a_i\nabla \tilde{\mathbf{c}}^i, \nabla \phi)ds+C\|\nabla\phi\|_{L^2}\|\nabla \widetilde{\Psi}_n-\nabla\widetilde{\Psi}\|_{L^2(0,T; H^1)}\|\mathbf{c}^i\|_{L^2(0,T;H^1)},
\end{align*}
then the convergence $\widetilde{\Psi}_n\rightarrow \widetilde{\Psi}$ in $L^2(0,T; H^2)$, $\tilde{\mathbf{c}}^i_n\rightarrow \tilde{\mathbf{c}}^i$ weakly in $L^2(0,T;H^1)$ $\tilde{\mathbb{P}}$-a.s. and bound \eqref{e3.29} yield
\begin{align}\label{e3.34}
|I_2|\rightarrow 0, ~{\rm as}~ n\rightarrow \infty, ~ \tilde{\mathbb{P}}\mbox{-a.s.}
\end{align}
Therefore, from \eqref{e3.32}-\eqref{e3.34}, it holds $\tilde{\mathbb{P}}$-a.s.
$$\lim_{n\rightarrow\infty}\int_{0}^{t}\left(a_i\left(\nabla \tilde{\mathbf{c}}^i_n+z_i\tilde{\mathbf{c}}^i_n\nabla \widetilde{\Psi}_n\right), \nabla\phi\right)ds=\int_{0}^{t}\left(a_i\left(\nabla \tilde{\mathbf{c}}^i+z_i\tilde{\mathbf{c}}^i\nabla \widetilde{\Psi}\right), \nabla\phi\right)ds.$$

Decompose
\begin{align}\label{e3.35}
\int_{0}^{t}(\tilde{u}_n\cdot\nabla \mathbf{c}_n^i-\tilde{u}\cdot\nabla \tilde{\mathbf{c}}^i, \phi)ds&=\int_{0}^{t}((\tilde{u}_n-\tilde{u})\cdot\nabla \tilde{\mathbf{c}}_n^i, \phi)ds
+\int_{0}^{t}(\tilde{u}\cdot\nabla (\tilde{\mathbf{c}}_n^i-\tilde{\mathbf{c}}^i), \phi)ds.
\end{align}
For the first term on the right-hand side of \eqref{e3.35}, we first obtain
\begin{align*}
\left|\int_{0}^{t}((\tilde{u}_n-\tilde{u})\cdot\nabla \tilde{\mathbf{c}}_n^i, \phi)ds\right|&\leq C\|\phi\|_{H^1}\|\tilde{u}_n-\tilde{u}\|^\frac{1}{2}_{L^2(0,T;\mathbb{L}^2)}\|\nabla\tilde{u}_n, \nabla\tilde{u}\|^\frac{1}{2}_{L^2(0,T;\mathbb{L}^2)}\|\nabla \tilde{\mathbf{c}}_n^i\|_{L^2(0,T;L^2)},
\end{align*}
using the convergence $\tilde{u}_n\rightarrow \tilde{u}$ in $L^2(0,T;\mathbb{L}^2)$, $\tilde{\mathbb{P}}$-a.s. and bound \eqref{e3.30}, we deduce
\begin{align}
\left|\int_{0}^{t}((\tilde{u}_n-\tilde{u})\cdot\nabla \tilde{\mathbf{c}}_n^i, \phi)ds\right|\rightarrow 0, ~{\rm as}~ n\rightarrow \infty, ~ \tilde{\mathbb{P}}\mbox{-a.s.}
\end{align}
Since $\tilde{\mathbf{c}}_n^i\rightharpoonup \tilde{\mathbf{c}}^i$ in $L^2(0,T;H^1)$, we get
\begin{align}\label{e3.37}
\int_{0}^{t}(\tilde{u}\cdot\nabla (\tilde{\mathbf{c}}_n^i-\tilde{\mathbf{c}}^i), \phi)ds\rightarrow 0,  ~{\rm as}~ n\rightarrow \infty, ~ \tilde{\mathbb{P}}\mbox{-a.s.}
\end{align}
From \eqref{e3.35}-\eqref{e3.37}, we have $\tilde{\mathbb{P}}$\mbox{-a.s.}
$$\lim_{n\rightarrow\infty}\int_{0}^{t}(\tilde{u}_n\cdot\nabla \tilde{\mathbf{c}}_n^i-\tilde{u}\cdot\nabla \tilde{\mathbf{c}}^i, \phi)ds=0.$$

Next, we pass the limit in the velocity equation. Since $\tilde{u}_n$ and $\tilde{\mathbf{c}}_n^i$ have same regularity and the convection terms have same construction, we only focus on the term $\kappa\tilde{\rho}\nabla\widetilde{\Psi}$ and stochastic term. {For any $\varphi\in \mathbb{H}^1$, decompose}
\begin{align}\label{e3.38}
\int_{0}^{t}(\kappa\tilde{\rho}_n\nabla\widetilde{\Psi}_n-\kappa\tilde{\rho}\nabla\widetilde{\Psi}, \varphi)ds=\int_{0}^{t}(\kappa(\tilde{\rho}_n-\tilde{\rho})\nabla\widetilde{\Psi}_n, \varphi)ds
+\int_{0}^{t}(\kappa\tilde{\rho}\nabla(\widetilde{\Psi}_n-\widetilde{\Psi}), \varphi)ds.
\end{align}
Using the fact that $\tilde{\rho}_n\rightarrow \tilde{\rho}$ in $L^2(0,T;L^2)$, we obtain $\tilde{\mathbb{P}}$-a.s.
\begin{align}
\left|\int_{0}^{t}(\kappa(\tilde{\rho}_n-\tilde{\rho})\nabla\widetilde{\Psi}_n, \varphi)ds\right|\leq C\|\varphi\|_{\mathbb{H}^1}\|\tilde{\rho}_n-\tilde{\rho}\|_{L^2(0,T;L^2)}\|\nabla\widetilde{\Psi}_n\|_{L^\infty(0,T;L^3)}\rightarrow 0.
\end{align}
Similar to \eqref{e3.37}, using the weak convergence of $\widetilde{\Psi}_n$, we have $\tilde{\mathbb{P}}$-a.s.
\begin{align}\label{e3.40}
\int_{0}^{t}(\kappa\tilde{\rho}\nabla(\widetilde{\Psi}_n-\widetilde{\Psi}), \varphi)ds\rightarrow 0.
\end{align}
From \eqref{e3.38}-\eqref{e3.40}, we arrive at
\begin{align}
\lim_{n\rightarrow\infty}\int_{0}^{t}(\kappa\tilde{\rho}_n\nabla\widetilde{\Psi}_n-\kappa\tilde{\rho}\nabla\widetilde{\Psi}, \varphi)ds=0, ~\tilde{\mathbb{P}}\mbox{-a.s.}
\end{align}

For the stochastic term, we are going to verify the both conditions in Lemma \ref{lem4.6} for passing the limit. First, Proposition \ref{pro3.1} gives
$$\widetilde{\mathcal{W}}_{n}\rightarrow \widetilde{\mathcal{W}}~ {\rm in~ probability~ in} ~C([0,T];\mathcal{H}_0).$$
Using the continuity of $f$ and the strong convergence of $\tilde{u}_n, \tilde{\mathbf{c}}^i$ in $L^2(0,T;L^2)$, we have for $(w,t)\in \tilde{\Omega}\times [0,T]$, a.e.
\begin{align}
\|f(\tilde{u}_n, \nabla\widetilde{\Psi}_n)-f(\tilde{u}, \nabla\widetilde{\Psi})\|_{L_2(\mathcal{H};L^2)}\rightarrow 0.
\end{align}
From bounds \eqref{e3.30}-\eqref{e.33}, we infer that $\|f(\tilde{u}_n, \nabla\widetilde{\Psi}_n)\|_{L_2(\mathcal{H};L^2)}$ is uniformly integrable. The Vitali convergence theorem could be applied to obtain
\begin{align}
\|f(\tilde{u}_n, \nabla\widetilde{\Psi}_n)-f(\tilde{u}, \nabla\widetilde{\Psi})\|_{L^p(\tilde{\Omega}; L^2(0,T;L_2(\mathcal{H};L^2)))}\rightarrow 0.
\end{align}
Then, by Lemma \ref{lem4.6}, we have
$$\int_{0}^{t}f(\tilde{u}_{n}, \nabla\widetilde{\Psi}_{n})d \widetilde{\mathcal{W}}_{n}\rightarrow \int_{0}^{t}f(\tilde{u}, \nabla\widetilde{\Psi})d \widetilde{\mathcal{W}},$$
in probability in $L^2(0,T;L_2(\mathcal{H};L^2))$.

 Next, we claim that the sets
\begin{align}\label{e3.44}
 \left\{\int_{0}^{t}(-\tilde{u}_n\cdot\nabla\tilde{u}_n
 +\mu\Delta \tilde{u}_n-\kappa\tilde{\rho}_n\nabla\widetilde{\Psi}_n, \varphi)ds\right\}_{n\geq 1}
\end{align}
and
\begin{align}\label{e3.45}
 \left\{\int_{0}^{t}(a_i\Delta \tilde{\mathbf{c}}^i_n-\tilde{u}_n\cdot\nabla\tilde{\mathbf{c}}^i_n+a_i{\rm div}(z_i\tilde{\mathbf{c}}_n^i\nabla\widetilde{\Psi}_n), \phi)ds\right\}_{n\geq 1}
\end{align}
are uniformly integrable in $L^p(\tilde{\Omega}; L^2(0,T))$.

{Indeed, using bounds \eqref{e3.29}-\eqref{e.33} and the Gagliardo-Nirenberg inequality
\begin{align}\label{e4.18}
\left\{\begin{array}{ll}
\!\!\!\|v\|_{L^4}\leq C\|v\|_{L^2}^\frac{1}{2} \|\nabla v\|_{L^2}^\frac{1}{2},~ {\rm if~ }~ d=2,\\
\!\!\!\|v\|_{L^3}\leq C\|v\|_{L^2}^\frac{1}{2} \|\nabla v\|_{L^2}^\frac{1}{2},~ {\rm if~ }~ d=3,
\end{array}\right.
\end{align}
we have
\begin{align*}
&\left\|\int_{0}^{t}(-\tilde{u}_n\cdot\nabla\tilde{u}_n, \varphi)ds\right\|_{L^p(\tilde{\Omega}; L^2(0,T))}\leq C(T)\widetilde{\mathbb{E}}\left(\int_{0}^{T}(-\tilde{u}_n\cdot\nabla\tilde{u}_n, \varphi)dt\right)^p\\
&\leq C\|\varphi\|^p_{\mathbb{H}^1}\widetilde{\mathbb{E}}\left(\int_{0}^{T}\|\tilde{u}_n\|^\frac{1}{2}_{\mathbb{L}^2}\|\nabla\tilde{u}_n\|^\frac{3}{2}_{\mathbb{L}^2}dt\right)^p\\
&\leq C\|\varphi\|^p_{\mathbb{H}^1}\|\tilde{u}_n\|_{L^{2p}(\tilde{\Omega}; L^\infty(0,T;\mathbb{L}^2))}^\frac{p}{2}\|\tilde{u}_n\|_{L^{2p}(\tilde{\Omega}; L^2(0,T;\mathbb{H}^1))}^\frac{3p}{2},
\end{align*}
and
\begin{align*}
&\left\|\int_{0}^{t}(\kappa\tilde{\rho}_n\nabla\widetilde{\Psi}_n, \varphi)ds\right\|_{L^p(\tilde{\Omega}; L^2(0,T))}\\
&\leq C(T)\widetilde{\mathbb{E}}\left(\int_{0}^{T}(\kappa\tilde{\rho}_n\nabla\widetilde{\Psi}_n, \varphi)dt\right)^p\\
&\leq C(T)\|\varphi\|^p_{\mathbb{H}^1}
\widetilde{\mathbb{E}}\left(\int_{0}^{T}\|\tilde{\rho}_n\|_{L^2}^\frac{1}{2}\|\nabla\tilde{\rho}_n\|_{L^2}^\frac{1}{2}\|\nabla\widetilde{\Psi}_n\|_{L^2}dt\right)^p\\
&\leq C\|\varphi\|^p_{\mathbb{H}^1}\|\tilde{\rho}_n\|_{L^{2p}(\tilde{\Omega}; L^\infty(0,T;L^2))}^\frac{p}{2}\|\tilde{\rho}_n\|_{L^{2p}(\tilde{\Omega}; L^2(0,T;H^1))}^\frac{3p}{2},
\end{align*}
and from the blocking boundary condition, we get
\begin{align*}
&\left\|\int_{0}^{t}({\rm div}(a_i\nabla\tilde{\mathbf{c}}_n^i+ z_i\tilde{\mathbf{c}}_n^i\nabla\widetilde{\Psi}_n), \phi)ds\right\|_{L^p(\tilde{\Omega}; L^2(0,T))}\\
&\leq C(T)\widetilde{\mathbb{E}}\left(\int_{0}^{T}({\rm div}(a_i\nabla\tilde{\mathbf{c}}_n^i+z_i\tilde{\mathbf{c}}_n^i\nabla\widetilde{\Psi}_n), \phi)dt\right)^p\\
&\leq C\|\nabla\phi\|^p_{L^2}\widetilde{\mathbb{E}}\left(\int_{0}^{T}\|\tilde{\mathbf{c}}^i_n\|_{L^3}\|\nabla\widetilde{\Psi}_n\|_{L^6}+\|\nabla \tilde{\mathbf{c}}_n^i\|_{L^2}dt\right)^p\\
&\leq C\|\nabla\phi\|^p_{L^2}\|\tilde{\mathbf{c}}^i_n\|_{L^{2p}(\tilde{\Omega}; L^2(0,T;H^1))}^p\|\widetilde{\Psi}_n\|_{L^{2p}(\tilde{\Omega}; L^\infty(0,T;H^2))}^p.
\end{align*}
Similarly,
\begin{align*}
&\left\|\int_{0}^{t}(\tilde{u}_n\cdot\nabla\tilde{\mathbf{c}}^i_n, \phi)ds\right\|_{L^p(\tilde{\Omega}; L^2(0,T))}\nonumber\\ &\leq C\|\phi\|^p_{H^1}\left(\|\tilde{u}_n\|_{L^{2p}(\tilde{\Omega}; L^\infty(0,T;\mathbb{L}^2))}^p+\|\nabla\tilde{u}_n\|_{L^{2p}(\tilde{\Omega}; L^2(0,T;\mathbb{L}^2))}^p\right)\|\tilde{\mathbf{c}}^i_n\|_{L^{2p}(\tilde{\Omega}; L^2(0,T;H^1))}^p.
\end{align*}}
The uniform integrability of linear term $\mu\Delta \tilde{u}_n$ follows from bound \eqref{e3.30} directly.

We have everything in hand to pass the limit. Define the functionals as
\begin{align*}
\mathbf{F}_1(\tilde{u}_{n},\tilde{\mathbf{c}}^i_{n},\widetilde{\mathcal{W}}_{n},\varphi)&=(\tilde{u}_{n}(0), \varphi)+\int_{0}^{t}\mu\langle A_0\tilde{u}_{n}, \varphi\rangle ds-\int_{0}^{t}(\tilde{u}_{n}\cdot\nabla\tilde{u}_{n}, \varphi)ds\\
&\quad-\int_{0}^{t}(\kappa\tilde{\mathbf{\rho}}_n \nabla \tilde{\Psi}_n, \varphi)ds+\left(\int_{0}^{t}f(\tilde{u}_{n}, \nabla\widetilde{\Psi}_{n})d \widetilde{\mathcal{W}}_{n}, \varphi\right),\\
\mathbf{F}_2(\tilde{u}_{n},\tilde{\mathbf{c}}^i_{n},\phi )&=(\tilde{\mathbf{c}}^i_{n}(0), \phi)-\int_{0}^{t}a_i(\nabla\tilde{\mathbf{c}}^i_{n}+z_i\tilde{\mathbf{c}}^i_n\nabla\widetilde{\Psi}_n, \nabla\phi) ds-\int_{0}^{t}(\tilde{u}_{n}\cdot\nabla  \tilde{\mathbf{c}}^i_{n}, \phi)ds,
\end{align*}
and
\begin{align*}
\mathbf{F}_1(\tilde{u},\tilde{\mathbf{c}}^i,\widetilde{\mathcal{W}},\varphi)&=(\tilde{u}(0), \varphi)+\int_{0}^{t}\mu\langle A_0\tilde{u}, \varphi\rangle ds-\int_{0}^{t}(\tilde{u}\cdot\nabla\tilde{u}, \varphi)ds\\
&\quad-\int_{0}^{t}(\kappa\tilde{\mathbf{\rho}} \nabla \tilde{\Psi}, \varphi)ds+\left(\int_{0}^{t}f(\tilde{u}, \nabla\widetilde{\Psi})d \widetilde{\mathcal{W}}, \varphi\right),\\
\mathbf{F}_2(\tilde{u},\tilde{\mathbf{c}}^i,\phi )&=(\tilde{\mathbf{c}}^i(0), \phi)-\int_{0}^{t}a_i(\nabla\tilde{\mathbf{c}}^i+z_i\tilde{\mathbf{c}}^i\nabla\widetilde{\Psi}, \nabla\phi) ds-\int_{0}^{t}(\tilde{u}\cdot\nabla \tilde{\mathbf{c}}^i, \phi)ds.
\end{align*}

Combining all convergence properties and the uniform integrability of \eqref{e3.44}, \eqref{e3.45},  the Vitali convergence theorem implies, when $n\rightarrow\infty$
\begin{align}
&\widetilde{\mathbb{E}}\int_{0}^{T}1_{A}(\mathbf{F}_1( \tilde{u}_{n},\tilde{\mathbf{c}}^i_{n},\widetilde{\mathcal{W}}_{n},\varphi)-\mathbf{F}_1( \tilde{u},\tilde{\mathbf{c}}^i,\widetilde{\mathcal{W}},\varphi))dt\rightarrow 0,\label{e.51}\\
&\widetilde{\mathbb{E}}\int_{0}^{T}1_{A}(\mathbf{F}_2(\tilde{u}_{n},\tilde{\mathbf{c}}^i_{n},\phi )-\mathbf{F}_2(\tilde{u},\tilde{\mathbf{c}}^i,\phi ))dt\rightarrow 0,\label{e.52}
\end{align}
for any set $A\subset [0,T]$, $\varphi\in \mathbb{H}^1, \phi\in H^1$, where $1_{\cdot}$ is the indicator function.

Note that from Proposition \ref{pro3.1}, we could deduce that the sequence $(\tilde{u}_n, \tilde{\mathbf{c}}_n^i)$ also satisfies the approximate system \eqref{Equ3.1}, that is,
\begin{align}
\mathbf{F}_1( \tilde{u}_{n},\tilde{\mathbf{c}}^i_{n},\widetilde{\mathcal{W}}_{n},\varphi)&=(\tilde{u}_n(t) ,\varphi),\\
\mathbf{F}_2(\tilde{u}_{n},\tilde{\mathbf{c}}^i_{n},\phi )&=(\tilde{\mathbf{c}}^i_{n}(t),\phi ), \label{e.54}
\end{align}
for the detailed proof, see \cite{ww}. {Considering \eqref{e.51}-\eqref{e.54} and the weak convergence of $\tilde{u}_n, \tilde{\mathbf{c}}^i_{n}$, for $\varphi\in \mathbb{H}^1, \phi\in H^1$, it holds
\begin{align}
\mathbf{F}_1( \tilde{u},\tilde{\mathbf{c}}^i,\widetilde{\mathcal{W}},\varphi)&=(\tilde{u}(t) ,\varphi),\label{e.55}\\
\mathbf{F}_2(\tilde{u},\tilde{\mathbf{c}}^i,\phi )&=(\tilde{\mathbf{c}}^i(t),\phi ).\label{e.56}
\end{align}
Thus, we establish the existence of global weak martingale solution to system \eqref{Equ1.1}-\eqref{e1.7} in the sense of Definition \ref{def3.1}.}

\section{Local existence and uniqueness of strong pathwise solution in $d=2,3$.}

In this section, we prove the existence of unique strong pathwise solution in $2D$ and $3D$ cases and a blow-up criterion when the capacitance is 0 (${\partial_{\vec{n}}\Psi(x,t)}|_{\partial \mathcal{D}}=\eta$). At the beginning, we give three kinds of conceptions of solution: strong solution both strong in probability and PDEs sense, maximal strong solution, and global solution.
\begin{definition}[strong pathwise solution]\label{def2.1} Let $(\Omega,\mathcal{F},\{\mathcal{F}_t\}_{t\geq 0},\mathbb{P})$ be a fixed probability space, $\mathcal{W}$ be an $\mathcal{F}_t$-cylindrical Wiener process. We say $(u, \mathbf{c}^i, \tau)$ is a strong pathwise solution to system \eqref{Equ1.1}-\eqref{e1.7} if the following hold:

i. $\tau$ $\mathbb{P}$-a.s. is a strictly positive stopping time;

ii. the process $u$ is an $\mathbb{H}^1$-valued $\mathcal{F}_t$-progressively measurable satisfying
$$u(\cdot\wedge \tau)\in L^p(\Omega; C([0,T]; \mathbb{H}^1)\cap L^2(0,T;H^2));$$

iii. the processes $\mathbf{c}^i, i=1,2,\cdots, m$ are $H^1$-valued $\mathcal{F}_t$-progressively measurable satisfying
$$\mathbf{c}^i\geq 0,~\mbox{a.e.} ~{\rm and} ~\mathbf{c}^i(\cdot\wedge \tau)\in L^p(\Omega; C([0,T]; H^1\cap L^j)\cap L^2(0,T;H^2)),$$
 for all $j\geq 1$, and $\Psi$ is an $H^3$-valued $\mathcal{F}_t$-progressively measurable process with
$$\Psi(\cdot\wedge \tau)\in L^p(\Omega; C([0,T]; H^3)\cap L^2(0,T;H^4));$$

iv. it holds $\mathbb{P}$-a.s.
\begin{eqnarray*}
&&\mathbf{c}^i(t\wedge \tau)=\mathbf{c}^i_0-\int_{0}^{t\wedge \tau}u\cdot \nabla \mathbf{c}^ids+\int_{0}^{t\wedge \tau}a_i{\rm div}(\nabla \mathbf{c}^i+z_i\mathbf{c}^i\nabla\Psi)ds,~i=1,2,\cdots, m,\\
&&u(t\wedge \tau)=u_0+\mu\int_{0}^{t\wedge \tau}P\Delta uds-\int_{0}^{t\wedge \tau}P(u\cdot \nabla)uds- \int_{0}^{t\wedge \tau}P\kappa\rho\nabla \Psi ds\nonumber \\ &&\qquad\qquad+\int_{0}^{t\wedge \tau}Pf(u, \nabla\Psi)d\mathcal{W},\\
&&-\Delta \Psi=\sum_{i=1}^m z_i\mathbf{c}^i=\rho, ~\mbox{a.e.}
\end{eqnarray*}
 for all $t>0$.
\end{definition}

\begin{definition}[maximal pathwise solution and global solution]\label{def4.2} $(u, \mathbf{c}^i, \tau_R, \tau)$ is a maximal solution to system \eqref{Equ1.1}-\eqref{e1.7} if

${\rm i}.$ $\tau$ $\mathbb{P}$-a.s. is a strictly positive stopping time;

${\rm ii}.$ $\tau_R$ is an increasing sequence with $\lim_{R\rightarrow \infty}\tau_R=\tau$, $\mathbb{P}$-a.s.;

${\rm iii}.$ for every $R$, $(u, \mathbf{c}^i, \tau_R)$ is a local strong pathwise solution in the sense of Definition \ref{def2.1};

We say $(u, \mathbf{c}^i)$ is global solution if the lifetime $\tau=\infty$, $\mathbb{P}$-a.s.
\end{definition}

We call a stopping time $\tau$ is accessible if there exists a sequence of stopping time $\tau_R$ such that $\tau_R<\tau$, $\mathbb{P}$-a.s. and $\lim_{R\rightarrow \infty}\tau_R=\tau$, $\mathbb{P}$-a.s.

Let us formulate the main results of this section.
\begin{theorem}\label{thm} Assume that the initial data $(u_0, \mathbf{c}_0^i)$ are $\mathcal{F}_0$-measurable random variables satisfying
$$u_0\in L^p(\Omega; \mathbb{H}^1), ~\mathbf{c}_0^i\in L^p(\Omega; H^1\cap L^j)~{\rm and }~ \mathbf{c}_0^i\geq 0, ~\mathbb{P}\mbox{-a.s.}$$
for all $p\in [2,\infty), j\in [1,\infty)$ and the function $\eta$ satisfies $\eta \in H^\frac{5}{2}(\partial\mathcal{D})$ and $\|\eta\|_{L^\infty(\partial\mathcal{D})}\leq \frac{1}{2C}$ for a constant $C$. The noise operator $f$ satisfies assumptions \eqref{as3}, \eqref{as4}. Then, there exists a unique maximal strong pathwise solution to system \eqref{Equ1.1}-\eqref{e1.7} with $\varsigma=0$ in the sense of Definition \ref{def2.1}.
\end{theorem}

According to Definition \ref{def4.2}, we know that on the set $\{\tau<\infty\}$,
$$\limsup_{t\rightarrow \tau}\|(u, \mathbf{c}^i)\|_{\mathbb{H}^1\times H^1}=\infty.$$
Thanks to  $\|(\nabla u, \nabla \Psi)\|_{\mathbb{L}^2\times W^{1,3+}}\leq C\|(u, \mathbf{c}^i)\|_{\mathbb{H}^1\times H^1}$, we may think that the explosion time of $\|(\nabla u, \nabla \Psi)\|_{\mathbb{L}^2\times W^{1,3+}}$ is earlier than $\|(u, \mathbf{c}^i)\|_{\mathbb{H}^1\times H^1}$. However, the following theorem tells us that the explosion time coincides.
\begin{theorem}\label{thm1} $(u, \mathbf{c}^i)$ is the maximal strong pathwise solution with lifetime $\tau$ established in above, then we have $\mathbb{P}$-a.s.
$$1_{\left\{\limsup_{t\rightarrow \tau}\|(u, \mathbf{c}^i)\|_{\mathbb{H}^1\times H^1}=\infty\right\}}=1_{\left\{\limsup_{t\rightarrow \tau}\|(\nabla u, \nabla \Psi)\|_{\mathbb{L}^2\times W^{1,3+}}=\infty\right\}}.$$
\end{theorem}

The rest of this section is used to prove Theorems \ref{thm}, \ref{thm1} and we will divide the proof into three parts. In the first part, we prove the existence of strong martingale solution, the proof consists of constructing the approximate solutions, establishing the high-order a priori estimates and the stochastic compactness, identifying the limit. In the second part, we prove the uniqueness, by using the Gy\"{o}ngy-Krylov characterization to deduce the existence of local strong pathwise solution,  extending the local solution to maximal pathwise solution, and improving the time regularity. In the third part, we establish the blow-up criterion, i.e., Theorem \ref{thm1}.

\subsection{Existence of a local martingale solution}

 In our situation, we first consider a modified system. Let $\Phi_{R}:[0,\infty)\rightarrow [0,1]$ be a $C^{\infty}$-smooth function defined as:
\begin{eqnarray*}
 \Phi_{R}(x) \!=\!\left\{\begin{array}{ll}
                  \!\!1,& \mbox{if} \ 0<x<R,  \\
                  \!\!0,& \mbox{if} \ x>2R,  \\
                \end{array}\right.
\end{eqnarray*}
where $R$ is a fixed constant.

Multiplying by $\Phi_R(\|\nabla u\|_{\mathbb{L}^2})$ in front of the terms $u\cdot \nabla u$ and $u\cdot\nabla \mathbf{c}^i$, and $\Phi_R(\|\nabla \Psi\|_{W^{1,3^+}})$ in front of the terms ${\rm div}(z_i\mathbf{c}_i \nabla \Psi)$ and $\kappa\rho\nabla \Psi$, we get the following modified system:
\begin{eqnarray}\label{M1.1}
\left\{\begin{array}{ll}
\!\!\partial_t\mathbf{c}^i+\Phi_R(\|\nabla u\|_{\mathbb{L}^2})(u\cdot \nabla )\mathbf{c}^i=a_i{\rm div}(\nabla \mathbf{c}^i+\Phi_R(\|\nabla \Psi\|_{W^{1,3^+}})z_i\mathbf{c}^i\nabla\Psi),\\~i=1,2,\cdots, m,\\
\!\!-\Delta \Psi=\sum_{i=1}^m z_i\mathbf{c}^i=\rho,\\
\!\!\partial_t u-\mu\Delta u+\Phi_R(\|\nabla u\|_{\mathbb{L}^2})(u\cdot \nabla)u-\nabla p+ \kappa\Phi_R(\|\nabla \Psi\|_{W^{1,3^+}})\rho\nabla \Psi\\ =f(u, \nabla \Psi)\frac{d\mathcal{W}}{dt},\\
\!\!\nabla\cdot u=0.\\
\end{array}\right.
\end{eqnarray}

We list the definition of martingale solution of system \eqref{M1.1}, the solution is weak in probability sense but strong in PDEs sense.
\begin{definition}[global martingale solution]\label{def4.3} Let $\lambda$ be a Borel probability measure on space $H^1(\mathcal{D})\times \mathbb{H}^1(\mathcal{D})$ with $$\int_{H^1(\mathcal{D})\times \mathbb{H}^1(\mathcal{D})}|x|^p d\lambda\leq C,$$ for a constant $C$. We say $(\Omega, \mathcal{F},\{\mathcal{F}_t\}_{t\geq 0},\mathbb{P}, u, \mathbf{c}^i, \mathcal{W})$ is a martingale solution to system \eqref{M1.1} with the initial data $\lambda$ if

i. $(\Omega,\mathcal{F},\{\mathcal{F}_t\}_{t\geq 0},\mathbb{P})$ is a stochastic basis with a complete right-continuous filtration, $\mathcal{W}$ is a Wiener process relative to the filtration $\mathcal{F}_t$;

ii. the process $u$ is an $\mathbb{H}^1$-valued $\mathcal{F}_t$-progressively measurable satisfying for all $p\in[2,\infty)$
$$u\in L^p(\Omega; L^\infty(0,T; \mathbb{H}^1)\cap L^2(0,T;H^2));$$

iii. the processes $\mathbf{c}^i, i=1,\cdots, m$ are $H^1$-valued $\mathcal{F}_t$-progressively measurable satisfying for all $p\in [2,\infty), j\in[1,\infty)$
$$\mathbf{c}^i\geq 0,~\mbox{a.e.} ~{\rm and} ~\mathbf{c}^i\in L^p(\Omega; L^\infty(0,T; H^1\cap L^j)\cap L^2(0,T;H^2)),$$
 and $\Psi$ is an $H^3$-valued $\mathcal{F}_t$-progressively measurable process with
$$\Psi\in L^p(\Omega; L^\infty(0,T; H^3)\cap L^2(0,T;H^4));$$

iv. the initial law $\lambda=\mathbb{P}\circ (u_0, \mathbf{c}^i_0)^{-1}$;

v. it holds $\mathbb{P}$-a.s.
\begin{eqnarray*}
&&\mathbf{c}^i(t)=\mathbf{c}^i(0)-\int_{0}^{t}\Phi_R(\|\nabla u\|_{\mathbb{L}^2})(u\cdot \nabla )\mathbf{c}^ids\\&&\quad\qquad+\int_{0}^{t}a_i{\rm div}(\nabla \mathbf{c}^i+\Phi_R(\|\nabla \Psi\|_{W^{1,3^+}})z_i\mathbf{c}^i\nabla\Psi)ds,~ i=1,2,\cdots, m,\\
&&u(t)=u(0)+\mu\int_{0}^{t}P\Delta uds-\int_{0}^{t}P\Phi_R(\|\nabla u\|_{\mathbb{L}^2})(u\cdot \nabla)uds\nonumber \\ &&\quad\qquad- \int_{0}^{t}P\kappa\Phi_R(\|\nabla \Psi\|_{W^{1,3^+}})\rho\nabla \Psi ds+\int_{0}^{t}Pf(u, \nabla\Psi)d\mathcal{W},\\
&&-\Delta \Psi=\sum_{i=1}^m z_i\mathbf{c}^i=\rho, ~\mbox{a.e.}
\end{eqnarray*}
 for all $t>0$.
\end{definition}

\subsubsection{Approximate solutions and uniform estimates}

In this subsection, we are going to find a sequence of smooth approximate solutions $(u_n, \mathbf{c}_n^i, \Psi_n)$ satisfying the following approximate system:
\begin{eqnarray}\label{M1.2}
\left\{\begin{array}{ll}
\!\!\!\partial_t\mathbf{c}^i_n+\Phi_R(\|\nabla u_n\|_{\mathbb{L}^2})(u_n\cdot \nabla )\mathcal{J}_{\varepsilon(n)}\mathbf{c}^i_n\\=a_i{\rm div}(\nabla \mathbf{c}_n^i+\Phi_R(\|\nabla \Psi_n\|_{W^{1,3^+}})z_i\mathbf{c}_n^i\nabla\mathcal{J}_{\varepsilon(n)}\Psi_n),
~i=1,2,\cdots, m,\\
\!\!\!-\Delta \Psi_n=\sum_{i=1}^m z_i\mathbf{c}_n^i=\rho_n,\\
\!\!\!\partial_t u_n-\mu P\Delta u_n+\Phi_R(\|\nabla u_n\|_{\mathbb{L}^2})P(u_n\cdot \nabla)\mathcal{J}_{\varepsilon(n)}u_n\\+P\kappa\Phi_R(\|\nabla \Psi_n\|_{W^{1,3^+}})\rho_n\nabla \mathcal{J}_{\varepsilon(n)}\Psi_n
 =Pf(u_n, \nabla\Psi_n)\frac{d\mathcal{W}}{dt},\\
\end{array}\right.
\end{eqnarray}
where the constant $3^+:=3+\delta$ for small $\delta>0$, $\mathcal{J}_{\varepsilon(n)}$ is the Friedrich mollifiers with $\varepsilon(n)=\frac{1}{n}$, $n\in \mathbb{R}^+$. For more properties and the construction of such operators, see Lemma \ref{lem4.4} and Remark \ref{rem4.1}.

We linearize system \eqref{M1.2} to obtain
\begin{eqnarray}\label{M1.3}
\left\{\begin{array}{ll}
\!\!\!\partial_t\mathbf{c}^i_n+\Phi_R(\|\nabla \ddot{u}_n\|_{\mathbb{L}^2})(\ddot{u}_n\cdot \nabla )\mathcal{J}_{\varepsilon(n)}\ddot{\mathbf{c}}^i_n\\=a_i{\rm div}(\nabla \mathbf{c}_n^i+\Phi_R(\|\nabla \ddot{\Psi}_n\|_{W^{1,3^+}})z_i\ddot{\mathbf{c}}_n^i\nabla\mathcal{J}_{\varepsilon(n)}\ddot{\Psi}_n),~i=1,2,\cdots, m,\\
\!\!\!-\Delta \ddot{\Psi}_n=\sum_{i=1}^m z_i\ddot{\mathbf{c}}_n^i=\ddot{\rho}_n,\\
\!\!\!\partial_t u_n-\mu P\Delta u_n+\Phi_R(\|\nabla \ddot{u}_n\|_{\mathbb{L}^2})P(\ddot{u}_n\cdot \nabla)\mathcal{J}_{\varepsilon(n)}\ddot{u}_n\\+P\kappa\Phi_R(\|\nabla \ddot{\Psi}_n\|_{W^{1,3^+}})\ddot{\rho}_n\nabla \mathcal{J}_{\varepsilon(n)}\ddot{\Psi}_n
 =P\mathcal{J}_{\varepsilon(n)}f(\ddot{u}_n,\nabla\ddot{\Psi}_n)\frac{d\mathcal{W}}{dt},\\
\end{array}\right.
\end{eqnarray}
with the boundary conditions:
\begin{align}\label{e1.7*}
\left\{\begin{array}{ll}
\!\!\!(\partial_{\vec{n}}\mathbf{c}^i+z_i\ddot{\mathbf{c}}^i\partial_{\vec{n}}\ddot{\Psi})|_{\partial \mathcal{D}}=0,\\
\!\!\!\partial_{\vec{n}}\ddot{\Psi}|_{\partial \mathcal{D}}=\eta,\\
\!\!\!u|_{\partial \mathcal{D}}=0,
\end{array}\right.
\end{align}
and the initial data $(u_0, \mathbf{c}^i_0)$, where $(\ddot{u}, \ddot{\mathbf{c}}^i, \ddot{\Psi})$ is the known functions belonging to space $X\!:=\!X_1\times X_2$,
\begin{align*}
&X_1:=\left\{u:u\in L^p_\omega (L_t^\infty \mathbb{H}^1\cap L^2_t H^2), u|_{t=0}=u_0\right\},\\
&X_2:=\left\{(\mathbf{c}^i,\Psi):(\mathbf{c}^i,\Psi)\in L^p_\omega (L_t^\infty (H^1\times H^3)\cap L^2_t(H^2\times H^4)), \mathbf{c}^i|_{t=0}=\mathbf{c}^i_0\right\},
\end{align*}
for any $p\in [2,\infty)$.

\begin{remark} We add the truncation operator $\Phi_R$ in front of the nonlinear term $\ddot{\rho}_n\nabla \mathcal{J}_{\varepsilon(n)}\ddot{\Psi}_n$  only to close the contraction estimate of mapping $\mathcal{M}$ and not get the a priori estimates.
\end{remark}

\begin{remark} According to the definition of stochastic integral, we have to understand the stochastic integral $\int_{0}^{t}\mathcal{J}_{\varepsilon(n)}f(\ddot{u}_n,\nabla\ddot{\Psi}_n)d\mathcal{W}$ in the following form:
$$\int_{0}^{t}\mathcal{J}_{\varepsilon(n)}f(\ddot{u}_n,\nabla\ddot{\Psi}_n)d\mathcal{W}=\sum_{k\geq 1}\int_{0}^{t}\mathcal{J}_{\varepsilon(n)}(f(\ddot{u}_n,\nabla\ddot{\Psi}_n)e_k) d\beta_k,$$
which is an $H^k$-valued square integrable martingale for $k\geq 2$.
\end{remark}

Observe that the linear stochastic system \eqref{M1.3}-\eqref{e1.7*} with the Neumann boundary condition is perturbed by an additive noise. The solvability of the stochastic system is characterized in the following auxiliary proposition.
\begin{proposition}\label{p1} Suppose the functions $\ddot{u}, \ddot{\mathbf{c}}^i, \ddot{\Psi}$ belong to $X$, then there exists a unique solution $(u, \mathbf{c}^i, \Psi)\in X$ to system \eqref{M1.3}-\eqref{e1.7*}.
\end{proposition}
\begin{proof} Inspired by \cite{Kim1}, we first convert the stochastic system into a random system. For the random system, using the deterministic method, we could show that for a.s. $\omega$, there exists a unique strong solution $$(u, \mathbf{c}^i, \Psi)\in L^\infty(0,T;\mathbb{H}^1\times (H^1\times H^3))\cap L^2(0,T;H^2\times(H^2\times H^4)),$$
see \cite{Denk} and \cite[Lemma 2.1]{bothe} for the details of proof. Then,  similar to \cite{Kim1}, we could show that the solution is measurable. At last, we could establish the $p$-th moment estimate by same argument of the following lemma, i.e., Lemma \ref{lem4.1}.
\end{proof}

Define the mapping
$$\mathcal{M}:X\longmapsto X,~ (\ddot{u}_n, \ddot{\mathbf{c}}_n^i, \ddot{\Psi}_n)\longmapsto (u_n, \mathbf{c}_n^i, \Psi_n).$$
By Proposition \ref{p1}, we know that the mapping $\mathcal{M}$ is well-defined.

Next, we show that the mapping $\mathcal{M}$ is a contraction from $X$ into $X$. Once these properties are verified, the existence and uniqueness of approximate solutions of system \eqref{M1.2} will follows.

\begin{lemma}\label{lem4.1} $\mathcal{M}$ is a mapping from $X$ into itself.
\end{lemma}
\begin{proof}Using the It\^{o} formula to the function $\frac{1}{2}\|\nabla u_n\|_{\mathbb{L}^2}^2$, it holds that
\begin{align}\label{e4.5*}
\frac{1}{2}d\|\nabla u_n\|^2_{\mathbb{L}^2}&+\mu\|\Delta u_n\|^2_{\mathbb{L}^2}dt=\Phi_R(\|\nabla \ddot{u}_n\|_{\mathbb{L}^2})(\ddot{u}_n\cdot \nabla \mathcal{J}_{\varepsilon(n)}\ddot{u}_n, \Delta u_n)dt\nonumber\\&-\Phi_R(\|\nabla \ddot{\Psi}_n\|_{W^{1,3^+}})\kappa(\nabla(\ddot{\rho}_n\nabla \mathcal{J}_{\varepsilon(n)}\ddot{\Psi}_n), \nabla u_n)dt\nonumber\\ &+(\nabla P\mathcal{J}_{\varepsilon(n)}f(\ddot{u}_n, \nabla \ddot{\Psi}_n), \nabla u_n)d\mathcal{W}\nonumber\\&+\frac{1}{2}\|\nabla P \mathcal{J}_{\varepsilon(n)}f(\ddot{u}_n, \nabla \ddot{\Psi}_n)\|_{L_2(\mathcal{H};\mathbb{L}^2)}^2dt.
\end{align}
Taking inner product with $-\Delta \mathbf{c}_n^i$, we get
\begin{align}\label{e4.6*}
&\frac{1}{2}d\left(\|\nabla \mathbf{c}_n^i\|^2_{L^2}+\int_{\partial\mathcal{D}}z_i\ddot{\mathbf{c}}_n^i\mathbf{c}_n^i\eta d\mathcal{S}\right)+a_i\|\Delta \mathbf{c}_n^i\|^2_{L^2}dt\nonumber\\
&=(\Phi_R(\|\nabla \ddot{u}_n\|_{\mathbb{L}^2})(\ddot{u}_n\cdot \nabla )\mathcal{J}_{\varepsilon(n)}\ddot{\mathbf{c}}_n^i,  \Delta\mathbf{c}_n^i)dt\\
&\quad-\Phi_R(\|\nabla \ddot{\Psi}_n\|_{W^{1,3^+}})a_i({\rm div}(z_i\ddot{\mathbf{c}}_n^i\nabla\mathcal{J}_{\varepsilon(n)}\ddot{\Psi}_n), \Delta\mathbf{c}_n^i)dt.\nonumber
\end{align}
We are going to control all terms on the right-hand sides of \eqref{e4.5*}, \eqref{e4.6*}. Using the H\"{o}lder inequality, the Schwartz inequality and Lemma \ref{lem4.4}, we obtain
\begin{align*}
&\Phi_R(\|\nabla \ddot{u}_n\|_{\mathbb{L}^2})|(\ddot{u}_n\cdot \nabla \mathcal{J}_{\varepsilon(n)}\ddot{u}_n, \Delta u_n)|\\
&\qquad\qquad\leq \frac{\mu}{4}\|\Delta u_n\|^2_{\mathbb{L}^2}+C\Phi_R^2\|\ddot{u}_n\|_{\mathbb{H}^1}^2\|\nabla \mathcal{J}_{\varepsilon(n)}\ddot{u}_n\|_{\mathbb{H}^1}^2\\
&\qquad\qquad\leq \frac{\mu}{4}\|\Delta u_n\|^2_{\mathbb{L}^2}+Cn^2\Phi_R^2\|\ddot{u}_n\|_{\mathbb{H}^1}^2\|\ddot{u}_n\|_{\mathbb{H}^1}^2,\\
&\Phi_R(\|\nabla \ddot{\Psi}_n\|_{W^{1,3^+}})|-\kappa(\nabla(\ddot{\rho}_n\nabla \mathcal{J}_{\varepsilon(n)}\ddot{\Psi}_n), \nabla u_n)|\\
& \qquad\qquad\leq \frac{\mu}{4}\|\Delta u_n\|^2_{\mathbb{L}^2}+C\Phi_R^2\|\ddot{\rho}_n\|_{H^1}^2\|\nabla \mathcal{J}_{\varepsilon(n)}\ddot{\Psi}_n\|_{H^1}^2\\
&\qquad\qquad\leq \frac{\mu}{4}\|\Delta u_n\|^2_{\mathbb{L}^2}+C\Phi_R^2\|\ddot{\rho}_n\|_{H^1}^2\|\nabla \ddot{\Psi}_n\|_{H^1}^2,\\
&\Phi_R(\|\nabla \ddot{u}_n\|_{\mathbb{L}^2})|(\ddot{u}_n\cdot \nabla \mathcal{J}_{\varepsilon(n)}\ddot{\mathbf{c}}_n^i,  \Delta\mathbf{c}_n^i)|\\
& \qquad\qquad\leq  \frac{a_i}{4}\|\Delta \mathbf{c}^i_n\|^2_{L^2}+C\Phi_R^2\|\ddot{u}_n\|_{\mathbb{H}^1}^2\|\nabla \mathcal{J}_{\varepsilon(n)}\ddot{\mathbf{c}}_n^i\|_{H^1}^2\\
&\qquad\qquad\leq \frac{a_i}{4}\|\Delta \mathbf{c}^i_n\|^2_{L^2}+Cn^2\Phi_R^2\|\ddot{u}_n\|_{\mathbb{H}^1}^2\|\ddot{\mathbf{c}}_n^i\|_{H^1}^2,\\
&\Phi_R(\|\nabla \ddot{\Psi}_n\|_{W^{1,3^+}})|-a_i({\rm div}(z_i\ddot{\mathbf{c}}_n^i\nabla\mathcal{J}_{\varepsilon(n)}\ddot{\Psi}_n), \Delta\mathbf{c}_n^i)|
\\& \qquad\qquad\leq \frac{a_i}{4}\|\Delta \mathbf{c}^i_n\|^2_{L^2}+C\Phi_R^2\|\ddot{\mathbf{c}}_n^i\|_{H^1}^2\|\nabla \ddot{\Psi}_n\|_{W^{1,3^+}}^2.
\end{align*}
The H\"{o}lder inequality and the trace embedding $H^1(\mathcal{D})\hookrightarrow L^2(\partial\mathcal{D})$ yield
\begin{align*}
\left|\int_{\partial\mathcal{D}}z_i\ddot{\mathbf{c}}_n^i\mathbf{c}_n^i\eta d\mathcal{S}\right|&\leq C\|\eta\|_{L^\infty(\partial\mathcal{D})}\|\ddot{\mathbf{c}}_n^i\|_{L^2(\partial\mathcal{D})}\|\mathbf{c}_n^i\|_{L^2(\partial\mathcal{D})}\\
&\leq \frac{1}{4}\|\mathbf{c}_n^i\|_{H^1}^2+C\|\eta\|_{L^\infty(\partial\mathcal{D})}^2\|\ddot{\mathbf{c}}_n^i\|_{H^1}^2.
\end{align*}
Applying the Burkholder-Davis-Gundy inequality to the stochastic term, and by assumption \eqref{as3}, Lemma \ref{lem4.4}, we build
\begin{align*}
&\mathbb{E}\left(\sup_{t\in[0,T]}\left|\int_{0}^{t}(\nabla P\mathcal{J}_{\varepsilon(n)}f(\ddot{u}_n, \nabla \ddot{\Psi}_n), \nabla u_n)d\mathcal{W}\right|^p\right) \\
&\leq \frac{1}{2}\mathbb{E}\left(\sup_{t\in[0,T]}\|\nabla u_n\|_{\mathbb{L}^2}^{2p}\right)+C\mathbb{E}\left(\int_{0}^{T}\|\nabla \ddot{u}_n\|_{\mathbb{L}^2}^2+\| \ddot{\mathbf{c}}^i_n\|_{L^2}^2dt\right)^p\\
&\leq \frac{1}{2}\mathbb{E}\left(\sup_{t\in[0,T]}\|\nabla u_n\|_{\mathbb{L}^2}^{2p}\right)+CT^{p}\mathbb{E}\left[\sup_{t\in [0,T]}(\|\nabla \ddot{u}_n\|_{\mathbb{L}^2}^{2p}+\| \ddot{\mathbf{c}}^i_n\|_{L^2}^{2p})\right].
\end{align*}
For the last term in \eqref{e4.5*}, assumption \eqref{as3} gives
\begin{align*}
\frac{1}{2}\|\nabla \mathcal{J}_{\varepsilon(n)}f(\ddot{u}_n, \nabla \ddot{\Psi}_n)\|_{L_2(\mathcal{H};L^2)}^2\leq \frac{\ell_3}{2}(\|\ddot{u}_n\|_{\mathbb{H}^1}^2+ \|\ddot{\Psi}_n\|_{H^2}^2).
\end{align*}

Combining all the estimates, integrating of time $t$ and then taking supremum, power $p$, expectation, we complete the proof.
\end{proof}
\begin{lemma}  $\mathcal{M}$ is a contraction mapping.
\end{lemma}
\begin{proof} Assume that $(\ddot{u}^l_n, \ddot{\mathbf{c}}_n^{i, l}, \ddot{\Psi}_n^l)\in X, \mbox{ for }l=1,2, ~i=1,2,\cdots, m$, we denote the difference of $(\ddot{u}^l_n, \ddot{\mathbf{c}}_n^{i, l}, \ddot{\Psi}_n^l)\in X, l=1,2$ as $(\ddot{u}_n, \ddot{\mathbf{c}}_n^i, \ddot{\Psi}_n)$ and the difference of $(u^l_n, \mathbf{c}_n^{i,l}, \Psi_n^l)$ as $(u_n, \mathbf{c}_n^{i}, \Psi_n)$ respectively satisfying
\begin{eqnarray}\label{M1.4}
\left\{\begin{array}{ll}
\!\!\!\partial_t\mathbf{c}^i_n-a_i \Delta\mathbf{c}_n^i+\Phi_R(\|\nabla \ddot{u}^1_n\|_{\mathbb{L}^2})\big[(\ddot{u}_n\cdot \nabla )\mathcal{J}_{\varepsilon(n)}\ddot{\mathbf{c}}^{i,1}_n+(\ddot{u}^2_n\cdot \nabla )\mathcal{J}_{\varepsilon(n)}\ddot{\mathbf{c}}^{i}_n\big]\\
\quad+\big[\Phi_R(\|\nabla \ddot{u}^1_n\|_{\mathbb{L}^2})-\Phi_R(\|\nabla \ddot{u}^2_n\|_{\mathbb{L}^2})\big](\ddot{u}^2_n\cdot \nabla )\mathcal{J}_{\varepsilon(n)}\ddot{\mathbf{c}}^{i,2}_n\\=\Phi_R(\|\nabla \ddot{\Psi}^1_n\|_{W^{1,3^+}})a_i\big[{\rm div}(z_i\ddot{\mathbf{c}}_n^i\nabla\mathcal{J}_{\varepsilon(n)}\ddot{\Psi}^1_n)+{\rm div}(z_i\ddot{\mathbf{c}}_n^{i,2}\nabla\mathcal{J}_{\varepsilon(n)}\ddot{\Psi}_n)\big]\\
\quad+\big[\Phi_R(\|\nabla \ddot{\Psi}^1_n\|_{W^{1,3^+}})-\Phi_R(\|\nabla \ddot{\Psi}^2_n\|_{W^{1,3^+}})\big]a_i{\rm div}(z_i\ddot{\mathbf{c}}_n^{i,2}\nabla\mathcal{J}_{\varepsilon(n)}\ddot{\Psi}^2_n),\\
\!\!\!-\Delta \ddot{\Psi}_n=\sum_{i=1}^m z_i\ddot{\mathbf{c}}_n^i=\ddot{\rho}_n,\\
\!\!\!\partial_t u_n-\mu P\Delta u_n+\Phi_R(\|\nabla \ddot{u}^1_n\|_{\mathbb{L}^2})\big[P(\ddot{u}_n\cdot \nabla)\mathcal{J}_{\varepsilon(n)}\ddot{u}^1_n+P(\ddot{u}^2_n\cdot \nabla)\mathcal{J}_{\varepsilon(n)}\ddot{u}_n\big]\\
\quad+\big[\Phi_R(\|\nabla \ddot{u}^1_n\|_{\mathbb{L}^2})-\Phi_R(\|\nabla \ddot{u}^2_n\|_{\mathbb{L}^2})\big]P(\ddot{u}^2_n\cdot \nabla)\mathcal{J}_{\varepsilon(n)}\ddot{u}^2_n\\
\quad+P\kappa\Phi_R(\|\nabla \ddot{\Psi}^1_n\|_{W^{1,3^+}})\big[\ddot{\rho}_n\nabla \mathcal{J}_{\varepsilon(n)}\ddot{\Psi}^1_n+\ddot{\rho}^2_n\nabla \mathcal{J}_{\varepsilon(n)}\ddot{\Psi}_n\big]\\
\quad+P\kappa\big[\Phi_R(\|\nabla \ddot{\Psi}^1_n\|_{W^{1,3^+}})-\Phi_R(\|\nabla \ddot{\Psi}^2_n\|_{W^{1,3^+}})\big]\ddot{\rho}^2_n\nabla \mathcal{J}_{\varepsilon(n)}\ddot{\Psi}^2_n\\
 =P\mathcal{J}_{\varepsilon(n)}(f(\ddot{u}^1_n,\nabla\ddot{\Psi}_n^1)-f(\ddot{u}^2_n,\nabla\ddot{\Psi}_n^2))\frac{d\mathcal{W}}{dt},\\
\end{array}\right.
\end{eqnarray}
with the boundary conditions:
\begin{align}\label{e1.7**}
\left\{\begin{array}{ll}
\!\!\!(\partial_{\vec{n}}\mathbf{c}^i+z_i(\ddot{\mathbf{c}}^{i,1}\partial_{\vec{n}}\ddot{\Psi}^1-\ddot{\mathbf{c}}^{i,2}\partial_{\vec{n}}\ddot{\Psi}^2))|_{\partial \mathcal{D}}=0,\\
\!\!\!\partial_{\vec{n}}\ddot{\Psi}|_{\partial \mathcal{D}}=0,\\
\!\!\!u|_{\partial \mathcal{D}}=0,
\end{array}\right.
\end{align}
and the initial data $(u_n(0), \mathbf{c}^i_n(0))=(0, 0)$. We also have that $\int_{\mathcal{D}}\mathbf{c}^i_ndx=\int_{\mathcal{D}}\mathbf{c}^i_n(0)dx=0$ a.e., as a result, the norm $\|\mathbf{c}^i_n\|_{L^2}$ can be controlled by $\|\nabla\mathbf{c}^i_n\|_{L^2}$.

Observe that introducing cut-off functions brings additional difficulties in closing the estimates. Therefore, we further introduce a stopping time technique. Define the stopping times $\tau_R^l:=\inf\left\{t\geq 0; \|(\ddot{u}_n^l, \ddot{\mathbf{c}}_n^{i,l})\|^2_{\widetilde{X}}\geq 2R\right\}$ for $l=1,2$, where the space $$\widetilde{X}:= L^\infty(0,T;H^1\times \mathbb{H}^1)\cap L^2(0,T;(H^2)^2).$$ If the right-hand side set is empty, we take $\tau_R^l=T$. Without loss of generality, we assume $\tau_R^1\leq \tau_R^2$. Consequently, if $t\geq \tau_R^2$, we have
$$\Phi_{R}=0,~ i=1,2.$$

Using the It\^{o} formula to the function $\frac{1}{2}\|\nabla u_n\|_{\mathbb{L}^2}^2$, and taking inner product with $-\Delta \mathbf{c}^i_n$ in $\eqref{M1.4}_1$, as in Lemma \ref{lem4.1}, we shall estimate all terms.

Applying the H\"{o}lder inequality and Lemma \ref{lem4.4}, we have
\begin{align*}
&\mathbb{E}\left(\int_{0}^{T}\left|\big(\Phi_R(\|\nabla \ddot{u}^1_n\|_{\mathbb{L}^2})\big[(\ddot{u}_n\cdot \nabla )\mathcal{J}_{\varepsilon(n)}\ddot{\mathbf{c}}^{i,1}_n+(\ddot{u}^2_n\cdot \nabla )\mathcal{J}_{\varepsilon(n)}\ddot{\mathbf{c}}^{i}_n\big], -\Delta \mathbf{c}^i_n\big)\right|dt\right)^p\\
&=\mathbb{E}\left(\int_{0}^{\tau_R^1}\left|(\Phi_R(\|\nabla \ddot{u}^1_n\|_{\mathbb{L}^2})\big[(\ddot{u}_n\cdot \nabla )\mathcal{J}_{\varepsilon(n)}\ddot{\mathbf{c}}^{i,1}_n+(\ddot{u}^2_n\cdot \nabla )\mathcal{J}_{\varepsilon(n)}\ddot{\mathbf{c}}^{i}_n\big], \Delta \mathbf{c}^i_n)\right|dt\right)^p\\
&\leq \frac{a_i}{8}\mathbb{E}\left(\int_{0}^{T}\|\Delta \mathbf{c}^i_n\|_{L^2}^2dt\right)^p+Cn^{2p}\mathbb{E}\left(\int_{0}^{\tau_R^1}\|\ddot{u}_n\|_{\mathbb{H}^1}^2\|\ddot{\mathbf{c}}^{i,1}_n\|_{H^1}^2
+\|\ddot{u}^2_n\|_{\mathbb{H}^1}^2\|\ddot{\mathbf{c}}^{i}_n\|_{H^1}^2dt\right)^p\\
&\leq \frac{a_i}{8}\mathbb{E}\left(\int_{0}^{T}\|\Delta \mathbf{c}^i_n\|_{L^2}^2dt\right)^p+Cn^{2p}R^{p}\mathbb{E}\left(\int_{0}^{\tau_R^1}\|\ddot{u}_n\|_{\mathbb{H}^1}^2
+\|\ddot{\mathbf{c}}^{i}_n\|_{H^1}^2dt\right)^p\\
&\leq \frac{a_i}{8}\mathbb{E}\left(\int_{0}^{T}\|\Delta \mathbf{c}^i_n\|_{L^2}^2dt\right)^p+Cn^{2p}R^{p}T^p\mathbb{E}\bigg(\sup_{t\in [0,T]}(\|\ddot{u}_n\|_{\mathbb{H}^1}^{2p}
+\|\ddot{\mathbf{c}}^{i}_n\|_{H^1}^{2p})\bigg),
\end{align*}
and
\begin{align*}
&\mathbb{E}\left(\int_{0}^{T}\left|\big(\big[\Phi_R(\|\nabla \ddot{u}^1_n\|_{\mathbb{L}^2})-\Phi_R(\|\nabla \ddot{u}^2_n\|_{\mathbb{L}^2})\big](\ddot{u}^2_n\cdot \nabla )\mathcal{J}_{\varepsilon(n)}\ddot{\mathbf{c}}^{i,2}_n, -\Delta \mathbf{c}^i_n\big)\right|dt\right)^p\\
&=\mathbb{E}\left(\int_{0}^{\tau_R^2}\left|(\big[\Phi_R(\|\nabla \ddot{u}^1_n\|_{\mathbb{L}^2})-\Phi_R(\|\nabla \ddot{u}^2_n\|_{\mathbb{L}^2})\big](\ddot{u}^2_n\cdot \nabla )\mathcal{J}_{\varepsilon(n)}\ddot{\mathbf{c}}^{i,2}_n, \Delta \mathbf{c}^i_n)\right|dt\right)^p\\
&\leq Cn^{2p}\mathbb{E}\left(\int_{0}^{\tau_R^2}\|\nabla \ddot{u}_n\|_{\mathbb{L}^2}\|\Delta \mathbf{c}^i_n\|_{L^2}\|\ddot{u}^2_n\|_{\mathbb{H}^1}\|\ddot{\mathbf{c}}^{i,2}_n\|_{H^1}dt\right)^p\\
&\leq Cn^{2p}R^p\mathbb{E}\left(\int_{0}^{T}\|\nabla \ddot{u}_n\|^2_{\mathbb{L}^2}dt\right)^p+ \frac{a_i}{8}\mathbb{E}\left(\int_{0}^{T}\|\Delta \mathbf{c}^i_n\|_{L^2}^2dt\right)^p\\
&\leq Cn^{2p}R^{p}T^p\mathbb{E}\bigg(\sup_{t\in [0,T]}\|\ddot{u}_n\|_{\mathbb{H}^1}^{2p}\bigg)
+ \frac{a_i}{8}\mathbb{E}\left(\int_{0}^{T}\|\Delta \mathbf{c}^i_n\|_{L^2}^2dt\right)^p.
\end{align*}
Also,
\begin{align*}
&\mathbb{E}\left(\int_{0}^{T}\left|\Phi_R(\|\nabla \ddot{\Psi}^1_n\|_{W^{1,3^+}})a_i({\rm div}(z_i\ddot{\mathbf{c}}_n^i\nabla\mathcal{J}_{\varepsilon(n)}\ddot{\Psi}^1_n)+{\rm div}(z_i\ddot{\mathbf{c}}_n^{i,2}\nabla\mathcal{J}_{\varepsilon(n)}\ddot{\Psi}_n), -\Delta \mathbf{c}^i_n)\right|dt\right)^p\\
&\leq C\mathbb{E}\bigg(\int_{0}^{\tau_R^1}\Phi_R(\|\nabla \ddot{\Psi}^1_n\|_{W^{1,3^+}})\|\Delta \mathbf{c}^i_n\|_{L^2}(\|\nabla\ddot{\mathbf{c}}_n^i\|_{L^2}\|\nabla\ddot{\Psi}^1_n\|_{L^6}+\|\ddot{\mathbf{c}}_n^i\|_{L^6}\|\nabla \ddot{\Psi}^1_n\|_{W^{1,3^+}}\\
&\quad+\|\nabla\ddot{\mathbf{c}}_n^{i,2}\|_{L^2}\|\nabla\ddot{\Psi}_n\|_{L^6}+\|\ddot{\mathbf{c}}_n^{i,2}\|_{L^6}\|\nabla  \mathcal{J}_{\varepsilon(n)}\ddot{\Psi}_n\|_{H^{2}})
dt\bigg)^p\\
&\leq \frac{a_i}{8}\mathbb{E}\left(\int_{0}^{T}\|\Delta\mathbf{c}^i_n\|_{L^2}^2dt\right)^p+C(1+n^{2p})R^pT^p\mathbb{E}\bigg(\sup_{t\in [0,T]}\|\ddot{\mathbf{c}}^i_n\|_{H^1}^{2p}\bigg),
\end{align*}
and
\begin{align*}
&\mathbb{E}\left(\int_{0}^{T}\left|\big(\big[\Phi_R(\|\nabla \ddot{\Psi}^1_n\|_{W^{1,3^+}})-\Phi_R(\|\nabla \ddot{\Psi}^2_n\|_{W^{1,3^+}})\big]a_i{\rm div}(z_i\ddot{\mathbf{c}}_n^{i,2}\nabla\mathcal{J}_{\varepsilon(n)}\ddot{\Psi}^2_n), -\Delta \mathbf{c}^i_n\big)\right|dt\right)^p\\
&\leq C\mathbb{E}\bigg(\int_{0}^{\tau_R^2}\|\nabla \ddot{\Psi}_n\|_{W^{1,3^+}}\|\Delta \mathbf{c}^i_n\|_{L^2}\|\nabla\ddot{\mathbf{c}}^{i,2}_n\|_{L^2}\|\nabla  \mathcal{J}_{\varepsilon(n)}\ddot{\Psi}^2_n\|_{H^{2}}dt\bigg)^p\\
&\leq Cn^{2p}R^p\mathbb{E}\bigg(\int_{0}^{\tau_R^2}\|\nabla \ddot{\Psi}_n\|_{W^{1,3^+}}\|\Delta \mathbf{c}^i_n\|_{L^2}dt\bigg)^p\\
&\leq  \frac{a_i}{8}\mathbb{E}\left(\int_{0}^{T}\|\Delta\mathbf{c}^i_n\|_{L^2}^2dt\right)^p+Cn^{4p}R^{2p}T^p\mathbb{E}\bigg(\sup_{t\in [0,T]}\|\ddot{\mathbf{c}}^i_n\|_{H^1}^{2p}\bigg).
\end{align*}
Using the H\"{o}lder inequality and the trace inequality, we have
\begin{align*}
\left|\int_{\partial\mathcal{D}}z_i(\ddot{\mathbf{c}}_n^{i,1}-\ddot{\mathbf{c}}_n^{i,2})\mathbf{c}_n^i\eta d\mathcal{S}\right|&\leq C\|\eta\|_{L^\infty(\partial\mathcal{D})}\|\ddot{\mathbf{c}}_n^i\|_{L^2(\partial\mathcal{D})}\|\mathbf{c}_n^i\|_{L^2(\partial\mathcal{D})}\\
&\leq \frac{1}{4}\|\mathbf{c}_n^i\|_{H^1}^2+C\|\eta\|_{L^\infty(\partial\mathcal{D})}^2\|\ddot{\mathbf{c}}_n^i\|_{H^1}^2.
\end{align*}

Similarly, we could estimate the rest of terms in velocity equation by the same way due to the same constitution, we do not give the details. Hence, choosing the time $T$ small enough and assuming that $\eta$ satisfies $C\|\eta\|_{L^\infty(\partial\mathcal{D})}^{2p}\leq \frac{1}{2}$, we infer that the mapping is a contraction for any fixed $n, R$. This completes the proof.
\end{proof}

We verified both conditions of contraction mapping theorem, the local existence and uniqueness of approximate solutions $(u_n, \mathbf{c}^i_n)\in X$ to system \eqref{M1.2} follows. In order to extend the local solution to global, we proceed to show that the bounds are independent of $n$.
\begin{lemma}\label{lem2.1} Assume that assumption \eqref{as3} holds. Then, the sequence of solutions $(u_n, \mathbf{c}_n^i)$ has the following a priori estimates for all $p, j\in [1,\infty)$
\begin{align}
&\mathbb{E}\left(\sup_{t\in [0,T]}\|\mathbf{c}_n^i\|_{L^j}^p\right)\leq C,\label{e4.1}\\
&\mathbb{E}\left(\sup_{t\in [0,T]}\|u_n\|^{2p}_{\mathbb{H}^1}\right)+\mathbb{E}\left(\int_{0}^{T}\|\Delta u_n\|_{\mathbb{L}^2}^2dt\right)^p\leq C,\label{e4.2}
\end{align}
where constant $C$ is independent of $n$ but depends on $R$.
\end{lemma}

\begin{proof} For simplification, we use $(u, \mathbf{c}_i, \Psi)$ instead of $(u_n, \mathbf{c}_n^i, \Psi_n)$. Following the ideas of \cite{lee}, we use the Moser-type iteration to show \eqref{e4.1}. Taking inner product with $\mathbf{c}^{2k-1}_i$ for $k=2,3,\cdots$ in $\eqref{M1.2}_1$, using the fact that $(u\cdot\nabla \mathbf{c}_i, \mathbf{c}^{2k-1}_i)=0$, we obtain
\begin{align}\label{e2.5}
\frac{1}{2k}d\|\mathbf{c}_i\|_{L^{2k}}^{2k}+\frac{2k-1}{k^2}a_i\|\nabla \mathbf{c}^k_i\|_{L^2}^2dt&\leq C\widetilde{\Phi}_R\frac{2k-1}{k}\|\nabla\Psi\|_{L^6}\|\mathbf{c}^k_i\|_{L^3}\|\nabla\mathbf{c}^k_i\|_{L^2}\nonumber\\
&\leq CR\|\mathbf{c}^k_i\|_{L^1}^\frac{2}{3}\|\nabla\mathbf{c}^k_i\|_{L^2}^\frac{4}{3}dt,
\end{align}
where $\widetilde{\Phi}_R:=\Phi_R(\|\nabla\Psi\|_{W^{1,3^+}})$.
The Young inequality gives
\begin{align}\label{e4.6}
d\|\mathbf{c}_i\|_{L^{2k}}^{2k}+a_i\|\nabla \mathbf{c}^k_i\|_{L^2}^2dt&\leq C_kR\|\mathbf{c}^k_i\|_{L^1}^2dt,
\end{align}
where $C_k$ is a constant satisfying $C_k\leq ck^m$ for $m$ large. By interpolating $\|\mathbf{c}_i^k\|_{L^2}^2\leq C(\|\nabla\mathbf{c}_i^k\|_{L^2}^2+\|\mathbf{c}_i^k\|_{L^1}^2)$, and using \eqref{e4.6}, we find
\begin{align}\label{e4.7}
\frac{d}{dt}\|\mathbf{c}_i\|_{L^{2k}}^{2k}\leq -C\|\mathbf{c}_i\|_{L^{2k}}^{2k}+C_kR\|\mathbf{c}_i\|_{L^k}^{2k}.
\end{align}
We use the Gronwall lemma and \eqref{e4.7} to obtain
\begin{align*}
\sup_{t\in [0,T]}\|\mathbf{c}_i\|_{L^{2k}}^{2k}&\leq \|\mathbf{c}_i(0)\|_{L^{2k}}^{2k}+C_kR\sup_{t\in [0,T]}\|\mathbf{c}_i\|_{L^k}^{2k}\\
&\leq C\|\mathbf{c}_i(0)\|_{L^{\infty}}^{2k}+C_kR\sup_{t\in [0,T]}\|\mathbf{c}_i\|_{L^k}^{2k}\\
&\leq C\|\mathbf{c}_i(0)\|_{L^{\infty}}^{2k}+CRk^m\left(\sup_{t\in [0,T]}\|\mathbf{c}_i\|_{L^k}^{2k}\right),
\end{align*}
consequently,
\begin{align}\label{e4.8}
\sup_{t\in [0,T]}\|\mathbf{c}_i\|_{L^{2k}}\leq C^\frac{1}{2k}\|\mathbf{c}_i(0)\|_{L^{\infty}}+(CR)^\frac{1}{2k}k^\frac{m}{2k}\left(\sup_{t\in [0,T]}\|\mathbf{c}_i\|_{L^k}\right).
\end{align}
Letting $k=2^j$ for all $j\in \mathbb{N}$ in \eqref{e4.8}, we get
\begin{align*}
\sup_{t\in [0,T]}\|\mathbf{c}_i\|_{L^{2^{j+1}}}\leq C^\frac{1}{{2^{j+1}}}\|\mathbf{c}_i(0)\|_{L^{\infty}}+(CR)^\frac{1}{2^{j+1}}2^\frac{jm}{2^{j+1}}\left(\sup_{t\in [0,T]}\|\mathbf{c}_i\|_{L^{2^j}}\right).
\end{align*}
By the iteration,
\begin{align}\label{e4.9}
\sup_{t\in [0,T]}\|\mathbf{c}_i\|_{L^{2^{j+1}}}\leq j(CR)^{\sum_{j=1}^{\infty}\frac{1}{2^{j+1}}}2^{\sum_{j=1}^{\infty}\frac{jm}{2^{j+1}}}\left(\|\mathbf{c}_i(0)\|_{L^{\infty}}+\sup_{t\in [0,T]}\|\mathbf{c}_i\|_{L^{2}}\right).
\end{align}
Taking power $p$ and then expectation in \eqref{e4.9}, since the series are convergent,  we have
\begin{align}\label{4.27}
\mathbb{E}\left(\sup_{t\in [0,T]}\|\mathbf{c}_i\|^p_{L^{2^{j+1}}}\right)\leq C(j,m, p)\left(\mathbb{E}\|\mathbf{c}_i(0)\|^p_{L^{\infty}}+\mathbb{E}\bigg(\sup_{t\in [0,T]}\|\mathbf{c}_i\|^p_{L^{2}}\bigg)\right).
\end{align}
Bound \eqref{e4.1} follows from \eqref{4.27} and \eqref{e3.13}.

Next, we use the It\^{o} formula to the function $\frac{1}{2}\|\nabla u\|^2_{\mathbb{L}^2}$ to obtain
\begin{align}\label{e4.11}
&\frac{1}{2}d\|\nabla u\|^2_{\mathbb{L}^2}+\mu\|\Delta u\|^2_{\mathbb{L}^2}dt\nonumber\\ &=\overline{\Phi}_R(u\cdot \nabla \mathcal{J}_\varepsilon u, \Delta u)dt-\widetilde{\Phi}_R(\nabla(\kappa\rho\nabla \mathcal{J}_\varepsilon\Psi), \nabla u)dt+(\nabla P\mathcal{J}_{\varepsilon}f(u, \nabla \Psi), \nabla u)d\mathcal{W}\nonumber\\ &\quad+\frac{1}{2}\|\nabla P\mathcal{J}_{\varepsilon}f(u, \nabla \Psi)\|_{L_2(\mathcal{H};\mathbb{L}^2)}^2dt\nonumber\\
&\leq C\overline{\Phi}_R\|\Delta u\|_{\mathbb{L}^2}\|u\|_{L^4}\|\nabla u\|_{L^4}dt+C\|\Delta u\|_{\mathbb{L}^2}\|\rho\|_{L^3}\|\nabla\Psi\|_{L^6}dt\nonumber\\
&\quad+\frac{\ell_3}{2}(\|\nabla u\|_{\mathbb{L}^2}^2+\|\Delta \Psi\|^2_{L^2})dt+(\nabla \mathcal{J}_{\varepsilon}f(u, \nabla \Psi), \nabla u)d\mathcal{W}\nonumber\\
&\leq C\overline{\Phi}_R\|\Delta u\|_{\mathbb{L}^2}^\frac{3}{2}\|u\|_{\mathbb{L}^2}^\frac{1}{2}\|\nabla u\|_{\mathbb{L}^2}dt+\frac{1}{4}\|\Delta u\|_{\mathbb{L}^2}^2dt+C\|\rho\|_{L^3}^2\|\nabla\Psi\|_{L^6}^2dt\nonumber\\
&\quad+\ell_3(\|\nabla u\|_{\mathbb{L}^2}^2+\|\Delta \Psi\|^2_{L^2})dt+(\nabla \mathcal{J}_{\varepsilon}f(u, \nabla \Psi), \nabla u)d\mathcal{W}\nonumber\\
&\leq \frac{1}{2}\|\Delta u\|_{\mathbb{L}^2}^2dt+\left(\overline{\Phi}_R^2\|u\|_{\mathbb{L}^2}^2\|\nabla u\|_{\mathbb{L}^2}^2\right)\|\nabla u\|_{\mathbb{L}^2}^2dt+C\|\rho\|_{L^3}^2\|\nabla\Psi\|_{L^6}^2dt\nonumber\\
&\quad+C(\|\nabla u\|_{\mathbb{L}^2}^2+\|\Delta\Psi\|^2_{L^2})dt+(\nabla \mathcal{J}_{\varepsilon}f(u, \nabla \Psi), \nabla u)d\mathcal{W},
\end{align}
where $\overline{\Phi}_R:=\Phi_R(\|\nabla u_n\|^2_{\mathbb{L}^2})$. In the first inequality above, we used the H\"{o}lder inequality, Lemma \ref{lem4.4} and assumption \eqref{as3}. In the second inequality, we used the Gagliardo-Nirenberg inequality \eqref{e4.18}, leading to
\begin{align*}
|(u\cdot \nabla \mathcal{J}_\varepsilon u, \Delta u)|&\leq C\left\{\begin{array}{ll}
\!\!\!\|\Delta u\|_{\mathbb{L}^2}\|u\|_{\mathbb{L}^4}\|\nabla u\|_{\mathbb{L}^4},~ {\rm if~ }~d=2  ,\\
\!\!\!\|\Delta u\|_{\mathbb{L}^2}\|u\|_{\mathbb{L}^6}\| \nabla u\|_{\mathbb{L}^3},~  {\rm if~ }~d=3,
\end{array}\right.\nonumber\\
&\leq C\left\{\begin{array}{ll}
\!\!\!\|\Delta u\|_{\mathbb{L}^2}^\frac{3}{2}\|u\|^\frac{1}{2}_{\mathbb{L}^2}\|\nabla u\|_{\mathbb{L}^2},~ {\rm if~ }~d=2  ,\\
\!\!\!\|\Delta u\|_{\mathbb{L}^2}^\frac{3}{2}\|u\|_{\mathbb{L}^6}\| \nabla u\|^\frac{1}{2}_{\mathbb{L}^2},~  {\rm if~ }~d=3.
\end{array}\right.
\end{align*}
For the stochastic term, assumption \eqref{as3}, Lemma \ref{lem4.4}, and the similar argument of \eqref{e3.11} imply
\begin{align}\label{4.19}
&\mathbb{E}\left(\sup_{s\in [0,t]}\left|\int_{0}^{s}(\nabla P\mathcal{J}_{\varepsilon}f(u, \nabla \Psi), \nabla u)d\mathcal{W}\right|^p\right)\nonumber\\
&\leq \frac{1}{2}\mathbb{E}\left(\sup_{s\in [0,t]}\|\nabla u\|_{\mathbb{L}^2}^{2p}\right)+C\mathbb{E}\left(\int_{0}^{t}\|\nabla u\|_{\mathbb{L}^2}^{2}+\|\Delta \Psi\|_{L^2}^{2}ds\right)^p.
\end{align}
From \eqref{e4.11} and \eqref{4.19}, again by the Gronwall lemma, we have
\begin{align*}
\mathbb{E}\left(\sup_{t\in [0,T]}\|u\|^{2p}_{\mathbb{H}^1}\right)+\mathbb{E}\left(\int_{0}^{T}\|\Delta u\|_{\mathbb{L}^2}^2dt\right)^p\leq C,
\end{align*}
where $C$ is independent of $n$. This completes the proof.
\end{proof}
\begin{remark}
 Actually, in inequality \eqref{e2.5}, it is enough for the cut-off function $\Phi$ depending on the norm $\|\nabla\Psi\|_{L^6}$.
\end{remark}

\begin{remark}\label{rem2.2} In the last line of \eqref{e4.11}, we use the truncation operator $\overline{\Phi}_R$ to control the ``bad" term $\|u\|_{\mathbb{L}^2}^2\|\nabla u\|_{\mathbb{L}^2}^2$. Therefore, if we could show
\begin{align*}
\mathbb{E}\int_{0}^{t}\|u\|_{\mathbb{L}^2}^2\|\nabla u\|_{\mathbb{L}^2}^2ds\leq C,
\end{align*}
for a constant $C$ which is independent of $R,n$, then the global solution follows.
\end{remark}

Next, we improve the space regularity of $\mathbf{c}^i_n$.
\begin{lemma}\label{lem2.2} The approximate solutions $\mathbf{c}^i_n$ satisfy
\begin{align}\label{e2.16}
\mathbb{E}\left(\sup_{t\in [0,T]}\|\mathbf{c}_n^i\|^{2p}_{H^1}\right)+\mathbb{E}\left(\int_{0}^{T}a_i\|\mathbf{c}^i_n\|_{H^2}^2dt\right)^p\leq C,
\end{align}
for a constant $C$ depending on $R$ but independent of $n$ and $p\in[1,\infty)$.
\end{lemma}
\begin{proof} Taking inner product with $-\Delta \mathbf{c}^i$ in equation $\eqref{M1.2}_1$, yields that
\begin{align}\label{e4.16}
&\frac{1}{2}d\left(\|\nabla\mathbf{c}^i\|_{L^2}^2+\int_{\partial\mathcal{D}}\eta z_i(\mathbf{c}^i)^2d\mathcal{S}\right)+a_i\|\Delta \mathbf{c}^i\|_{L^2}^2dt\nonumber\\
&\qquad\qquad\quad=\overline{\Phi}_R(u\cdot \nabla \mathcal{J}_\varepsilon\mathbf{c}^i, \Delta \mathbf{c}^i)dt+\widetilde{\Phi}_R(a_i{\rm div}(z_i\mathbf{c}^i\nabla \mathcal{J}_\varepsilon\Psi), -\Delta \mathbf{c}^i)dt.
\end{align}
Again using the H\"{o}lder inequality and the Gagliardo-Nirenberg inequality \eqref{e4.18}, we get
\begin{align}
\overline{\Phi}_R|(u\cdot \nabla \mathcal{J}_\varepsilon\mathbf{c}^i, \Delta \mathbf{c}^i)| &\leq \left\{\begin{array}{ll}
\!\!\!\frac{a_i}{6}\|\Delta \mathbf{c}^i\|_{L^2}^2+C\overline{\Phi}_R^2\|u\|_{\mathbb{L}^2}^2\|\nabla u\|_{\mathbb{L}^2}^2\|\nabla \mathbf{c}^i\|^2_{L^2},~ &{\rm if~ }~d=2  ,\\
\!\!\!\frac{a_i}{6}\|\Delta \mathbf{c}^i\|_{L^2}^2+C\overline{\Phi}_R^2\|\nabla u\|_{\mathbb{L}^2}^4\|\nabla \mathbf{c}^i\|^2_{L^2},~ & {\rm if~ }~d=3.
\end{array}\right. \label{e2.22}
\end{align}
Similarly,
\begin{align}
\widetilde{\Phi}_R|(a_iz_i\nabla \mathbf{c}^i\nabla\mathcal{J}_\varepsilon\Psi, -\Delta \mathbf{c}^i)|&\leq C\widetilde{\Phi}_R\|\nabla \Psi\|_{L^\infty}\|\nabla \mathbf{c}^i\|_{L^2}\|\Delta \mathbf{c}^i\|_{L^2}\nonumber\\
&\leq  \frac{a_i}{6}\|\Delta \mathbf{c}^i\|_{L^2}^2+C\widetilde{\Phi}_R^2\|\nabla \Psi\|^2_{L^\infty}\|\nabla \mathbf{c}^i\|_{L^2}^2,
\end{align}
and
\begin{align}\label{e4.19}
\widetilde{\Phi}_R|(a_iz_i\mathbf{c}^i\Delta\mathcal{J}_\varepsilon\Psi,-\Delta \mathbf{c}^i )|&\leq C\widetilde{\Phi}_R\|\Delta \mathbf{c}^i\|_{L^2}\|\mathbf{c}^i\|_{L^6}\|\Delta \Psi\|_{L^3}\nonumber\\
&\leq \frac{a_i}{6}\|\Delta \mathbf{c}^i\|_{L^2}^2+C\widetilde{\Phi}_R^2\|\Delta \Psi\|^2_{L^3}(\|\nabla \mathbf{c}^i\|_{L^2}^2+\|\mathbf{c}^i\|_{L^2}^2).
\end{align}
Collecting \eqref{e4.16}-\eqref{e4.19}, we arrive at
\begin{align*}
&\frac{1}{2}d\left(\|\nabla\mathbf{c}^i\|_{L^2}^2+\int_{\partial\mathcal{D}}\eta z_i(\mathbf{c}^i)^2d\mathcal{S}\right)+a_i\|\Delta \mathbf{c}^i\|_{L^2}^2dt\\
&\leq \frac{a_i}{2}\|\Delta \mathbf{c}^i\|_{L^2}^2dt+C\widetilde{\Phi}_R^2(R^4+\|\nabla \Psi\|^2_{L^\infty}+\|\Delta \Psi\|^2_{L^3})\|\nabla \mathbf{c}^i\|_{L^2}^2dt\\
&\quad+C\widetilde{\Phi}_R^2(R^4+\|\nabla \Psi\|^2_{L^\infty}+\|\Delta \Psi\|^2_{L^3})\|\mathbf{c}^i\|_{L^2}^2dt\nonumber\\
&\leq \frac{a_i}{2}\|\Delta \mathbf{c}^i\|_{L^2}^2dt+C(1+R^4)\|\nabla \mathbf{c}^i\|_{L^2}^2dt+C(1+R^4)\|\mathbf{c}^i\|_{L^2}^2dt.
\end{align*}
Taking integral of time $t$, and then supremum over $t\in[0,T]$,  by Lemma \ref{lem2.1*}, we have
\begin{align}\label{e4.20}
&\frac{1}{2}\sup_{t\in [0, T]}\|\nabla\mathbf{c}^i\|_{L^2}^2+\int_0^T \frac{a_i}{2}\|\Delta \mathbf{c}^i\|_{L^2}^2dt\nonumber\\
&\leq \|\nabla\mathbf{c}^i_0\|_{L^2}^2+\int_{\partial\mathcal{D}}\eta z_i(\mathbf{c}^i_0)^2d\mathcal{S}+C\int_0^T (1+R^4)\|\nabla \mathbf{c}^i\|_{L^2}^2dt\nonumber\\&\quad+C\int_0^T (1+R^4)\|\mathbf{c}^i\|_{L^2}^2dt+\frac{1}{2}\sup_{t\in [0, T]}\left|\int_{\partial\mathcal{D}}\eta z_i(\mathbf{c}^i)^2d\mathcal{S}\right|\nonumber\\
&\leq \|\nabla\mathbf{c}^i_0\|_{L^2}^2+\|\eta\|_{L^\infty}|z_i|\|\mathbf{c}^i_0\|_{L^2(\partial\mathcal{D})}^2+C\int_0^T (1+R^4)\|\nabla \mathbf{c}^i\|_{L^2}^2dt\\&\quad+C\int_0^T (1+R^4)\|\mathbf{c}^i\|_{L^2}^2dt+\frac{1}{2}\|\eta\|_{L^\infty}|z_i|\sup_{t\in [0, T]}\|\mathbf{c}^i\|^2_{L^2(\partial\mathcal{D})}\nonumber\\
&\leq \|\nabla\mathbf{c}^i_0\|_{L^2}^2+\|\eta\|_{L^\infty}|z_i|\|\mathbf{c}^i_0\|_{L^2(\partial\mathcal{D})}^2+C\int_0^T (1+R^4)\|\nabla \mathbf{c}^i\|_{L^2}^2dt\nonumber\\&\quad+C\int_0^T (1+R^4)\|\mathbf{c}^i\|_{L^2}^2dt+\frac{1}{2}\|\eta\|_{L^\infty}|z_i|\sup_{t\in [0, T]}(\delta\|\nabla\mathbf{c}^i\|^2_{L^2}+C(\delta)\|\mathbf{c}^i\|^2_{L^1}),\nonumber
\end{align}
choosing $\delta$ small enough such that
$$\delta\|\eta\|_{L^\infty}|z_i|\leq \frac{1}{2},$$
as well as \eqref{e1.8}, then we have from \eqref{e4.20}
\begin{align}\label{e4.21}
&\frac{1}{4}\sup_{t\in [0, T]}\|\nabla\mathbf{c}^i\|_{L^2}^2+\int_0^T \frac{a_i}{2}\|\Delta \mathbf{c}^i\|_{L^2}^2dt\nonumber\\
&\leq \|\nabla\mathbf{c}^i_0\|_{L^2}^2+\|\eta\|_{L^\infty}|z_i|\|\mathbf{c}^i_0\|_{L^2(\partial\mathcal{D})}^2+C\int_0^T (1+R^4)\|\nabla \mathbf{c}^i\|_{L^2}^2dt\nonumber\\&\quad+C\int_0^T (1+R^4)\|\mathbf{c}^i\|_{L^2}^2dt+\frac{1}{2}C(\delta)\|\eta\|_{L^\infty}|z_i|\sup_{t\in [0, T]}\|\mathbf{c}^i\|^2_{L^1}\nonumber\\
&\leq \|\nabla\mathbf{c}^i_0\|_{L^2}^2+\|\eta\|_{L^\infty}|z_i|\|\mathbf{c}^i_0\|_{L^2(\partial\mathcal{D})}^2+C\int_0^T (1+R^4)\|\nabla \mathbf{c}^i\|_{L^2}^2dt\nonumber\\&\quad+C\int_0^T (1+R^4)\|\mathbf{c}^i\|_{L^2}^2dt+\frac{1}{2}C(\delta)\|\eta\|_{L^\infty}|z_i|\|\mathbf{c}^i_0\|^2_{L^1}.
\end{align}
Taking power $p$ and taking expectation in \eqref{e4.21}, using the Gronwall lemma,  \eqref{e3.13}, we conclude
\begin{align*}
\mathbb{E}\left(\sup_{t\in [0,T]}\|\nabla\mathbf{c}^i\|_{L^2}^{2p}\right)+\mathbb{E}\left(\int_{0}^{T}a_i\|\Delta \mathbf{c}^i\|_{L^2}^2dt\right)^p\leq C,
\end{align*}
where $C$ depends on $R,T$ and initial data but independent of $n$.
\end{proof}

Again, using the elliptic regularity theory, from \eqref{e1.3} and \eqref{e2.16}, we can infer that
\begin{align}
\Psi\in L^p(\Omega;L^\infty(0,T;H^3)\cap L^2(0,T; H^4)).
\end{align}

\subsubsection{Stochastic compactness} Similar to the stochastic compactness argument in the previous section, we also define the measure set and the path space as follows:
$$\widetilde{\pounds}=\widetilde{\pounds}^u\ast\widetilde{\pounds}^{\mathbf{c}^i}\ast\widetilde{\pounds}^\Psi\ast \widetilde{\pounds}^\mathcal{W},$$
where $\widetilde{\pounds}^u$ is the law of $u$ on the path space
$$\widetilde{\mathbb{X}}_u:=L_{w}^{2}(0,T;H^{2}(\mathcal{D}))\cap L^2(0,T; \mathbb{H}^1(\mathcal{D})),$$
$\widetilde{\pounds}^{\mathbf{c}^i}$ is the law of $\mathbf{c}^i$ on the path space
$$\widetilde{\mathbb{X}}_{\mathbf{c}^i}:=L_{w}^{2}(0,T;H^{2}(\mathcal{D}))\cap L^2(0,T; H^1(\mathcal{D})),$$
$\widetilde{\pounds}^{\Psi}$ is the law of $\Psi$ on the path space
$$\widetilde{\mathbb{X}}_{\Psi}:=L_{w}^{2}(0,T;H^{4}(\mathcal{D}))\cap L^2(0,T; H^3(\mathcal{D}))\cap C([0,T];H^2(\mathcal{D})),$$
$\widetilde{\pounds}^{\mathcal{W}}$ is the law of $\mathcal{W}$ on the path space
$$\widetilde{\mathbb{X}}_{\mathcal{W}}:=C([0,T];\mathcal{H}_0).$$
Let
$$\widetilde{\mathbb{X}}:=\widetilde{\mathbb{X}}_u\times \widetilde{\mathbb{X}}_{\mathbf{c}^i}\times\widetilde{\mathbb{X}}_{\Psi}\times\widetilde{\mathbb{X}}_{\mathcal{W}}.$$

\begin{corollary}\label{cor1}
The measure set $\big\{\widetilde{\pounds}^n\big\}_{n\geq 1}$ is tight on path space $\widetilde{\mathbb{X}}$.
\end{corollary}
\begin{proof}
  The proof is totally same with Lemma \ref{lem3.4}, we omit the details.
\end{proof}

Following Corollary \ref{cor1}, applying the Skorokhod-Jakubowski theorem \ref{thm6.2}, we also have the stochastic compactness result as Proposition \ref{pro3.1}:
 There exist a subsequence $\big\{\widetilde{\pounds}^{n_{k}}\big\}_{k\geq 1}$, a new probability space $(\tilde{\Omega},\tilde{\mathcal{F}},\tilde{\mathbb{P}})$ and new $\widetilde{\mathbb{X}}$-valued measurable random variables sequence $(\tilde{u}_{n_{k}},\tilde{\mathbf{c}}^i_{n_{k}}, \widetilde{\Psi}_{n_{k}}, \widetilde{\mathcal{W}}_{n_{k}})$ and $(\tilde{u},\tilde{\mathbf{c}}^i,\widetilde{\Psi}, \widetilde{\mathcal{W}})$ such that

{\rm i}. $(\tilde{u}_{n_{k}}, \tilde{\mathbf{c}}^i_{n_{k}}, \widetilde{\Psi}_{n_{k}}, \widetilde{\mathcal{W}}_{n_{k}})\rightarrow(\tilde{u},\tilde{\mathbf{c}}^i, \widetilde{\Psi}, \widetilde{\mathcal{W}})$ $\tilde{\mathbb{P}}$-a.s. in the topology of $\widetilde{\mathbb{X}}$,

{\rm ii}. the laws of $(\tilde{u}_{n_{k}},\tilde{\mathbf{c}}^i_{n_{k}}, \widetilde{\Psi}_{n_{k}}, \widetilde{\mathcal{W}}_{n_{k}})$ and $(\tilde{u},\tilde{\mathbf{c}}^i, \widetilde{\Psi}, \widetilde{\mathcal{W}})$ are given by $\big\{\widetilde{\pounds}^{n_{k}}\big\}_{k\geq 1}$ and $\widetilde{\pounds}$, respectively,

{\rm iii}. $\widetilde{\mathcal{\mathcal{W}}}_{n_{k}}$ is a Wiener process relative to the filtration $\tilde{\mathcal{F}}_{t}^{n_{k}}=\sigma(\tilde{u}_{n_{k}},\tilde{\mathbf{c}}_{n_k}^i,\widetilde{\mathcal{W}}_{n_{k}})$.

In the following, we still use $(\tilde{u}_{n},\tilde{\mathbf{c}}^i_{n}, \widetilde{\Psi}_{n}, \widetilde{\mathcal{W}}_{n})$ standing for the new subsequence.
As a result of ii, the sequence $(\tilde{u}_n, \tilde{\mathbf{c}}^i_n, \widetilde{\Psi}_n)$ also shares the following bounds:
\begin{align}
&\widetilde{\mathbb{E}}\left(\sup_{t\in [0,T]}\|\tilde{\mathbf{c}}_n^i\|_{H^1}^{2p}\right)+\widetilde{\mathbb{E}}\left(\int_{0}^{T}\| \tilde{\mathbf{c}}^i_n\|^2_{H^2}dt\right)^p\leq C,\label{b2.25}\\
&\widetilde{\mathbb{E}}\left(\sup_{t\in [0,T]}\|\tilde{u}_n\|_{\mathbb{H}^1}^{2p}\right)+\widetilde{\mathbb{E}}\left(\int_{0}^{T}\| \tilde{u}_n\|^2_{H^2}dt\right)^p\leq C,\label{b2.26}\\
&\widetilde{\mathbb{E}}\left(\sup_{t\in [0,T]}\|\widetilde{\Psi}_n\|_{H^3}^{2p}\right)+\widetilde{\mathbb{E}}\left(\int_{0}^{T}\|\widetilde{\Psi}_n\|^2_{H^4}dt\right)^p\leq C, \label{b2.27}
\end{align}
where the constant $C$ is independent of $n$.

\subsubsection{Identify the limit} Here, we only consider the truncate nonlinear terms. Choose $\varphi\in \mathbb{L}^2$, and decompose
\begin{align}\label{4.31}
&\int_{0}^{t}(\Phi_R(\|\nabla \tilde{u}_n\|_{\mathbb{L}^2})(\tilde{u}_n\cdot \nabla)\mathcal{J}_{\varepsilon(n)}\tilde{u}_n-\Phi_R(\|\nabla \tilde{u}\|_{\mathbb{L}^2})(\tilde{u}\cdot \nabla)\tilde{u},\varphi)ds\nonumber\\
&=\int_{0}^{t}(\Phi_R(\|\nabla \tilde{u}_n\|_{\mathbb{L}^2})-\Phi_R(\|\nabla \tilde{u}\|_{\mathbb{L}^2}))(\tilde{u}_n\cdot \nabla\mathcal{J}_{\varepsilon(n)}\tilde{u}_n,\varphi)ds\nonumber\\
&\quad+\int_{0}^{t}\Phi_R(\|\nabla \tilde{u}\|_{\mathbb{L}^2})((\tilde{u}_n-\tilde{u})\cdot \nabla \mathcal{J}_{\varepsilon(n)}\tilde{u}_n,\varphi)ds\nonumber\\
&\quad+\int_{0}^{t}\Phi_R(\|\nabla \tilde{u}\|_{\mathbb{L}^2})(\tilde{u}\cdot \nabla(\mathcal{J}_{\varepsilon(n)}\tilde{u}_n-\tilde{u}),\varphi)ds\nonumber\\
&=:J_1+J_2+J_3.
\end{align}
Employing the smoothness of $\Phi_R$ and the mean-value theorem, the embedding $H^1\hookrightarrow L^6$, and the H\"{o}lder inequality, Lemma \ref{lem4.4}, we get
\begin{align*}
 |J_1|&\leq \|\varphi\|_{\mathbb{L}^2}\|\nabla \tilde{u}_n-\nabla \tilde{u}\|_{L^2(0,T;\mathbb{L}^2)}\|\tilde{u}_n\|_{L^\infty(0,T;\mathbb{L}^6)}\|\nabla \tilde{u}_n\|_{L^2(0,T;\mathbb{L}^3)}\nonumber\\
 &\leq \|\varphi\|_{\mathbb{L}^2}\|\nabla \tilde{u}_n-\nabla \tilde{u}\|_{L^2(0,T;\mathbb{L}^2)}\|\tilde{u}_n\|_{L^\infty(0,T;\mathbb{H}^1)}\|\nabla \tilde{u}_n\|_{L^2(0,T;\mathbb{H}^1)}.
\end{align*}
Applying the fact that $\tilde{u}_n\rightarrow \tilde{u}$ in $L^2(0,T;\mathbb{H}^1)$, $\tilde{\mathbb{P}}$-a.s. and \eqref{b2.26}, we get
\begin{align}\label{4.32}
 |J_1|\rightarrow 0,~ {\rm as}~ n\rightarrow \infty,~ \tilde{\mathbb{P}}\mbox{-a.s.}
\end{align}
{Also, $\tilde{u}_n\rightarrow \tilde{u}$ in $L^2(0,T;\mathbb{H}^1)$}, and weakly in $L^2(0,T; H^2)$, $\tilde{\mathbb{P}}$-a.s. and the H\"{o}lder inequality lead to $\tilde{\mathbb{P}}$-a.s.
\begin{align}\label{4.33}
\lim_{n\rightarrow \infty}|J_2+J_3|=0.
\end{align}
From \eqref{4.32}-\eqref{4.33}, we deduce that the right-hand side of \eqref{4.31} goes to zero. Similarly, we can show that for $\phi\in L^2$, it holds that $\tilde{\mathbb{P}}$-a.s.
$$\lim_{n\rightarrow \infty}\int_{0}^{t}(\Phi_R(\|\nabla \tilde{u}_n\|_{\mathbb{L}^2})(\tilde{u}_n\cdot \nabla)\mathcal{J}_{\varepsilon(n)}\tilde{\mathbf{c}}^i_n-\Phi_R(\|\nabla \tilde{u}\|_{\mathbb{L}^2})(\tilde{u}\cdot \nabla)\tilde{\mathbf{c}}^i,\phi)ds=0.$$

Decompose
\begin{align}\label{4.34}
&\int_{0}^{t}(a_i{\rm div}(\Phi_R(\|\nabla \widetilde{\Psi}_n\|_{W^{1,3^+}})z_i\tilde{\mathbf{c}}_n^i\nabla\mathcal{J}_{\varepsilon(n)}\widetilde{\Psi}_n)-a_i{\rm div}(\Phi_R(\|\nabla \widetilde{\Psi}\|_{W^{1,3^+}})z_i\tilde{\mathbf{c}}^i\nabla\widetilde{\Psi}),\phi)ds\nonumber\\
&=\int_{0}^{t}(\Phi_R(\|\nabla \widetilde{\Psi}_n\|_{W^{1,3^+}})-\Phi_R(\|\nabla \widetilde{\Psi}\|_{W^{1,3^+}}))(a_i{\rm div}(z_i\tilde{\mathbf{c}}_n^i\nabla\mathcal{J}_{\varepsilon(n)}\widetilde{\Psi}_n),\phi)ds\nonumber\\
&\quad+\int_{0}^{t}a_i\Phi_R(\|\nabla \widetilde{\Psi}\|_{W^{1,3^+}})({\rm div}(z_i(\tilde{\mathbf{c}}_n^i-\tilde{\mathbf{c}}^i)\nabla\mathcal{J}_{\varepsilon(n)}\widetilde{\Psi}_n),\phi)ds\nonumber\\
&\quad+\int_{0}^{t}a_i\Phi_R(\|\nabla \widetilde{\Psi}\|_{W^{1,3^+}})({\rm div}(z_i\tilde{\mathbf{c}}^i\nabla(\mathcal{J}_{\varepsilon(n)}\widetilde{\Psi}_n-\widetilde{\Psi})),\phi)ds\nonumber\\
&=:\widetilde{J}_1+\widetilde{J}_2+\widetilde{J}_3.
\end{align}
By using the mean-value theorem, the embedding $H^1\hookrightarrow L^6, H^2\hookrightarrow W^{1,3^+}$, the convergence that $\Psi_n\rightarrow \Psi$ in $L^2(0,T;H^3)$, $\tilde{\mathbb{P}}$-a.s. and the H\"{o}lder inequality, Lemma \ref{lem4.4}, estimates \eqref{b2.25}, \eqref{b2.27}, we obtain
\begin{align}\label{4.35}
|\widetilde{J}_1|&\leq C\|\phi\|_{L^2}\|\nabla \widetilde{\Psi}_n-\nabla \widetilde{\Psi}\|_{L^2(0,T;H^2)}\big(\|\nabla\tilde{\mathbf{c}}_n^i\|_{L^2(0,T;L^6)}\|\nabla\widetilde{\Psi}_n\|_{L^\infty(0,T;L^3)}\nonumber\\
&\quad+\|\tilde{\mathbf{c}}_n^i\|_{L^\infty(0,T;L^6)}\|\Delta\widetilde{\Psi}_n\|_{L^2(0,T;L^3)}\big)\nonumber\\
&\leq C\|\phi\|_{L^2}\|\nabla \widetilde{\Psi}_n-\nabla \widetilde{\Psi}\|_{L^2(0,T;H^2)}\big(\|\tilde{\mathbf{c}}_n^i\|_{L^2(0,T;H^2)}\|\widetilde{\Psi}_n\|_{L^\infty(0,T;H^2)}\nonumber\\
&\quad+\|\tilde{\mathbf{c}}_n^i\|_{L^\infty(0,T;H^1)}\|\Delta\widetilde{\Psi}_n\|_{L^2(0,T;L^3)}\big)
\rightarrow 0,~ {\rm as}~ n\rightarrow \infty,~ \tilde{\mathbb{P}}\mbox{-a.s.}
\end{align}
and
\begin{align}\label{4.36}
|\widetilde{J}_3|\leq C\|\phi\|_{L^2}\|\Delta (\mathcal{J}_{\varepsilon(n)}\widetilde{\Psi}_n- \widetilde{\Psi})\|_{L^2(0,T;H^1)}\|\tilde{\mathbf{c}}^i\|_{L^\infty(0,T;H^1)}
\rightarrow 0,~ {\rm as}~ n\rightarrow \infty,~ \tilde{\mathbb{P}}\mbox{-a.s.}
\end{align}
Using the embedding $ H^2\hookrightarrow L^\infty$ and the H\"{o}lder inequality, we get
\begin{align*}
|\widetilde{J}_2|&\leq C\|\phi\|_{L^2}\|\tilde{\mathbf{c}}_n^i-\tilde{\mathbf{c}}^i\|_{L^2(0,T;H^1)}(\|\nabla \widetilde{\Psi}_n\|_{L^2(0,T;L^\infty)}+\|\Delta \widetilde{\Psi}_n\|_{L^2(0,T;L^3)}),
\end{align*}
then the convergence that $\tilde{\mathbf{c}}_n^i\rightarrow\tilde{\mathbf{c}}^i$ in $L^2(0,T;H^1)$, $\tilde{\mathbb{P}}$-a.s. and bound \eqref{b2.27} imply
\begin{align}\label{4.37}
|\widetilde{J}_2|\rightarrow 0,~ {\rm as}~ n\rightarrow \infty,~ \tilde{\mathbb{P}}\mbox{-a.s.}
\end{align}
Combining \eqref{4.34}-\eqref{4.37}, we have $\tilde{\mathbb{P}}$-a.s.
\begin{align*}
&\lim_{n\rightarrow\infty}\int_{0}^{t}(a_i{\rm div}(\Phi_R(\|\nabla \widetilde{\Psi}_n\|_{W^{1,3^+}})z_i\tilde{\mathbf{c}}_n^i\nabla\mathcal{J}_{\varepsilon(n)}\widetilde{\Psi}_n)\\  &\qquad\qquad\qquad\qquad-a_i{\rm div}(\Phi_R(\|\nabla \widetilde{\Psi}\|_{W^{1,3^+}})z_i\tilde{\mathbf{c}}^i\nabla\widetilde{\Psi}),\phi)ds=0.
\end{align*}
Also, by a similar way, we have for $\varphi\in \mathbb{L}^2$, $\tilde{\mathbb{P}}$-a.s.
\begin{align*}
&\lim_{n\rightarrow\infty}\int_{0}^{t}(\kappa\Phi_R(\|\nabla \widetilde{\Psi}_n\|_{W^{1,3^+}})\tilde{\rho}_n\nabla \mathcal{J}_{\varepsilon(n)}\widetilde{\Psi}_n-\kappa\Phi_R(\|\nabla \widetilde{\Psi}\|_{W^{1,3^+}})\tilde{\rho}\nabla \widetilde{\Psi},\varphi)ds=0.
\end{align*}
Keeping the same line as Step 3.3 in section 3, we could pass the limit in stochastic integral term applying assumption \eqref{as4} and Lemma \ref{lem4.6}.

Having everything in hand, we could build the global existence of martingale solution to system \eqref{M1.1} in the sense of Definition \ref{def4.3}, which is strong in PDEs sense but weak in probability sense, for the detailed proof, see Step 3.3 in section 3. Introduce a stopping time
\begin{align*}
\tau_R:=\tau^1_R\wedge\tau_R^2,
\end{align*}
where
\begin{align*}
&\tau^1_R:=\left\{t\geq 0; \sup_{s\in [0,t]}\|\nabla u\|_{\mathbb{L}^2}\geq R\right\},\\
&\tau^2_R:=\left\{t\geq 0; \sup_{s\in [0,t]}\|\nabla \Psi\|_{W^{1,3^+}}\geq R\right\}.
\end{align*}
Then, we could show that $(u, \mathbf{c}^i, \tau_R)$ is a local strong martingale solution to system \eqref{M1.1}. From the system and the estimates itself, we could infer that, for $\tilde{\mathbb{P}}$ almost all $\tilde{w}\in \tilde{\Omega}$, the trajectory of $(u, \mathbf{c}^i)$ is almost everywhere equal to a continuous $\mathbb{H}^1\times H^1$-valued functions, we give the details of proof later.  As in section 3, we could extend the local solution to the maximal solution $(u, \mathbf{c}^i, \tau)$ where $\tau$ is an accessible stopping time,
and on the set $\{\tau<\infty\}$,
\begin{align*}
&\widetilde{\mathbb{E}}\left(\sup_{t\in [0,T\wedge\tau_R]}\|\tilde{\mathbf{c}}_n^i\|_{H^1}^{2p}\right)+\widetilde{\mathbb{E}}\left(\int_{0}^{T\wedge\tau_R}\| \tilde{\mathbf{c}}^i_n\|^2_{H^2}dt\right)^p\leq C(R),\\
&\widetilde{\mathbb{E}}\left(\sup_{t\in [0,T\wedge\tau_R]}\|\tilde{u}_n\|_{\mathbb{H}^1}^{2p}\right)+\widetilde{\mathbb{E}}\left(\int_{0}^{T\wedge\tau_R}\| \tilde{u}_n\|^2_{H^2}dt\right)^p\leq C(R),\\
&\widetilde{\mathbb{E}}\left(\sup_{t\in [0,T\wedge\tau_R]}\|\widetilde{\Psi}_n\|_{H^3}^{2p}\right)+\widetilde{\mathbb{E}}\left(\int_{0}^{T\wedge\tau_R}\|\widetilde{\Psi}_n\|^2_{H^4}dt\right)^p\leq C(R),
\end{align*}
but the above bounds will blow-up in $[0, T\wedge\tau]$.

\subsection{Uniqueness}

In this section, we formulate the following uniqueness result.
\begin{proposition}\label{lem2.3}
The uniqueness of the solution holds in the following sense:
 $(u_{1},\mathbf{c}^i_{1},\tau_{1})$ and $(u_{2},\mathbf{c}^i_{2},\tau_{2})$ are two local strong solutions of system \eqref{Equ1.1}-\eqref{e1.7}, with  $$\mathbb{P}\left\{(u_{1}(0),\mathbf{c}^i_{1}(0))=(u_{2}(0),\mathbf{c}^i_{2}(0))\right\}=1,$$ then,
\begin{eqnarray*}
\mathbb{P}\left\{(u_{1}(t,x),\mathbf{c}^i_{1}(t,x))=(u_{2}(t,x), \mathbf{c}^i_{2}(t,x));\forall t\in[0,\tau_{1}\wedge\tau_{2}]\right\}=1.
\end{eqnarray*}
\end{proposition}

\begin{proof} The difference of two solutions $u=u_1-u_2, \mathbf{c}^i=\mathbf{c}^i_1-\mathbf{c}^i_2, \Psi=\Psi_1-\Psi_2$ satisfy
\begin{align}\label{e2.42}
\left\{\begin{array}{ll}
\!\!\!\partial_t \mathbf{c}^i+(u\cdot \nabla) \mathbf{c}^i_1+(u_2\cdot\nabla)\mathbf{c}^i=a_i{\rm div}(\nabla \mathbf{c}^i+z_i\mathbf{c}_i\nabla \Psi_1+\mathbf{c}_2^i\nabla\Psi),\\
\!\!\!\partial_t u+P(u\cdot \nabla)u_1+P(u_2\cdot\nabla)u=\mu P\Delta u-\kappa P \sum_{i=1}^{m}z_i\mathbf{c}^i\nabla\Psi_1-\kappa P\rho_2\nabla \Psi\\ \qquad\qquad\qquad\qquad\qquad\qquad\qquad+P(f(u_1,\nabla\Psi_1)-f(u_2,\nabla\Psi_2))\frac{d\mathcal{W}}{dt},\\
\!\!\!-\Delta \Psi=\sum_{i=1}^mz_i\mathbf{c}^i, i=1,2,\cdots, m,\\
\end{array}\right.
\end{align}
with the initial data $(u_0, \mathbf{c}_0^i)=(0,0)$ and the boundary conditions
\begin{align}\label{e4.34}
\left\{\begin{array}{ll}
\!\!\!u|_{\partial\mathcal{D}}=0,\\
\!\!\!(\partial_n\mathbf{c}^i+z_i(\mathbf{c}_1^i\partial_n\Psi_1-\mathbf{c}_2^i\partial_n\Psi_2))|_{\partial\mathcal{D}}=0,\\
\!\!\!\partial_n\Psi|_{\partial\mathcal{D}}=0.\\
\end{array}\right.
\end{align}

Taking inner product with $\mathbf{c}^i$ in equation $\eqref{e2.42}_1$, and using the fact that $(u_2\cdot\nabla \mathbf{c}^i, \mathbf{c}^i)=0$, the boundary condition $\eqref{e4.34}_2$, we have
\begin{align}\label{E2.42*}
\frac{1}{2}\partial_t \|\mathbf{c}^i\|^2_{L^2}+(u\cdot \nabla \mathbf{c}^i_1,\mathbf{c}^i)&=(a_i{\rm div}(\nabla \mathbf{c}^i+z_i\mathbf{c}^i\nabla \Psi_1+z_i\mathbf{c}_2^i\nabla\Psi), \mathbf{c}^i)\nonumber\\
&=-a_i\|\nabla\mathbf{c}^i\|_{L^2}^2-a_i(z_i\mathbf{c}^i\nabla \Psi_1+\mathbf{c}_2^i\nabla\Psi, \nabla\mathbf{c}^i)\nonumber\\
&\quad+\int_{\partial\mathcal{D}}a_i(\partial_n\mathbf{c}^i+ z_i\mathbf{c}^i\partial_n\Psi_1+z_i\mathbf{c}_2^i\partial_n\Psi)\mathbf{c}^i  d\mathcal{S}\nonumber\\
&=-a_i\|\nabla\mathbf{c}^i\|_{L^2}^2-a_i(z_i\mathbf{c}^i\nabla \Psi_1+\mathbf{c}_2^i\nabla\Psi, \nabla\mathbf{c}^i).
\end{align}
Using the It\^{o} formula to the function $\frac{1}{2}\|u\|_{\mathbb{L}^2}^2$, also by $(u_2\cdot\nabla u, u)=0$, we get
\begin{align}\label{e2.43}
&\frac{1}{2}\partial_t\|u\|_{\mathbb{L}^2}^2+\mu\|\nabla u\|_{\mathbb{L}^2}^2+(u\cdot \nabla u_1, u)=-\left(\kappa \sum_{i=1}^{m}z_i\mathbf{c}^i\nabla\Psi_1+\kappa\rho_2\nabla \Psi, u\right)\nonumber\\&+(f(u_1,\nabla\Psi_1)-f(u_2,\nabla\Psi_2), u)\frac{d\mathcal{W}}{dt}+\frac{1}{2}\|Pf(u_1,\nabla\Psi_1)-Pf(u_2,\nabla\Psi_2)\|^2_{L_2(\mathcal{H};\mathbb{L}^2)}.
\end{align}
Using the H\"{o}lder inequality and the embedding $H^1\hookrightarrow L^6$, one has
\begin{align}
&|-a_i(z_i\mathbf{c}^i\nabla \Psi_1+\mathbf{c}_2^i\nabla\Psi, \nabla\mathbf{c}^i)|\nonumber\\
&\leq C\|\nabla\mathbf{c}^i\|_{L^2}(\|\mathbf{c}^i\|_{L^2}\|\nabla \Psi_1\|_{L^\infty}+\|\mathbf{c}_2^i\|_{L^3}\|\nabla\Psi\|_{L^6})\nonumber\\
&\leq C\|\nabla\mathbf{c}^i\|_{L^2}(\|\mathbf{c}^i\|_{L^2}\|\nabla \Psi_1\|_{L^\infty}+\|\mathbf{c}_2^i\|_{L^3}\|\nabla\Psi\|_{H^1})\\
&\leq C\|\nabla\mathbf{c}^i\|_{L^2}(\|\mathbf{c}^i\|_{L^2}\|\nabla \Psi_1\|_{L^\infty}+\|\mathbf{c}_2^i\|_{L^3}\|\mathbf{c}^i\|_{L^2})\nonumber\\
&\leq \frac{a_i}{2}\|\nabla\mathbf{c}^i\|_{L^2}^2+C\|\mathbf{c}^i\|^2_{L^2}(\|\nabla \Psi_1\|^2_{L^\infty}+\|\mathbf{c}_2^i\|^2_{L^3}).\nonumber
\end{align}
Similarly, we have
\begin{align}
\begin{split}
|-(u\cdot \nabla \mathbf{c}^i_1,\mathbf{c}^i)|&\leq C\|u\|_{\mathbb{L}^6}\|\mathbf{c}^i\|_{L^2}\|\nabla \mathbf{c}^i_1\|_{L^3}\\
&\leq \frac{\mu}{4}\|\nabla u\|_{\mathbb{L}^2}^2+C\|\mathbf{c}^i\|_{L^2}^2\|\nabla \mathbf{c}^i_1\|_{L^3}^2,
\end{split}
\end{align}
and
\begin{align}
\begin{split}
|-(u\cdot \nabla u_1,u)|&\leq C\|u\|_{\mathbb{L}^6}\|u\|_{\mathbb{L}^2}\|\nabla u_1\|_{\mathbb{L}^3}\\
&\leq \frac{\mu}{4}\|\nabla u\|_{\mathbb{L}^2}^2+C\|u\|_{\mathbb{L}^2}^2\|\nabla u_1\|_{\mathbb{L}^3}^2.
\end{split}
\end{align}
Proceeding to estimate the right-hand side of \eqref{e2.43}, the H\"{o}lder inequality implies that
\begin{align}
&\left|-\left(\kappa \sum_{i=1}^{m}z_i\mathbf{c}^i\nabla\Psi_1+\kappa\rho_2\nabla \Psi, u\right)\right|\nonumber\\
&\leq C\|u\|_{\mathbb{L}^2}\left(\|\nabla\Psi_1\|_{L^\infty}\sum_{i=1}^m\|\mathbf{c}^i\|_{L^2}+\|\rho_2\|_{L^3}\|\nabla \Psi\|_{L^6}\right)\nonumber\\
&\leq C\|u\|_{\mathbb{L}^2}\left(\|\nabla\Psi_1\|_{L^\infty}\sum_{i=1}^m\|\mathbf{c}^i\|_{L^2}+\|\rho_2\|_{L^3}\|\nabla \Psi\|_{H^1}\right)\nonumber\\
&\leq  \frac{\mu}{4}\|u\|_{\mathbb{L}^2}^2+C\|\nabla\Psi_1\|^2_{L^\infty}\sum_{i=1}^m\|\mathbf{c}^i\|^2_{L^2}+C\|\rho_2\|^2_{L^3}\sum_{i=1}^m\|\mathbf{c}^i\|^2_{L^2}.
\end{align}
By assumption \eqref{as1}, we obtain
\begin{align}\label{e2.48}
\|Pf(u_1,\nabla\Psi_1)-Pf(u_2,\nabla\Psi_2)\|_{L_2(\mathcal{H};\mathbb{L}^2)}^2\leq \ell_1(\|u\|_{\mathbb{L}^2}^2+\|\mathbf{c}^i\|_{L^2}^2).
\end{align}
Combining \eqref{E2.42*}-\eqref{e2.48}, we arrive at
\begin{align*}
\frac{d}{dt}\left(\sum_{i=1}^m\|\mathbf{c}^i\|^2_{L^2}+\|u\|_{\mathbb{L}^2}^2\right)&+\mu\|\nabla u\|_{\mathbb{L}^2}^2+\sum_{i=1}^ma_i\|\nabla\mathbf{c}^i\|_{L^2}^2\leq C\sum_{i=1}^m\|\mathbf{c}^i\|^2_{L^2}(\|\nabla \Psi_1\|^2_{L^\infty}+\|\mathbf{c}_2^i\|^2_{L^3})\\
&+C\sum_{i=1}^m\|\mathbf{c}^i\|_{L^2}^2\|\nabla \mathbf{c}^i_1\|_{L^3}^2+C\|u\|_{\mathbb{L}^2}^2\|\nabla u_1\|_{L^3}^2\\
&+C\|\nabla\Psi_1\|^2_{L^\infty}\sum_{i=1}^m\|\mathbf{c}^i\|^2_{L^2}+C\|\rho_2\|^2_{L^3}\sum_{i=1}^m\|\mathbf{c}^i\|^2_{L^2}\\
&+(f(u_1,\nabla\Psi_1)-f(u_2,\nabla\Psi_2), u)\frac{d\mathcal{W}}{dt}.
\end{align*}
Let
$$\Lambda(t):=C(1+\|\nabla \Psi_1\|^2_{L^\infty}+\|\mathbf{c}_2^i\|^2_{L^3}+\|\nabla \mathbf{c}^i_1\|_{L^3}^2+\|\nabla u_1\|_{\mathbb{L}^3}^2).$$
Then, applying It\^o's product formula to the function
$${\rm exp}\left(\int_{0}^{t}-\Lambda(s)ds\right)\left(\sum_{i=1}^m\|\mathbf{c}^i\|^2_{L^2}+\|u\|_{\mathbb{L}^2}^2\right),$$
leads to
\begin{align*}
&d\left[{\rm exp}\left(\int_{0}^{t}-\Lambda(s)ds\right)\left(\sum_{i=1}^m\|\mathbf{c}^i\|^2_{L^2}+\|u\|_{\mathbb{L}^2}^2\right)\right]\\
&=-\Lambda(t){\rm exp}\left(\int_{0}^{t}-\Lambda(s)ds\right)\left(\sum_{i=1}^m\|\mathbf{c}^i\|^2_{L^2}+\|u\|_{\mathbb{L}^2}^2\right)dt\\
&\quad+{\rm exp}\left(\int_{0}^{t}-\Lambda(s)ds\right) d\left(\sum_{i=1}^m\|\mathbf{c}^i\|^2_{L^2}+\|u\|_{\mathbb{L}^2}^2\right)\\
&\leq C(f(u_1,\nabla\Psi_1)-f(u_2,\nabla\Psi_2), u)d\mathcal{W}\times{\rm exp}\left(\int_{0}^{t}-\Lambda(s)ds\right).
\end{align*}
Integrating of time and taking expectation in above yield
$$\mathbb{E}\left[{\rm exp}\left(\int_{0}^{t\wedge \tau_1\wedge\tau_2}-\Lambda(s)ds\right)\left(\sum_{i=1}^m\|\mathbf{c}^i\|^2_{L^2}+\|u\|_{\mathbb{L}^2}^2\right)\right]=0.$$
Here, the stochastic term vanishes due to the facts that it is a square-integral martingale, its expectation is zero.

Bounds \eqref{b2.25}-\eqref{b2.27} give
\begin{align*}
\int_{0}^{t\wedge \tau_1\wedge\tau_2}\Lambda(s)ds=\int_{0}^{t\wedge \tau_1\wedge\tau_2}\|\nabla \Psi_1\|^2_{L^\infty}+\|\mathbf{c}_2^i\|^2_{L^3}+\|\nabla \mathbf{c}^i_1\|_{L^3}^2+\|\nabla u_1\|_{\mathbb{L}^3}^2ds<\infty,~\mathbb{P}\mbox{-a.s.},
\end{align*}
therefore,
$${\rm exp}\left(\int_{0}^{t\wedge \tau_1\wedge\tau_2}-\Lambda(s)ds\right)>0,~ \mathbb{P}\mbox{-a.s.}$$
Finally, we conclude that
$$\mathbb{E}\left(\sum_{i=1}^m\|\mathbf{c}^i\|^2_{L^2}+\|u\|_{\mathbb{L}^2}^2\right)=0.$$
This completes the proof.
\end{proof}

\noindent\underline{Proof of Theorem \ref{thm}}.  To complete the proof of Theorem \ref{thm}, the Gy\"{o}ngy-Krylov characterization tells us that the approximate solutions converge in probability on the original fixed probability space once the pathwise uniqueness holds.
\begin{proposition}[\!\!\cite{Krylov}]\label{pro5.2}
Let $X$ be a complete separable metric space and suppose that $\{Y_{n}\}_{n\geq0}$ is a sequence of $X$-valued random variables on a probability space $(\Omega,\mathcal{F},\mathbb{P})$. Let $\{\lambda_{m,n}\}_{m,n\geq1}$ be the set of joint laws of $\{Y_{n}\}_{n\geq1}$, that is,
\begin{equation*}
\lambda_{m,n}(E):=\mathbb{P}\{(Y_{n},Y_{m})\in E\},~~~E\in\mathcal{B}(X\times X).
\end{equation*}
Then $\{Y_{n}\}_{n\geq1}$ converges in probability if and only if for every subsequence of the joint probability laws $\{\mu_{m_{k},n_{k}}\}_{k\geq1}$, there exists a further subsequence that converges weakly to a probability measure $\lambda$ such that
\begin{equation*}
\lambda\{(u,v)\in X\times X: u=v\}=1.
\end{equation*}
\end{proposition}

Define by $\pounds^{n,m}$ the joint law of $(u_n, \mathbf{c}_n^i, \Psi_n, u_m, \mathbf{c}_m^i, \Psi_m, \mathcal{W})$ on the path space $\widetilde{\mathbb{X}}$, where
$$\widetilde{\mathbb{X}}:=\widetilde{\mathbb{X}}_u\times \widetilde{\mathbb{X}}_{\mathbf{c}^i}\times\widetilde{\mathbb{X}}_{\Psi}\times\widetilde{\mathbb{X}}_u\times \widetilde{\mathbb{X}}_{\mathbf{c}^i}\times\widetilde{\mathbb{X}}_{\Psi}\times\widetilde{\mathbb{X}}_{\mathcal{W}},$$
and the two sequences $(u_n, \mathbf{c}_n^i, \Psi_n, u_m, \mathbf{c}_m^i, \Psi_m, \mathcal{W})$ are the approximate solutions of system \eqref{M1.1}.

\begin{corollary}
The measure set $\{\pounds^{n,m}\}_{n,m\geq 1}$  is tight on space $\widetilde{\mathbb{X}}$.
\end{corollary}

\begin{proof} The proof is completely analogous to Corollary \ref{cor1}, we omit the details.
\end{proof}

Both conditions of Proposition \ref{pro5.2} could be verified after using the Skorokhod representation theorem to measure set $\{\pounds^{n,m}\}_{n,m\geq 1}$ and Proposition \ref{lem2.3}, which implies that the sequence $(u_n, \mathbf{c}_n^i, \Psi_n)$ is convergent in probability on the original fixed probability space $\mathbf{S}$ under the topology of $\widetilde{\mathbb{X}}$, for the detailed proof, we recommend \cite[Lemma 5.1]{Debussche}.

By passing the limit, the argument is totally same with that of in the process of proving the existence of martingale solution in section 4.1.2,  we could deduce that the limit $(u, \mathbf{c}^i, \Psi)$ satisfies system \eqref{Equ1.1}-\eqref{e1.6}. The regularity estimates for $(u, \mathbf{c}^i, \Psi)$ are a result of the lower semi-continuity of norms.

We proceed to improve the time regularity which is needed to justify well-defined of stopping time $\tau_R$. Since $(u, \mathbf{c}^i)\in L^2(0,T; ( H^2\cap \mathbb{H}^1)\times H^2)$, by interpolation \cite[Chapter 3]{Teman}, we could deduce that $$(u, \mathbf{c}^i)\in C([0,T]; \mathbb{H}^1\times H^1)$$ once we have $$\frac{d}{dt}(u, \mathbf{c}^i)\in L^2(0,T; \mathbb{L}^2\times L^2).$$
Decompose $u=\bar{u}+v$, where $\bar{u}$ and $v$ satisfy the following systems respectively:
\begin{eqnarray}\label{sys}
\left\{\begin{array}{ll}
\!\!\!\partial_t\mathbf{c}^i+(\bar{u}+v)\cdot \nabla \mathbf{c}^i=a_i{\rm div}(\nabla \mathbf{c}^i+z_i\mathbf{c}^i\nabla\Psi),~i=1,2,\cdots, m,\\
\!\!\!-\Delta \Psi=\sum_{i=1}^m z_i\mathbf{c}^i=\rho,\\
\!\!\!\partial_t \bar{u}-P\mu\Delta \bar{u}+P((\bar{u}+v)\cdot \nabla)(\bar{u}+v)+ P\kappa\rho\nabla \Psi =0,\\
\end{array}\right.
\end{eqnarray}
and
\begin{eqnarray}\label{4.48}
\left\{\begin{array}{ll}
\!\!\!dv-P\Delta vdt=Pf(u, \nabla \Psi)d\mathcal{W},\\
\!\!\!v_0=0,
\end{array}\right.
\end{eqnarray}
Since $f\in L^p(\Omega; L^2(0,T; L_2(\mathcal{H};\mathbb{L}^2)))$, we have
\begin{eqnarray}\label{4.49}
v\in L^p(\Omega; C([0,T]; \mathbb{H}^1)\cap L^2(0,T;H^2)).
\end{eqnarray}
From the system \eqref{sys} itself and the regularity estimates of $(u, \mathbf{c}^i, v)$, we can infer
$$\frac{d}{dt}(\bar{u}, \mathbf{c}^i)\in L^2(0,T; \mathbb{L}^2\times L^2),~\mathbb{P}\mbox{-a.s.}$$
We finally obtain
$$(\bar{u}, \mathbf{c}^i)\in C([0,T]; \mathbb{H}^1\times H^1),~ \mathbb{P}\mbox{-a.s.}$$
This together with \eqref{4.49} yields that
$$(u, \mathbf{c}^i)\in C([0,T]; \mathbb{H}^1\times H^1), ~\mathbb{P}\mbox{-a.s.}$$
The proof of Theorem \ref{thm} is completed.

\subsection{Proof of Theorem \ref{thm1}} In this section, we show that the explosion occurs simultaneously. Following the idea of \cite[Theorem 17]{cri}, we first show that:
\begin{proposition}\label{pro4.4} For $\mathcal{R}, \widetilde{\mathcal{R}}\in \mathbb{Z}^+ $, define two stopping time
\begin{eqnarray}
&&\tau_{1,\mathcal{R}}:=\inf\left\{t\geq 0; \|u\|_{\mathbb{H}^1}\geq \mathcal{R}\right\}\wedge \inf\left\{t\geq 0; \|\mathbf{c}^i\|_{H^1}\geq \mathcal{R}\right\},\nonumber \\
&&\tau_{2,\widetilde{\mathcal{R}}}:=\inf\left\{t\geq 0; \|\nabla u\|_{\mathbb{L}^2}\geq \widetilde{\mathcal{R}}\right\}\wedge \inf\left\{t\geq 0; \|\nabla\Psi\|_{ W^{1,3^+}}\geq \widetilde{\mathcal{R}}\right\},\nonumber
\end{eqnarray}
let $\tau_{1}=\lim_{\mathcal{R}\rightarrow\infty}\tau_{1,\mathcal{R}}$, $\tau_{2}=\lim_{\mathcal{R}\rightarrow\infty}\tau_{2,\widetilde{\mathcal{R}}}$, then we have $\mathbb{P}$\mbox{-a.s.}
$$\tau_1=\tau_2.$$
\end{proposition}
\begin{proof} The proof is divided into two steps.

\underline{Step 1}. It holds $\tau_{1}\leq \tau_2$, $\mathbb{P}$-a.s. Indeed, $$\sup_{t\in[0, \tau_{1,\mathcal{R}}]}\|(\nabla u, \nabla \Psi)\|_{\mathbb{L}^2\times W^{1,3^+}}\leq C\sup_{t\in[0, \tau_{1,\mathcal{R}}]}\|(u, \mathbf{c}^i) \|_{\mathbb{H}^1\times H^{1}}\leq C\mathcal{R},$$
which implies $\tau_{1, \mathcal{R}}\leq \tau_{2, C\mathcal{R}}\leq \tau_2$, $\mathbb{P}$-a.s. Finally, passing $\mathcal{R}\rightarrow \infty$, we get the desired result.

\underline{Step 2}. It holds $\tau_{1}\geq \tau_2$, $\mathbb{P}$-a.s. To begin with, we show for any $k, \widetilde{\mathcal{R}}\in \mathbb{Z}^+$
\begin{align}\label{4.47}
\mathbb{P}\left\{\sup_{{\tiny t\in [0, \tau_{2,\widetilde{\mathcal{R}}}\wedge k]}}\|(u, \mathbf{c}^i)\|_{\mathbb{H}^1\times H^1}<\infty\right\}=1.
\end{align}
Notice that the regularity of $u, \mathbf{c}^i$ is insufficient to use the It\^{o} formula to $\|u\|^2_{\mathbb{H}^1}$, therefore we have to mollify the system as \eqref{M1.2}. Then using the same argument as Lemmas \ref{lem2.1} and \ref{lem2.2}, we can find that there exists a constant $C(k, \widetilde{\mathcal{R}})$ such that
$$\mathbb{E}\left(\sup_{{\small t\in[0, \tau_{2,\widetilde{\mathcal{R}}}\wedge k]}}\|(u, \mathbf{c}^i)\|^{2}_{\mathbb{H}^1\times H^1}\right)\leq C.$$
Hence, \eqref{4.47} follows.

Since
\begin{align*}
\mathbb{P}\{\tau_{2}\leq \tau_1\}&=\mathbb{P}\left\{\cap_{\widetilde{\mathcal{R}}\in \mathbb{Z}^+}\left\{\tau_{2, \widetilde{\mathcal{R}}}\leq \tau_1\right\}\right\}=\mathbb{P}\left\{\cap_{k,\widetilde{\mathcal{R}}\in \mathbb{Z}^+}\left\{\tau_{2, \widetilde{\mathcal{R}}}\wedge k\leq \tau_1\right\}\right\}\nonumber\\
&\leq \mathbb{P}\left\{\tau_{2, \widetilde{\mathcal{R}}}\wedge k\leq \tau_1\right\},
\end{align*}
as well as
\begin{align*}
\left\{\sup_{t\in [0, \tau_{2,\widetilde{\mathcal{R}}}\wedge k]}\|(u, \mathbf{c}^i)\|_{\mathbb{H}^1\times H^1}<\infty\right\}&=\bigcup_{\mathcal{R}\in \mathbb{Z}^+}\left\{\sup_{t\in [0, \tau_{2,\widetilde{\mathcal{R}}}\wedge k]}\|(u, \mathbf{c}^i)\|_{\mathbb{H}^1\times H^1}<\mathcal{R}\right\}\nonumber\\
&\subset\bigcup_{\mathcal{R}\in \mathbb{Z}^+}\left\{\tau_{2,\widetilde{\mathcal{R}}}\wedge k\leq \tau_{1, \mathcal{R}}\right\}\nonumber\\
&\subset\bigcup_{\mathcal{R}\in \mathbb{Z}^+}\left\{\tau_{2,\widetilde{\mathcal{R}}}\wedge k\leq \tau_{1}\right\},
\end{align*}
yields the desired result. The proof is completed.
\end{proof}

From the time continuity of $\|(u, \mathbf{c}^i)\|_{\mathbb{H}^1\times H^1}$ and the uniqueness, we know that $\tau_1=\tau_2$ is the maximal existence time, then Theorem \ref{thm1} follows from Proposition \ref{pro4.4}.

\section{Global strong pathwise solution for 2d case}

In this section, we show that the maximal strong pathwise solution turns out to be global one in $2D$ case. We need to show that the lifetime $\tau=\infty$, $\mathbb{P}\mbox{-a.s.}$, that is,
\begin{align}\label{e2.52}
\mathbb{P}\{\tau<\infty\}=0.
\end{align}

Before proving \eqref{e2.52}, we give a critical estimate as mentioned in Remark \ref{rem2.2}.
\begin{lemma}\label{lem2.4} Assume that assumption \eqref{as1} holds. The initial data $u_0$ is $\mathcal{F}_0$-measurable random variable with $u_0\in L^4(\Omega; \mathbb{L}^2)$. Then, the strong pathwise solution $(u, \tau)$ satisfies the following estimate:
\begin{align}\label{e2.53}
\mathbb{E}\left(\sup_{t\in [0,T\wedge \tau]}\|u\|_{\mathbb{L}^2}^4\right)+\mathbb{E}\int_{0}^{T\wedge\tau}\|\nabla u\|_{\mathbb{L}^2}^2\|u\|_{\mathbb{L}^2}^2dt\leq C,
\end{align}
where the constant $C$ is independent of $\tau$.
\end{lemma}
\begin{proof} Using the It\^{o} formula to the function $\left(\|u\|_{\mathbb{L}^2}^2\right)^2$, we obtain
\begin{align}\label{e2.54}
d\|u\|_{\mathbb{L}^2}^4&+4\|\nabla u\|_{\mathbb{L}^2}^2\|u\|_{\mathbb{L}^2}^2dt+4(\kappa\rho\nabla \Psi, u)\|u\|_{\mathbb{L}^2}^2dt= 4(f(u, \nabla\Psi), u)\|u\|_{\mathbb{L}^2}^2d\mathcal{W}\nonumber\\
&+4(f(u,\nabla\Psi), u)^2dt+2\|Pf(u, \nabla\Psi)\|_{L_2(\mathcal{H};\mathbb{L}^2)}^2\|u\|_{\mathbb{L}^2}^2dt.
\end{align}
Integrating of $t$, taking supremum over $t\in [0, T\wedge \tau]$ and expectation, we have
\begin{align}\label{e5.4}
\mathbb{E}\bigg(\sup_{t\in [0, T\wedge \tau]}\|u\|_{\mathbb{L}^2}^4\bigg)&+\mathbb{E}\int_{0}^{T\wedge \tau}\|\nabla u\|_{\mathbb{L}^2}^2\|u\|_{\mathbb{L}^2}^2dt\leq C\mathbb{E}\int_{0}^{T\wedge \tau}|(\kappa\rho\nabla \Psi, u)|\|u\|_{\mathbb{L}^2}^2dt\nonumber\\&+C\mathbb{E}\bigg(\sup_{t\in [0, T\wedge \tau]}\left|\int_{0}^{t}(f(u, \nabla\Psi), u)\|u\|_{\mathbb{L}^2}^2d\mathcal{W}\right|\bigg)\nonumber\\&
+C\mathbb{E}\int_{0}^{T\wedge \tau}\|Pf(u, \nabla\Psi)\|_{L_2(\mathcal{H};\mathbb{L}^2)}^2\|u\|_{\mathbb{L}^2}^2dt\nonumber\\
&+ C\mathbb{E}\int_{0}^{T\wedge \tau}(f(u,\nabla\Psi), u)^2dt.
\end{align}
Estimating the first term in the right-hand side by the H\"{o}lder inequality and the Agmon inequality $\|v\|_{L^\infty}\leq C\|v\|_{H^1}^\frac{1}{2}\|v\|_{H^2}^\frac{1}{2}$, we see
\begin{align}
&\mathbb{E}\int_{0}^{T\wedge \tau}|(\kappa\rho\nabla \Psi, u)|\|u\|_{\mathbb{L}^2}^2dt\nonumber\\ &\leq \mathbb{E}\int_{0}^{T\wedge \tau}\kappa\|\rho\|_{L^2}\|\nabla \Psi\|_{L^\infty}\|u\|_{\mathbb{L}^2}^3dt\nonumber\\
&\leq \frac{1}{4}\mathbb{E}\bigg(\sup_{t\in [0, T\wedge \tau]}\|u\|_{\mathbb{L}^2}^4\bigg)+C\mathbb{E}\int_{0}^{T\wedge \tau}\kappa\|\rho\|_{L^2}^4\|\nabla \Psi\|_{L^\infty}^4dt\nonumber\\
&\leq \frac{1}{4}\mathbb{E}\bigg(\sup_{t\in [0, T\wedge \tau]}\|u\|_{\mathbb{L}^2}^4\bigg)+C\mathbb{E}\int_{0}^{T\wedge \tau}\kappa\|\rho\|_{L^2}^4\|\nabla \Psi\|_{H^1}^2\|\nabla \Psi\|_{H^2}^2dt\nonumber\\
&\leq \frac{1}{4}\mathbb{E}\bigg(\sup_{t\in [0, T\wedge \tau]}\|u\|_{\mathbb{L}^2}^4\bigg)+C\left[\mathbb{E}\bigg(\sup_{t\in [0, T\wedge \tau]}\big(\kappa\|\rho\|_{L^2}^8\|\nabla \Psi\|_{H^1}^4\big)\bigg)\right]^\frac{1}{2}\nonumber\\
&\quad\times\left[\mathbb{E}\left(\int_{0}^{T\wedge \tau}\|\nabla \Psi\|_{H^2}^2dt\right)^2\right]^\frac{1}{2}\nonumber\\
&\leq C\left(\mathbb{E}\bigg(\sup_{t\in [0, T]}\|\nabla \Psi\|_{H^1}^{8}\bigg)\right)^\frac{1}{2}\cdot\left(\mathbb{E}\bigg(\sup_{t\in [0, T]}\|\rho\|_{L^2}^{16}\bigg)\right)^\frac{1}{2}\cdot\left[\mathbb{E}\left(\int_{0}^{T}\|\nabla \Psi\|_{H^2}^2dt\right)^2\right]^\frac{1}{2}\nonumber\\
&\quad+\frac{1}{4}\mathbb{E}\bigg(\sup_{t\in [0, T\wedge \tau]}\|u\|_{\mathbb{L}^2}^4\bigg).
\end{align}
The Burkholder-Davis-Gundy inequality as well as assumption \eqref{as1} yield
\begin{align}
&\mathbb{E}\left(\sup_{t\in [0, T\wedge \tau]}\left|\int_{0}^{t}(f(u,\nabla\Psi), u)\|u\|_{\mathbb{L}^2}^2d\mathcal{W}\right|\right)\nonumber\\
&\leq C\mathbb{E}\left(\int_{0}^{T\wedge \tau}\|f(u, \nabla\Psi)\|_{L_2(\mathcal{H};\mathbb{L}^2)}^2\|u\|_{\mathbb{L}^2}^6 dt\right)^\frac{1}{2}\nonumber\\
&\leq C\mathbb{E}\left(\int_{0}^{T\wedge \tau}(\|\mathbf{c}_i\|_{L^2}^2+\|u\|_{\mathbb{L}^2}^2)\|u\|_{\mathbb{L}^2}^6 dt\right)^\frac{1}{2}\nonumber\\
&\leq \frac{1}{4}\mathbb{E}\bigg(\sup_{t\in [0, T\wedge \tau]}\|u\|_{\mathbb{L}^2}^4\bigg)+C\mathbb{E}\int_{0}^{T\wedge \tau}\|\mathbf{c}_i\|_{L^2}^4+\|u\|_{\mathbb{L}^2}^4 dt.
\end{align}
Again, by assumption \eqref{as1}, we get
\begin{align}\label{e2.58}
\mathbb{E}\int_{0}^{T\wedge \tau}\|Pf(u, \nabla\Psi)\|_{L_2(\mathcal{H};\mathbb{L}^2)}^2\|u\|_{\mathbb{L}^2}^2dt\leq C\mathbb{E}\int_{0}^{T\wedge \tau}\|\mathbf{c}_i\|_{L^2}^4+\|u\|_{\mathbb{L}^2}^4 dt.
\end{align}
We could estimate the last term in \eqref{e5.4} analogue to \eqref{e2.58}.
Considering \eqref{e2.54}-\eqref{e2.58} and the bound \eqref{e3.21}, we deduce that \eqref{e2.53} holds after using the Gronwall lemma.
\end{proof}

Now, inspired by \cite{Glatt}, we show that \eqref{e2.52} holds with the help of bound \eqref{e2.53}. Since the stopping time $\tau$ is accessible, there exists a sequence of increasing stopping time $\varrho_R$ such that
$$\{\tau<\infty\}=\cap_{R=1}^\infty\{\varrho_R<\infty\}=\cup_{T=1}^{\infty}\cap_{R=1}^\infty\{\varrho_R<T\}.$$
Define the stopping time
$$\tau_M:=\inf\left\{t\geq 0; \sup_{s\in [0, t\wedge \tau]}\|\nabla\Psi\|_{L^6}\geq M\right\},$$
if the set is empty, let $\tau_M=\infty$. Since $\Psi\in L^p(\Omega; L^\infty(0,T; H^2))$, we obtain $\mathbb{P}$-a.s.
$$\lim_{M\rightarrow\infty}\tau_M=\infty.$$
Define another stopping time
\begin{align*}
\tau_K:=\inf\left\{t\geq 0; \sup_{s\in [0, t\wedge \tau\wedge \tau_M]}\|\Delta \Psi\|_{L^3}+\int_{0}^{t\wedge \tau}\|\nabla u\|^2_{\mathbb{L}^2}\|u\|_{\mathbb{L}^2}^2ds\geq K\right\}\wedge T.
\end{align*}
Then, we have
\begin{align}\label{e2.61}
\mathbb{P}\{\varrho_R<T\}&=\mathbb{P}\{(\varrho_R<T)\cap (\tau_K> T)\}+\mathbb{P}\{(\varrho_R<T)\cap (\tau_K\leq T)\}\nonumber\\
&\leq \mathbb{P}\{(\varrho_R<T)\cap (\tau_K> T)\}+\mathbb{P}\{\tau_K\leq T\}.
\end{align}
From the definition of $\tau_K$ and the Chebyshev inequality, we know
\begin{align}\label{e2.62}
\mathbb{P}\{\tau_K\leq T\}&=\mathbb{P}\left\{\sup_{t\in [0, T\wedge \tau\wedge \tau_M]}\|\Delta\Psi\|_{L^3}+\int_{0}^{T\wedge \tau}\|\nabla u\|^2_{\mathbb{L}^2}\|u\|_{\mathbb{L}^2}^2dt\geq K\right\}\nonumber\\
&\leq \frac{1}{K}\mathbb{E}\left(\sup_{t\in [0, T\wedge \tau\wedge \tau_M]}\|\Delta\Psi\|_{L^3}+\int_{0}^{T\wedge \tau}\|\nabla u\|^2_{\mathbb{L}^2}\|u\|_{\mathbb{L}^2}^2dt\right)\\
&\leq\frac{C}{K}\mathbb{E}\left(\sup_{t\in [0, T\wedge \tau\wedge \tau_M]}\sum_{i=1}^m\|\mathbf{c}^i\|_{L^3}+\int_{0}^{T\wedge \tau}\|\nabla u\|^2_{\mathbb{L}^2}\|u\|_{\mathbb{L}^2}^2dt\right).\nonumber
\end{align}
From Lemma \ref{lem2.4}, we deduce
\begin{align}
\mathbb{E}\left(\int_{0}^{T\wedge \tau}\|\nabla u\|^2_{\mathbb{L}^2}\|u\|_{\mathbb{L}^2}^2dt\right)\leq C.
\end{align}
Moreover, we have
\begin{align}\label{e2.64}
\mathbb{E}\left(\sup_{t\in [0, T\wedge \tau\wedge \tau_M]}\sum_{i=1}^m\|\mathbf{c}^i\|_{L^3}\right)\leq C(M).
\end{align}
The estimate \eqref{e2.64} is totally same with that of Lemma \ref{lem2.1} and \eqref{e4.1}. Indeed, according to \eqref{e2.5} in Lemma \ref{lem2.1}, the bound will follow once we control the $\|\nabla\Psi\|_{L^6}$, here the stopping time $\tau_M$ plays an important role. Considering \eqref{e2.62}-\eqref{e2.64}, we have
\begin{align}\label{e2.65}
\lim_{K\rightarrow\infty}\mathbb{P}\{\tau_K\leq T\}=0.
\end{align}

From the definition of $\varrho_R$, we have
\begin{align}\label{e2.66}
\mathbb{P}\{(\varrho_R<T)\cap (\tau_K> T)\}&=\mathbb{P}\left\{\left(\sup_{t\in [0, T\wedge \tau]}\|(\nabla u, \nabla \mathbf{c}^i)\|^2_{L^2}>R\right)\cap(\tau_K> T)\right\}\nonumber\\
&\leq \mathbb{P}\left\{\sup_{t\in [0, \tau_K\wedge \tau]}\|(\nabla u, \nabla \mathbf{c}^i)\|^2_{L^2}>R\right\}\\
&\leq \frac{1}{R}\mathbb{E}\left(\sup_{t\in [0, \tau_K\wedge \tau]}\|(\nabla u, \nabla \mathbf{c}^i)\|^2_{L^2}\right).\nonumber
\end{align}
Therefore, we need that the bound $\mathbb{E}\left(\sup_{t\in [0, \tau_K\wedge \tau]}\|(\nabla u, \nabla \mathbf{c}^i)\|^2_{L^2}\right)$ is uniform for $R$. The estimate could be implemented by a same argument as Lemmas \ref{lem2.1}, \ref{lem2.2}. As mentioned in Remark \ref{rem2.2}, we use the stopping time $\tau_K$ to control the extra terms $\|\nabla u\|^2_{\mathbb{L}^2}\|u\|_{\mathbb{L}^2}^2$ in \eqref{e2.22} and \eqref{e4.11}, $\|\Delta\Psi\|_{L^3}$ in \eqref{e4.19}, therefore, there exists a constant $C$ depending on $K$ but independent of $R$ such that
\begin{align}\label{e2.67}
\mathbb{E}\left(\sup_{t\in [0, \tau_K\wedge \tau]}\|(\nabla u, \nabla \mathbf{c}^i)\|^2_{L^2}\right)\leq C(K).
\end{align}
Hence, from \eqref{e2.66}, \eqref{e2.67}, we find for fixed $K$
\begin{align}\label{e2.68}
\lim_{R\rightarrow\infty}\mathbb{P}\{(\varrho_R<T)\cap (\tau_K> T)\}=0.
\end{align}
Since $\varrho_R$ is increasing, using the continuity of finite measure,  we have
\begin{align*}
\mathbb{P}\left\{\cap_{R=1}^\infty\{\varrho_R<T\}\right\}=\lim_{N\rightarrow\infty}\mathbb{P}\left\{\cap_{R=1}^N\{\varrho_R<T\}\right\}\leq \lim_{N\rightarrow\infty}\mathbb{P}\left\{\varrho_N<T\right\}.
\end{align*}
Replacing the parameter $R$ by $N$ in \eqref{e2.61}, \eqref{e2.65}, \eqref{e2.68},  and taking $N\rightarrow \infty$ for fixed $K$, yield that
$$\lim_{N\rightarrow\infty}\mathbb{P}\left\{\varrho_N<T\right\}=\mathbb{P}\{\tau_K\leq T\}.$$
Then, passing $K\rightarrow\infty$, we establish estimate \eqref{e2.52}. Then, the maximal strong pathwise solution in two dimensional case is thus global one.

\section*{Acknowledgments}
H. Wang is supported by the National Natural Science Foundation of China (Grant No.11901066), the Natural Science Foundation of Chongqing (Grant No.cstc2019jcyj-msxmX0167) and projects No.2019CDXYST0015 and No.2020CDJQY-A040 supported by the Fundamental Research Funds for the Central Universities.

\end{document}